%% file: main.tex
\documentclass[12pt]{article}
\usepackage{mystyle}
\newcommand*{\TitleFont}{%
      \usefont{\encodingdefault}{\rmdefault}{b}{n}%
      \fontsize{16}{20}%
      \selectfont}
\title{\TitleFont Global well-posedness of the two-dimensional stochastic complex Ginzburg-Landau equation
with cubic nonlinearity}
\author{Toyomu Matsuda}
\affil{{\small Graduate School of Mathematics, Kyushu University \\
744 Motooka, Nishi-ku, Fukuoka 819-0395, Japan \\
email address: ma218004@math.kyushu-u.ac.jp}}
\date{}

\begin{document}
\maketitle
\begin{abstract}
  The aim of this paper is to prove, under minimum assumptions,
  the global well-posedness of the two-dimensional stochastic complex Ginzburg-Landau
  equation on the torus driven by the additive space-time white noise.
  In addition to the global well-posedness, we prove an estimate of the solution which is uniform
  with respect to the initial condition and the strong Feller property of the dynamics.
  \vspace{0.5em}
\end{abstract}

\tableofcontents
\subfile{introduction}

\subfile{stochastic_heat_equations}

\subfile{solution_on_torus}

\subfile{strong_Feller_property}
\appendix
\subfile{appendix_besov}

\subfile{appendix_complex_ito_integral}

\section*{Acknowledgement}
The author would like to thank his supervisor Prof. Yuzuru Inahama for bringing this problem to the author's attention and helping the author
throughout this work. The author also would like to thank Dr. Masato Hoshino for his valuable comments.

\printbibliography
\end{document}

%% file: introduction.tex
\section{Introduction}
In this paper, we consider the two-dimensional stochastic complex Ginzburg-Landau (SCGL) equation:
\begin{equation}
  \label{eq:CGL}
  \left \{
  \begin{aligned}
  &\partial_t u = (i+\mu) \Delta u - \nu \abs{u}^2 u + \lambda u + \xi \quad \mbox{in } (0, \infty) \times \torus, \\
  &u(0, \cdot) = u_0,
  \end{aligned} \right.
\end{equation}
where $\mu > 0$, $\nu \in \set{z \in \C \given \Re z > 0}$ and $\lambda \in \C$.
We denote by $\torus$ the two-dimensional torus (see Subsection \ref{subsection:notation_for_Besov}).
The random field $\xi$ is the complex space-time white noise, i.e., the centered Gaussian random field
whose covariance structure is formally given by
\begin{equation}
  \label{eq:covariance_of_white_noise}
  \expect[ \xi(t, x) \xi(s, y)] = 0, \qquad \expect [\xi(t, x) \conj{\xi(s, y)}] = \delta(t-s) \delta(x-y).
\end{equation}

The SCGL is an important example of stochastic partial differential equations (SPDEs) as there are many papers on it.
We are interested in the SCGL driven by the additive space-time white noise.
The SCGL was studied in \cite{Hai02} for the spatial dimension $d=1$ and in \cite{HIN17} and \cite{Hos18} for $d=3$.
The SCGL for $d=2$ was recently studied in \cite{Tre19}. We will discuss some differences between \cite{Tre19} and our work later.

We explain how the dimension effects the difficulty of solving the SCGL.
For $d=1$, the solution of the SCGL takes values in a function space and hence the SCGL can be solved in the framework of standard
SPDEs \cite{DPZ14}.
In contrast, for $d \geq 2$, the solution of the SCGL is not a function but a Schwartz distribution. Indeed,
since the regularity of the space-time white noise $\xi$ is
\begin{equation*}
  2 \cdot (-\frac{1}{2}) + d \cdot(-\frac{1}{2}) - \epsilon = -\frac{d+2}{2} - \epsilon,
\end{equation*}
the regularity of the solution $u$ of the SCGL should be $2 - \frac{d+2}{2} -\epsilon < 0$.
As a result, the cubic nonlinearity $\abs{u}^2 u$ in \eqref{eq:CGL} does not make sense.

However, we can give a natural meaning to the solution of the SCGL for $d = 2, 3$ by way of renormalization.
Using the theory of regularity structures \cite{Hai14} and the theory of paracontrolled distributions \cite{GIP15},
the work \cite{HIN17} proved the local well-posedness of the SCGL for $d=3$. Later, the work \cite{Hos18} proved the
global well-posedness of the SCGL for $d=3$ under the assumption $\mu > \frac{1}{2\sqrt{2}}$, by adopting the technique of
\cite{MW3dim}.

The SCGL for $d=2$ should be solved more easily than the SCGL for $d=3$, as we need neither the theory of regularity structures nor
the theory of paracontrolled distributions. In fact, Da Prato-Debussche method, introduced in \cite{DD03}, is sufficient.
Let $Z$ be the solution of the linear SPDE
\begin{equation*}
  \partial_t Z = [(i+\mu) \Delta - 1] Z + \xi.
\end{equation*}
We call $Z$ and its variants Ornstein-Uhlenbeck processes because of their obvious similarity to
the stochastic differential equations (SDEs) of the same name.
We decompose the solution $u$ of the SCGL \eqref{eq:CGL} by $u = Y + Z$.
Then, $Y$ formally solves
\begin{multline*}
  \partial_t Y = [(i+\mu)\Delta - 1] Y -
  (1+\lambda)(Z + Y) \\ - \nu(\abs{Y}^2Y + 2 Z\abs{Y}^2 + \conj{Z}Y^2 + 2\abs{Z}^2 Y + Z^2 \conj{Y} + \abs{Z}^2 Z).
\end{multline*}
Since $Z$ does not take values in a function space,
$Z^2$, $\abs{Z}^2 = Z \conj{Z}$ and $\abs{Z}^2 Z = Z^2 \conj{Z}$ do not make sense. However,
as will be explained in Section \ref{section:she}, we can naturally define renormalized Wick powers $Z^{:k, l:}$
for $k, l \in \N$. We replace $Z^k \conj{Z}^l$ by $Z^{:k, l:}$ above to obtain
the shifted equation
\begin{equation}\label{eq:shifted_equation}
  \begin{multlined}
    \partial_t Y = [(i+\mu)\Delta - 1] Y -
    (1+\lambda)(Z + Y) \\ - \nu(\abs{Y}^2Y + 2 Z\abs{Y}^2 + \conj{Z}Y^2 + 2Z^{:1,1:} Y + Z^{:2,0:} \conj{Y} + Z^{:2,1:}).
  \end{multlined}
\end{equation}
This partial differential equation can be rigorously solved in mild formulation.
The main objective of this paper is to prove the global well-posedness of the equation \eqref{eq:shifted_equation},
which will be proved in Section \ref{sec:solving_CGL}.

In Section \ref{sec:strong_Feller}, we will show that the solution of \eqref{eq:CGL} defines a Markov process on some Besov space and,
furthermore, we show that this Markov process is strong Feller, in the spirit of \cite[Section 5]{TW18}.

While the author was preparing this paper, the preprint \cite{Tre19} appeared which also addresses
the global well-posedness of the two-dimensional SCGL on the torus. However, there are two major differences.
One is that \cite{Tre19} considers more general nonlinearity $\abs{u}^{2m-2} u$ ($m = 1, 2, 3, \ldots $) than
cubic nonlinearity in \eqref{eq:CGL}.
The other is conditions on parameters for the global well-posedness.
In \cite{Tre19}, the global well-posedness for \eqref{eq:CGL} is proved when $\mu > \frac{1}{2\sqrt{2}}$ or $\mu = \frac{\Re \nu}{\Im \nu}$.
In this paper, such conditions will not be imposed.

\subsection{Statement of the main theorem}
Let $\contisp^{\alpha} \defby \B^{\alpha}_{\infty, \infty}$ be the Besov-H\"older space of regularity $\alpha$ on $\torus$.
See Subsection \ref{subsection:notation_for_Besov} for its definition. We suppose
$\underline{Z} = (Z, Z^{(2, 0)}, Z^{(1, 1)}, Z^{(2, 1)})$ satisfies
\begin{equation*}
   \underline{Z} \in C((0, \infty); \contisp^{-\alpha})^4
\end{equation*}
and
\begin{equation*}
  \sup_{0 < t \leq T} \norm{Z(t)}_{\contisp^{-\alpha}} + \sup_{\substack{0 < t \leq T,
  (k,l) = (2, 0), (1, 1), (2, 1)}} t^{\alpha} \norm{Z^{(k,l)}(t)}_{\contisp^{-\alpha}} < \infty
\end{equation*}
for every $\alpha \in (0, \infty)$ and $T \in (0, \infty)$.
We set
\begin{multline*}
    \Psi(Y, \underline{Z}) =
    (1+\lambda)(Z + Y) \\ - \nu(\abs{Y}^2Y + 2 Z\abs{Y}^2 + \conj{Z}Y^2 + 2Z^{(1,1)} Y + Z^{(2,0)} \conj{Y} + Z^{(2,1)})
\end{multline*}
and $A \defby (i+\mu) \Delta - 1$.
Now we can state the main theorem of this paper (See Theorem \ref{thm:global_wellposedness} as well);
\begin{theorem}\label{thm:main_thm_of_this_paper}
  Suppose $\mu, \Re(\nu) \in (0, \infty)$ and $\lambda \in \C$.
  Let $T \in (0, \infty)$, $\alpha_0 \in (0, \frac{2}{3})$ and $Y_0 \in \contisp^{-\alpha_0}$.
  Let $\underline{Z}$ be as above.
  Then, there exists exactly one solution $Y$ satisfying the following conditions:
  \begin{enumerate}[(i)]
    \item We have $Y \in C((0, T]; \contisp^{\alpha_1})$ and
    $\sup_{0 < t \leq T} t^{\gamma} \norm{Y_t}_{\contisp^{\alpha_1}} < \infty$
    for some $\alpha_1 \in (0, \infty)$ and $\gamma \in (0, \frac{1}{3})$.
    \item For every $t \in (0, T]$, we have
    \begin{equation*}
      Y_t = e^{tA} Y_0 + \int_0^t e^{(t-s)A} \Psi(Y_s, \underline{Z}_s) ds.
    \end{equation*}
  \end{enumerate}
\end{theorem}
By substituting the corresponding Wick powers of the Ornstein-Uhlenbeck process to $\underline{Z}$,
Theorem \ref{thm:main_thm_of_this_paper} implies
the global well-posedness of \eqref{eq:shifted_equation}.

We now discuss the difficulty in the proof of Theorem \ref{thm:main_thm_of_this_paper}.
The SCGL can be viewed as a complex analogue of the dynamic $\Phi^4_d$ equation:
\begin{equation}\label{eq:phi_4}
  \partial_t \Phi = \Delta \Phi - \Phi^3 + \xi_{\R},
\end{equation}
where $\xi_{\R}$ is the real space-time white noise on $\R \times \mathbbm{T}^d$.
The global well-posedness of the dynamic $\Phi^4_d$ equation was studied for $d=2$ in \cite{DD03}, \cite{MW2dim} and
\cite{TW18} and for $d=3$ in \cite{MW3dim}.
Although our strategy is the same as the ones from \cite{MW2dim} and \cite{TW18} in spirit, crucial difficulties arise due to
the presence of the dispersion $i \Delta$. In fact, the argument in \cite{MW2dim} and \cite{TW18} only leads to a priori
$L^p$ estimate for $p \in [2, 2(1+\mu^2 + \mu \sqrt{1+\mu^2}))$.
As will be explained in the beginning of Subsection \ref{subsec:a_priori_L_p}, this a priori $L^p$ estimate is
insufficient to prove the global well-posedness of the SCGL \eqref{eq:CGL} when $\mu \leq \frac{1}{2\sqrt{2}}$.

Such difficulty has been already observed in the work \cite{Hos18}.
This is why \cite{Hos18} has to assume $\mu > \frac{1}{2\sqrt{2}}$ to prove the global well-posedness
of the three-dimensional SCGL.
We stress, however, that the assumption of $\mu$ is natural in view of
\cite[Theorem 4.1]{DGL94}.

In contrast, \cite[Theorem 4.1]{DGL94} tells that, if we do not add a noise,
the global well-posedness of the two-dimensional complex Ginzburg-Landau equation with cubic nonlinearity holds only assuming
$\mu, \Re(\nu) \in (0, \infty)$. Therefore, it is natural to expect the global well-posedness of the two-dimensional SCGL
\eqref{eq:CGL} holds
even when $\mu \leq \frac{1}{2\sqrt{2}}$.
The main achievement of this paper is hence to improve a priori $L^p$ estimate by taking advantage of smoothing effect
of the semigroup $e^{t A}$, leading to the proof
of Theorem \ref{thm:main_thm_of_this_paper}.
The approach can be viewed as bootstrap arguments \cite[Section 1.3]{Tao06}.
\subsection{Results on the solutions of the SCGL}
The statement of Theorem \ref{thm:main_thm_of_this_paper} and its proof is deterministic.
However, the solution of the SCGL \eqref{eq:CGL} (see Definition \ref{def:def_of_CGL}) is a stochastic object and
studying probabilistic aspects of the solution is an important topic.
We derive three results in this direction.

The first result is on the convergence of smoothly approximated solutions of \eqref{eq:CGL}.
Let $\xi$ be the space-time white noise on $\R \times \torus$ (see Definition \ref{def:definition_of_white_noise_and_filtration}).
Let $\rho$ be a smooth function on $\R \times \R^2$ with rapid decay at infinity.
We set $\rho^{(\delta)} \defby \delta^{-4} \rho(\delta^{-2} t, \delta^{-1} x)$
and  $\xi^{(\delta)} \defby \rho^{(\delta)} * \xi$.
\begin{theorem}\label{thm:convergence_of_smooth_cgl}
  Let $\alpha_0 \in (0, \frac{2}{3})$, $u_0 \in \contisp^{-\alpha_0}$ and $u$ be the solution of \eqref{eq:CGL} with the
  initial condition $u_0$ (see Definition \ref{def:def_of_CGL}). Let $u^{(\delta)}$ be the mild solution of
  \begin{equation*}
    \left \{
    \begin{aligned}
      &\partial_t u^{(\delta)} = (i+\mu) \Delta u^{(\delta)} - \nu ( \abs{u^{(\delta)}}^2 u^{(\delta)}
      - 2c_{\delta} u^{(\delta)}) + \lambda u^{(\delta)} + \xi^{(\delta)} \,\, \mbox{in } (0, \infty) \times \torus, \\
      & u^{(\delta)}(0) = u_0, \\
    \end{aligned} \right.
  \end{equation*}
  where $c_{\delta} = c_{1;\delta}$ is the renormalized constant given in \eqref{eq:def_of_renormalized_constant},
  which diverges logarithmically to infinity as $\delta$ tends to $0$.
  Then, we have
  \begin{equation*}
    \lim_{\delta \to 0} \expect[ \sup_{0 \leq t \leq T} \norm{u(t) - u^{(\delta)}(t)}_{\contisp^{-\alpha_0}}^p ] = 0
  \end{equation*}
  for every $p \in (0, \infty)$ and $T \in (0, \infty)$.
\end{theorem}
Theorem \ref{thm:convergence_of_smooth_cgl} is contained in Corollary \ref{cor:convergence_of_smoothed_solution}
and the proof is given there.

The second result (Theorem \ref{thm:coming_down_from_infinity}) is on a strong bound on the solutions of \eqref{eq:CGL}.
\begin{theorem}\label{thm:coming_down_from_infinity}
  Let $\alpha_0 \in (0, \frac{2}{3})$ and $u(\cdot; u_0)$ be the solution of \eqref{eq:CGL}
  with the initial condition $u_0 \in \contisp^{-\alpha_0}$ (see
  Definition \ref{def:def_of_CGL}). Then, we have
  \begin{equation}\label{eq:coming_down_from_infinity}
    \sup_{u_0 \in \contisp^{-\alpha_0}} \sup_{t \geq t_0} \expect[ \norm{u(t; u_0)}_{\contisp^{-\alpha}}^p ] < \infty
  \end{equation}
  for every $\alpha, \alpha_0 \in (0, \frac{2}{3})$, $t_0 \in (0, \infty)$ and $p \in (0, \infty)$.
\end{theorem}
A surprising part of \eqref{eq:coming_down_from_infinity} is that there is no restriction on $u_0 \in \contisp^{-\alpha_0}$.
Such bound is obtained in \cite{MW3dim}, where they call the bound ``coming down from infinity".
The proof of Theorem \ref{thm:coming_down_from_infinity} is given at the end of Subsection \ref{subsec:continuity_and_coming_down}.

The third result is to prove the strong Feller property of the Markov process defined by \eqref{eq:CGL}.
\begin{theorem}\label{thm:strong_Feller}
  Let $\alpha_0 \in (0, \frac{2}{3})$ and $u(\cdot; u_0)$ be the solution of the SCGL
  \eqref{eq:CGL} with the initial condition $u_0 \in \contisp^{-\alpha_0}$
  (see Definition \ref{def:def_of_CGL}).
  Then,
  \begin{equation*}
    P_t \Phi (v) \defby \expect[ \Phi(u(t; v))], \quad v \in \contisp^{-\alpha_0}, \, t \geq 0
  \end{equation*}
  defines a strong Feller semigroup on $\contisp^{-\alpha_0}$.
  That is, $(P_t)_{t \geq 0}$ is a Markov semigroup on $\contisp^{-\alpha_0}$ and
  the map
  \begin{equation*}
    \contisp^{-\alpha_0} \ni v \mapsto \expect[ \Phi(u(t; v)) ] \in \R
  \end{equation*}
  is continuous for every $\Phi \in L^{\infty}(\contisp^{-\alpha_0})$ and $t > 0$.
\end{theorem}
This will be proved in Section \ref{sec:strong_Feller}, see Theorem \ref{thm:continuity_in_total_variation} and afterwards specifically.
Section \ref{sec:strong_Feller} follows \cite[Section 5]{TW18}.

\subsection{Outline}
In Section \ref{section:she}, we summarize basic properties of the Ornstein-Uhlenbeck (OU) processes.
Only in this section, since there will be no big technical differences, we consider the OU process in the plane as well.
In Subsection \ref{subsec:wick_powers}, we construct renormalized Wick powers of the stationary OU processes.
In Subsection \ref{subsec:convergence_of_mollified_processes}, we consider smooth approximations of the OU processes.
In Subsection \ref{subsection:nonstationary_OU}, we introduce nonstationary OU processes and study their
basic properties.

Section \ref{sec:solving_CGL} is the main part of this paper. After reviewing the main theorem of this paper
in Subsection \ref{subsec:review_of_main_thm}, we prove the local well-posedness of the shifted equation
in Subsection \ref{subsec:local_wellposedness}.
Then, we move to a priori $L^p$ estimate in Subsection \ref{subsec:a_priori_L_p}.
In Subsection \ref{subsec:bootstrap}, we improve bounds obtained in the previous subsection by bootstrap arguments.
In Subsection \ref{subsec:continuity_and_coming_down}, we prove smooth approximations of the SCGL and the
strong bound called coming down from infinity.

The aim of Section \ref{sec:strong_Feller} is to prove the strong Feller property of the Markov process defined by
the SCGL \eqref{eq:CGL}. This section follows \cite[Section 5]{TW18}.
In Subsection \ref{subsec:Markov_property}, we show that the solutions of \eqref{eq:CGL} defines
a Markov process on some Besov space. In Subsection \ref{subsec:approximation_by_SDEs},
we show that the SPDE \eqref{eq:CGL} can be approximated by a system of SDEs.
In Subsection \ref{subsec:bismut_elworthy_li_formula}, we prove the Bismut-Elworthy-Li formula in our context.
Finally in Subsection \ref{subsec:holder_continuity}, we prove the main theorem of this section,
which readily implies the strong Feller property.

In Appendix \ref{sec:besov}, we collect results of Besov spaces which are used in this paper.
In Appendix \ref{sec:complex_ito_integrals}, we recall basic properties of complex multiple It\^o-Wiener integrals.

\subsection{Notations}\label{subsection:notation_for_Besov}
\begin{enumerate}[(i)]
  \item $\N \defby \{0, 1, 2, \ldots\}$.
  \item $\log_+ x \defby \log( \max\{x, 1\})$.
  \item $A \defby (i+\mu) \Delta - 1$.
  \item We set $e_m(x) \defby e^{2\pi i m \cdot x}$ for $m \in \Z^2$ and $x \in \R^2$.
  \item We denote by $\torus_M \defby [-\frac{M}{2}, \frac{M}{2}]^2$ the two dimensional torus, with
  the convention $\torus_{\infty} \defby \R^2$. We set $\torus \defby \torus_1$.
  \item We denote by $B(x, r)$ the Euclidean ball in $\R^2$ centered at $x$ with radius $r$.
  \item For complex-valued functions $\phi, \, \psi$, we set
  \begin{equation*}
    \inp{\phi}{\psi}_M \defby \int_{\torus_M} \phi(x) \psi(x) \, dx
  \end{equation*}
  and $\inp{\phi}{\psi} \defby \inp{\phi}{\psi}_{1}$.
  Note that we do not take complex conjugate for the second variable.
  We write $\phi * \psi$ for the convolution of $\phi$ and $\psi$.
  \item
  We write $C^{\infty}(\torus_M)$ for the space of $M$-periodic smooth functions and $C^{\infty}_c$ for
  the space of compactly supported smooth functions on $\R^2$.
  \item Let $N \in \N \cap [1, \infty)$. We denote by $\S = \S(\R^N)$ the Schwartz space of smooth functions on $\R^N$ with rapid decay at infinity,
  and by $\S'$ its dual space of Schwartz distributions.
  We also use the notation $\inp{f}{\phi}$ for the pairing of $f\in \S'$ and $\phi \in \S$.
  \item
  We write $\F f$ for the Fourier transform of $f \in \S'$, where we have
  \begin{equation*}
    \F f(\xi) = \int_{\R^2} f(x) e^{-2 \pi i x \cdot \xi} \, dx
  \end{equation*}
  if the function $f$ is integrable. We write $\F^{-1} f(\xi) \defby [\F f] (-\xi)$ for the inverse Fourier transform.
  \item We denote by $\supp(f)$ the support of a distribution $f$.
  \item We write $X \lesssim Y$ if there exists a constant $C \in (0, \infty)$ such that $X \leq C Y$.
  When we emphasize that $C$ depends on parameters $a, b, \ldots$, we write $X \lesssim_{a, b, \ldots} Y$.
\end{enumerate}

Finally, we summarize notations related to Besov spaces.
Basic properties of Besov spaces are summarized in Appendix \ref{sec:besov}.

We fix smooth, radial functions
$\chi_{-1}$, $\chi :\R^2 \to [0, 1]$ which satisfy
\begin{align*}
  &\supp(\chi_{-1}) \subset B(0, 4/3) \\
  &\supp(\chi) \subset B(0, 8/3) \setminus B(0, 3/4) \\
  &\chi_{-1} + \sum_{k = 0}^{\infty} \chi(\cdot/2^k) = 1.
\end{align*}
We set $\chi_{k} \defby \chi(\cdot/2^k)$, $\eta_k \defby \F^{-1} \chi_k$
and $\eta \defby \eta_0$.
For a distribution $f$, we write $\delta_k f \defby \eta_k * f$.
\begin{definition}\label{def:definition_of_Besov}
  Let $f$ be a distribution on $\R^2$ such that $\inp{f}{\phi(\cdot + m)} = \inp{f}{\phi}$ for every
  $\phi \in \S$ and $m \in \Z^2$.
  We define
  \begin{equation*}
    \norm{f}_{\B^{\alpha}_{p,q}} \defby \norm{ (2^{\alpha k} \norm{\delta_k f}_{L^p(\torus)})_{k \geq -1} }_{l^q} \,,
  \end{equation*}
  where $p, q \in [1, \infty]$, $\alpha \in \R$.
  The Besov space $\B^{\alpha}_{p, q}$ is the completion of $C^{\infty}(\torus)$ with respect to the norm
  $\norm{\cdot}_{\B^{\alpha}_{p, q}}$.
\end{definition}
We note that, when $p=\infty$ or $q=\infty$, this definition is different from the usual one, where the Besov space
is a collection of all distributions for which the corresponding Besov norm is finite.
An advantage of our definition is that all Besov spaces are separable. This condition is necessary, for instance,
to apply Kolmogorov continuity theorem.

\begin{remark}
  We can similarly define Besov spaces $\B_{p,q}^{\alpha, M}$ on the torus $\torus_M$.
  We only consider these spaces for $M \neq 1$ in Section \ref{section:she}.
\end{remark}

%% file: stochastic_heat_equations.tex
\section{Renormalization of Ornstein-Uhlenbeck processes}\label{section:she}
In this section, we summarize basic properties of Ornstein-Uhlenbeck process
\begin{equation}\label{eq:OU_process}
  \partial_t Z = A Z + \xi, \qquad \mbox{where } A \defby (i+\mu) \Delta - 1.
\end{equation}
Here $\xi = \xi(t, x)$ is the complex space-time white noise. Formally, $\{ \xi(t, x) \}_{t, x}$ is
a family of the centered Gaussian random variables whose covariance structure is given by \eqref{eq:covariance_of_white_noise}.
More precisely, we have the following definition.
\begin{definition}\label{def:definition_of_white_noise_and_filtration}\leavevmode
  \begin{enumerate}[(i)]
    \item
      The complex space-time white noise $\xi$ on $\R \times \R^2$ is a family of centered complex Gaussian random variables
      $\set{\xi(\phi) \given \phi \in L^2(\R \times \R^2)}$ such that
      \begin{equation*}
        \expect[\xi(\phi) \xi(\psi)] = 0, \qquad \expect[\xi(\phi) \conj{\xi(\psi)}] = \inp{\phi}{\conj{\psi}}.
      \end{equation*}
    \item Let $\mathcal{G}_t$ be the $\P$-completion of
    \begin{equation*}
      \sigma(\set{ \xi(\phi) \given \phi\vert_{(t, \infty) \times \R^2} \equiv 0, \, \phi \in L^2(\R \times \R^2)}),
    \end{equation*}
    and $\F_t \defby \cap_{t<t'} \mathcal{G}_{t'}$.
  \end{enumerate}
\end{definition}
\begin{remark}\label{rem:white_noise_as_distribution}
  It is possible to realize the white noise $\xi$ as a Besov space-valued random variable by
  applying the three dimensional version of Lemma \ref{lem:modification_of_distribution}.
\end{remark}
\begin{remark}
  The goal of this paper is to prove the global well-posedness of the two-dimensional SCGL \emph{on the torus}, not in the plane.
  However, in this section, we consider not only Ornstein-Uhlenbeck processes on the torus but also ones in the plane,
  since we can treat them in the same framework.
  See Remark \ref{rem:SCGL_in_the_plane} for the SCGL in the plane.
\end{remark}
\begin{definition}\label{def:weighted_Besov}
  We set
  \begin{equation*}
    w_{\sigma}(x) \defby (1 + \abs{x}^2)^{-\frac{\sigma}{2}}, \quad \mbox{for $\sigma \in (2, \infty)$}
  \end{equation*}
  and $\hat{L}^p_{\sigma} \defby L^p(\R^2, w_{\sigma}(x) dx)$.
  We define the norm
  \begin{equation*}
    \norm{f}_{\hat{\B}^{\alpha, \sigma}_{p, q}} \defby \norm{ (2^{\alpha k} \norm{\delta_k f}_{\hat{L}^p_{\sigma}})_{k \geq -1} }_{l^q}
  \end{equation*}
  and the weighted Besov space $\hat{\B}^{\alpha, \sigma}_{p,q}$ as the completion of $C^{\infty}_c$ with respect to the above norm.
\end{definition}
\subsection{Realization of Wick powers}\label{subsec:wick_powers}
We use the theory of complex multiple It\^o-Wiener integrals to define integration with respect to complex white noise.
See Appendix \ref{sec:complex_ito_integrals}.
We write
\begin{multline*}
  \mathcal{J}_{k,l}(f) =: \int_{(\R \times \R^2)^{k+l}} f(t_1, x_1; \ldots ;t_{k+l}, x_{k+l})
  \xi(dt_1 dx_1) \cdots \xi(dt_k dx_k) \\ \conj{\xi}(dt_{k+1} dx_{k+1}) \cdots \conj{\xi}(dt_{k+l} dx_{k+l})
\end{multline*}
for $f \in L^2((\R \times \R^2)^{k+l})$ and $k, l \in \N$.
Heuristically, Duhamel's principle suggests that a stationary solution of \eqref{eq:OU_process} is given by
\begin{align*}
  Z(t, x) &= \int_{-\infty}^t \left[ e^{(t-s)A} \xi(s, \cdot) \right] (x) ds \\
  &= \int_{-\infty}^{\infty} \int_{\R^2} K(t-s, x-y) \xi(s, y) dy ds.
\end{align*}
where
\begin{equation}
  K(t, x) \defby \frac{e^{-t}}{4 \pi (i+\mu) t} \exp \left( - \frac{\abs{x}^2}{4(i+\mu)t} \right) \indic_{\{t \geq 0\}}.
\end{equation}
Therefore, we first define $Z(t)$ as a family of random variables
$\set{ \inp{Z(t)}{\phi} \given \phi \in L^2(\R \times \R^2)}$, where
\begin{equation*}
  \inp{Z(t)}{\phi}_{\infty} \defby \int_{\R \times \R^2} \left[
  \int_{\R^2} K(t-s, z-y) \phi(z) dz \right] \xi(ds dy).
\end{equation*}
The proof of Theorem \ref{thm:main_estimate_of_OU} implies that the function
\begin{equation*}
  \R \times \R^2 \ni (s, y) \mapsto \int_{\R^2} K(t-s, z-y) \phi(z) dz
\end{equation*}
is square integrable.
Theorem \ref{thm:main_estimate_of_OU} proves that $Z$ has a distribution-valued modification.

We define $Z^{:k, l:}$, which is a Wick renormalization of $Z^k \conj{Z}^l$, by
\begin{multline*}
  \inp{Z^{:k,l:}(t)}{\phi}_{\infty} \defby  \int_{(\R \times \R^2)^{k+l}}
  \inp{ \prod_{j=1}^k K(t-s_j, \cdot - y_j) \prod_{j=k+1}^{k+l} \conj{K}(t-s_j, \cdot - y_j)}{\phi}_{\infty} \\
  \xi(ds_1  dy_1) \cdots \xi(ds_k  dy_k) \conj{\xi}(ds_{k+1} dy_{k+1}) \cdots
  \conj{\xi}(ds_{k+l} dy_{k+l}).
\end{multline*}

We introduce a spatially periodised white noise on the same probability space by
$\xi_M(\phi) \defby \xi(\phi_M)$,
where
\begin{equation*}
  \phi_M(t, x) \defby \indic_{\torus_M}(x) \sum_{y \in M \Z^2} \phi(t, x+y).
\end{equation*}
Then a stationary solution of the equation \eqref{eq:OU_process} with $\xi$ replaced by $\xi_M$ should be given
as a family of random variables
\begin{equation*}
  \inp{Z_M(t)}{\phi}_M \defby \int_{\R \times \R^2} \left[ \indic_{\torus_M}(y)
  \int_{\torus_M} K_M(t-s, z-y) \phi(z) dz \right] \xi(ds dy),
\end{equation*}
where $\phi \in L^2_M(\R^2)$ and
\begin{equation}\label{eq:def_of_kerel_K_M}
  K_M(t, x) \defby \sum_{y \in M \Z^2} \frac{e^{-t}}{4 \pi (i+\mu) t} \exp \left( - \frac{\abs{x-y}^2}{4(i+\mu)t} \right) \indic_{\{t \geq 0\}}.
\end{equation}
Similarly, Wick products are given by
\begin{multline*}
  \inp{Z^{:k, l:}_M(t)}{\phi}_M  \\ \defby \int_{(\R \times \R^2)^{k+l}}
  \inp{ \prod_{j=1}^k \indic_{\torus_M}(y_j) K_M(t-s_j, \cdot - y_j)
  \prod_{j=k+1}^{k+l} \indic_{\torus_M}(y_j) \conj{K_M}(t-s_j, \cdot - y_j)}{\phi}_M \\
  \xi(ds_1 dy_1) \cdots \xi(ds_k  dy_k) \conj{\xi}(ds_{k+1}  dy_{k+1}) \cdots
  \conj{\xi}(ds_{k+l}  dy_{k+l}).
\end{multline*}
Finally, we set
\begin{equation*}
  \inp{Z^{:k,l:}_M(t)}{\phi}_{\infty} \defby \inp{Z^{:k,l:}_M(t)}{\phi_M}_M
\end{equation*}
for $\phi \in \S$.

We also set $K_{\infty} \defby K$ and $\abs{x}_M \defby \inf \set{ \abs{x-y} \given y \in M \Z^2}$.
We regard $\infty \Z^2$ as a singleton $\{0\}$.
\begin{lemma}\label{lem:esitemate_of_K_M}
  We define
  \begin{equation}\label{eq:definition_of_K_M}
    \mathscr{K}_{M}(\delta; x) \defby \sum_{y \in M \Z^2} \int_{\abs{\delta}}^{\infty}  \frac{e^{-s}}{8 \pi ( i \delta + \mu s)}
    \exp \left( - \frac{\abs{x- y}^2}{4  (i \delta + \mu s)}\right) \, ds.
  \end{equation}
  Then we have for $M \in [1, \infty]$, $x \in \R^2$ and $\delta \in \R$,
  \begin{equation}\label{eq:estimate_of_K_M}
    \abs{\mathscr{K}_{M}( \delta; x)} \lesssim 1 + \log_{+} \abs{x}_M^{-1},
  \end{equation}
  and for $M \in [1, \infty]$, $x\in \R^2$, $\lambda \in (0, 1)$, $\tau \in \R$ and $\delta \in [-T, T]$,
  \begin{equation}\label{eq:esitimate_of_continuity_of_K_M}
    \abs{\mathscr{K}_M(\tau; x) - \mathscr{K}_{M}(\tau + \delta; x)} \lesssim_{\lambda, T}
    \frac{\abs{\delta}^{\lambda}}{\abs{x}_M^{2\lambda}}.
  \end{equation}
\end{lemma}
\begin{proof}
  Without loss of generality, we can assume $x \in \torus_M$.
  First we note that
  \begin{align*}
    \int_{\abs{\delta}}^{\infty} \Bigg\vert \frac{e^{-s}}{i \delta + \mu s} & \exp  \left( -
    \frac{\abs{x-y}^2}{4(i \delta + \mu s)} \right) \Bigg\vert \, ds \\
    &= \int_{\abs{\delta}}^{\infty} \frac{e^{-s}}{\sqrt{(\mu s)^2 + \delta^2}}
    \exp \left( - \frac{\mu s \abs{x-y}^2}{4(  \delta^2 + (\mu s)^2)} \right) \, ds \\
    &\lesssim \int_{\abs{\delta}}^{\infty} \frac{e^{-s}}{s}
    \exp \left( - \frac{\mu \abs{x-y}^2}{4 (1 + \mu^2) s} \right) \, ds.
  \end{align*}
  The term corresponding to $y = 0$ is evaluated as
  \begin{align*}
    \int_{\abs{\delta}}^{\infty} \frac{e^{-s}}{s}
    \exp \left( - \frac{\mu \abs{x}^2}{4 (1 + \mu^2) s} \right) \, ds
    & \lesssim 1 + \int_0^1 \frac{1}{s} \exp \left( -\frac{\mu \abs{x}^2}{4(\mu^2 + 1)s} \right) \, ds \\
    &= 1 + \int_{\abs{x}^2}^{\infty} \frac{1}{u} \exp \left( - \frac{\mu u}{4(\mu^2 + 1)} \right) \, du \\
    &\lesssim 1 + \log_+ \abs{x}^{-1}.
  \end{align*}
  For the terms corresponding to $y \neq 0$, we note that for large $m$,
  \begin{align*}
    \sum_{y \in M\Z^2 \setminus \{0\}} \int_{\abs{\delta}}^{\infty} \frac{e^{-s}}{s} \exp \left(
    -\frac{\mu \abs{x-y}^2}{4  (1 + \mu^2) s} \right) \, ds
    &\lesssim_m \sum_{y \in M\Z^2 \setminus \{0\}} \int_0^{\infty} \frac{e^{-s}}{s}
    \left( \frac{s}{\abs{x-y}^2} \right)^m \\
    &\lesssim_m \sum_{y \in M\Z^2 \setminus \{0\}} \frac{1}{\abs{x-y}^{2m}} \\
    &\sim_m M^{-2m}.
  \end{align*}
  Thus, we proved \eqref{eq:estimate_of_K_M}.

  We now prove \eqref{eq:esitimate_of_continuity_of_K_M}. Since $\mathscr{K}_M(-\tau;x) = \conj{\mathscr{K}_M}(\tau;x)$,
  we can assume $\tau \geq 0$. First
  suppose that $\delta > 0$. We note that
  $8 \pi (\mathscr{K}_M(\tau; x) - \mathscr{K}_M(\tau + \delta; x))$ is decomposed to
  \begin{multline}\label{eq:decomposition_of_K_M_minus_K_M_delta}
    \sum_{y \in M\Z^2} \int_{\tau}^{\tau + \delta} e^{-s} f_{\abs{x-y}^2}( i\tau + \mu s) \, ds \\
    + \sum_{y \in M \Z^2} \int_{\tau + \delta}^{\infty} e^{-s} [
    f_{\abs{x-y}^2}(i \tau + \mu s) - f_{\abs{x-y}^2}( i (\tau +\delta) + \mu s) ] \, ds ,
  \end{multline}
  where $f_a(z) \defby \frac{1}{z} \exp( - \frac{a}{4z} )$.
  We begin with the second term.
  The derivative $f_a'(z)$ equals to
  \begin{equation*}
    -\frac{1}{z^2} \exp \left( -\frac{a}{4z} \right) + \frac{a}{4 z^3} \exp \left( -\frac{a}{4z} \right).
  \end{equation*}
  For any $m \in (0, \infty)$ and $z = \mu s + it$ with $\abs{t} \leq s$, we see that
  \begin{equation*}
    \abs*{ \frac{a}{z} }^m \abs*{\exp \left( -\frac{a}{4z} \right)}
    \lesssim \left( \frac{a}{\mu s} \right)^m \exp \left( -\frac{\mu a}{(\mu^2 + 1)s} \right) \lesssim_m 1,
  \end{equation*}
  and that
  \begin{equation*}
    \abs*{\frac{1}{z^2} \exp \left( -\frac{a}{4 z} \right)} + \abs*{\frac{a}{z^3} \exp\left(- \frac{a}{4z} \right)} \lesssim_m \frac{\abs{z}^{m-2}}{a^m} .
  \end{equation*}
  Therefore, we get
  \begin{align*}
    \abs{ f_a(\mu s + i (\delta+ \tau)) - f_a(\mu s + i\tau)} &= \abs{ \int_{\tau}^{\tau+\delta} f_a'(\mu s + i t) \, dt } \\
    &\lesssim_m \frac{1}{a^m} \int_{\tau}^{\tau+\delta} \frac{dt}{[(\mu s)^2 + t^2]^{1- \frac{m}{2}}}.
  \end{align*}
  Furthermore, for $m \neq 1$,
  \begin{align*}
    \int_{\tau+\delta}^{\infty} e^{-s} \int_{\tau}^{\tau+\delta} \frac{dt}{[(\mu s)^2 + t^2]^{1- \frac{m}{2}}} \, ds
    &= \int_{\tau+\delta}^{\infty} \frac{e^{-s}}{s^{1-m}} \int_{\frac{\tau}{s}}^{\frac{\tau+\delta}{s}}
    \frac{dr}{(\mu^2 + r^2)^{1 - \frac{m}{2}}} \, ds \\
    &\lesssim \delta \int_{\tau+\delta}^{\infty} \frac{e^{-s}}{s^{2-m}} \, ds \lesssim \delta^{\min \{m, 1\}}.
  \end{align*}
  Therefore, the second term of \eqref{eq:decomposition_of_K_M_minus_K_M_delta} is bounded by
  \begin{equation*}
    \frac{\abs{\delta}^\lambda}{\abs{x}^{2\lambda}} +
    \sum_{y \neq 0} \frac{\abs{\delta}}{\abs{x-y}^{2m}} \lesssim \frac{\abs{\delta}^{\lambda}}{\abs{x}^{2\lambda}}.
  \end{equation*}
  The first term of \eqref{eq:decomposition_of_K_M_minus_K_M_delta} can be bounded similarly but more easily.
  Indeed,
  \begin{align*}
    \sum_{y \in M\Z^2} \int_{\tau}^{\tau + \delta} &\abs{f_{\abs{x-y}^2}(i\tau + \mu s)} ds \\
    &\lesssim \int_{\tau}^{\tau + \delta} \frac{  s^{\lambda-1}}{\abs{x}^{2\lambda}} ds
    + \sum_{y \in M\Z^2 \setminus \{0\}} \int_{\tau}^{\tau+\delta} \frac{s^{m-1}}{\abs{x}^{2m}}
    \lesssim \frac{\delta^{\lambda}}{\abs{x}^{2\lambda}}.
  \end{align*}
  Hence, we end the proof of \eqref{eq:esitimate_of_continuity_of_K_M} when $\delta > 0$.
  If $0 < - \delta \leq \tau$, we observe that
  \begin{equation*}
    \abs{\mathscr{K}_M(\tau + \delta; x) - \mathscr{K}_M(\tau; x)}
    = \abs{\mathscr{K}_M(\tau + \delta; x) - \mathscr{K}_M(\tau + \delta - \delta; x)}
    \lesssim \frac{\abs{\delta}^{\lambda}}{\abs{x}^{2\lambda}}.
  \end{equation*}
  Since $\mathscr{K}_M(-\delta; x) = \conj{\mathscr{K}_M}(\delta; x)$, the case $\tau = 0$ and $\delta < 0$ is proved.
  Finally, if $\tau < - \delta$,
  \begin{align*}
    \abs{\mathscr{K}_M(\tau+\delta; x) - \mathscr{K}_M(\tau; x)}
    &\leq \abs{\mathscr{K}_M(\tau+\delta; x) - \mathscr{K}_M(0; x)} +
    \abs{\mathscr{K}_M(\tau; x) - \mathscr{K}_M(0; x)} \\
    &\lesssim \frac{\abs{\tau+\delta}^{\lambda}}{\abs{x}^{2\lambda}} +\frac{\abs{\tau}^{\lambda}}{\abs{x}^{2\lambda}}
    \lesssim \frac{\abs{\delta}^{\lambda}}{\abs{x}^{2\lambda}}. \qedhere
  \end{align*}
\end{proof}
\begin{lemma}\label{lem:estimate_of_K_M_infty}
  We define
  \begin{equation}
    \mathscr{K}_{M, \infty}(x_1, x_2) \defby \int_0^{\infty}
    \int_{\torus_M} K_M(s, x_1 - y) \conj{K}(s, x_2 - y) \, dy \, ds,
  \end{equation}
  and suppose that $\abs{x_1}$, $\abs{x_2} \leq M/8$.
  Then for $M \in [1, \infty]$ and $m \geq 1$, we have
  \begin{align*}
    \abs{ \mathscr{K}_{M, \infty}(x_1, x_2) - \mathscr{K_{\infty}}(0;x_1 - x_2)} \lesssim_m \frac{1}{M^{2m}}, \\
    \abs{ \mathscr{K}_{M, \infty}(x_1, x_2) - \mathscr{K_{\infty}}(0;x_1 - x_2)} \lesssim_m \frac{1}{M^{2m}}, \\
    \abs{ \mathscr{K}_{M, \infty} (x_1, x_2) } \lesssim 1 + \log_{+} \abs{x_1 - x_2}^{-1} .
  \end{align*}
\end{lemma}
\begin{proof}
  See \cite[\largelem 12]{MW2dim}.
\end{proof}
\begin{theorem}\label{thm:main_estimate_of_OU}
  For each $M \in [1, \infty]$ and $k, l \in \N$, $Z_M^{:k, l:}$ has a distribution-valued modification.
  Moreover, for every $T \in (0, \infty)$, $\alpha \in (0, 1)$ and $p > \frac{2}{\alpha}$,
  there exists a constant $C \in (0, \infty)$ such that for all $M \in [1, \infty]$ and $\sigma \in (2, \infty)$,
  \begin{align*}
    &\expect [\sup_{0 \leq t \leq T} \norm{Z_M^{:k,l:}(t)}_{\pbesov_{p,p}^{-\alpha, \sigma}}] \leq C \int_{\R^2} w_{\sigma}(x) \, dx \\
    &\expect [\sup_{0 \leq t \leq T} \norm{Z_M^{:k,l:}(t)}_{\B_{p,p}^{-\alpha, M}}] \leq C M^2 \\
    &\expect [\sup_{0 \leq t \leq T} \norm{Z_M^{:k,l:}(t) - Z^{:k,l:}(t)}_{\pbesov_{p,p}^{-\alpha, \sigma}}] \leq C M^{2 - \sigma}.
  \end{align*}
\end{theorem}
\begin{proof}
  The strategy is to apply Lemma \ref{lem:modification_of_distribution}. Therefore we have to check that
  \eqref{eq:assumption_on_bound_of_Z_eta_k} and \eqref{eq:assumption_on_difference_of_Z_eta_k} hold.
  By Proposition \ref{prop:Nelson_estimate}, it comes down to covariance estimates.
  We set $ d\bm{s} d\bm{y} \defby ds_1 \cdots ds_{k+l} dy_1 \cdots dy_{k+l}$. We have
  \begin{align*}
    \expect [&\abs{\inp{Z_M^{:k,l:}(t)}{\phi}_M}^2 / k! l!] \\
      &\leq \int_{(\R \times \torus_M)^{k+l}}
      \abs{ \inp{\prod_{j=1}^k K_M(t-s_j, \cdot - y_j) \prod_{j=k+1}^{k+l} \conj{K_M}(t-s_j, \cdot - y_j)} {\phi}_M}^2
      \, d\bm{s} d\bm{y} \\
      &=\int_{(\R \times \torus_M)^{k+l}}
      \left[ \int_{\torus_M} \prod_{j=1}^k K_M(s_j, x_1 - y_j) \prod_{j=k+1}^{k+l} \conj{K_M}(s_j, x_1 - y_j) \phi(x_1) \, dx_1 \right]\\
      &\hspace{4em} \times\left[ \int_{\torus_M} \prod_{j=1}^k \conj{K_M}(s_j, x_2 - y_j) \prod_{j=k+1}^{k+l} K_M(s_j, x_2 - y_j) \conj{\phi}(x_2) \, dx_2 \right]
      d\bm{s} d\bm{y} \\
      &= \int_{\torus_M \times \torus_M}  dx_1 \, dx_2  \phi(x_1) \conj{\phi}(x_2)
      \left[ \int_{\R \times \torus_M}  K_M(s, x_1 - y) \conj{K_M}(s, x_1 - y) \,ds dy \right]^k \\
      &\hspace{13em} \times \left[ \int_{\R \times \torus_M}  \conj{K_M}(s, x_2 - y) K_M(s, x_2 - y) \,ds dy \right]^l.
  \end{align*}
  We note that
  \begin{multline*}
    \int_{\R \times \torus_M}  K_M(t_1 - s, x_1 - y) \conj{K_M}(t_2 - s, x_2 - y) \, ds dy \\
    = \int_{-\infty}^{\min\{t_1, t_2\}}  \sum_{z \in M \Z^2} \int_{\R^2} K(t_1 - s, x_1+z-y)
    \conj{K}(t_2 - s, x_2 - y) \, dy.
  \end{multline*}
  Chapman-Kolmogorov relation implies that for $s_1, s_2 >0$,
  \begin{multline*}
    \int_{\R^2} \frac{1}{4 \pi s_1} \exp \left(- \frac{\abs{x_1-y}^2}{4s_1} \right)
    \frac{1}{4 \pi s_2} \exp \left(- \frac{\abs{x_2-y}^2}{4s_2} \right) \, dy \\
    = \frac{1}{4 \pi (s_1 + s_2)} \exp \left( -\frac{\abs{x_1-x_2}^2}{4(s_1 + s_2)} \right).
  \end{multline*}
  Thus, by analytic continuation, we have
  \begin{multline*}
    \int_{\R^2} K(t_1-s, x_1 + z - y) \conj{K}(t_2-s, x_2 - y) \, dy\\
    = \frac{e^{-(t_1+t_2-2s)}}{4 \pi [i(t_1-t_2) + \mu(t_1+t_2 - 2s)]}
    \exp \left( -\frac{\abs{x_1+z-x_2}^2}{4[i(t_1-t_2) + \mu(t_1+t_2 - 2s)]} \right).
  \end{multline*}
  Therefore,
  \begin{equation*}
    \int_{\R\times \torus_M} K_M(t_1-s, x_1  - y) \conj{K_M}(t_2-s, x_2 - y)  ds dy
    = \mathscr{K}_M(t_1 - t_2; x_1 - x_2),
  \end{equation*}
  where $\mathscr{K}_M$ is defined by \eqref{eq:definition_of_K_M}.
  Consequently,
  \begin{equation*}
    \expect \abs{ \inp{Z_M^{:k,l:}(t)}{\phi}_M}^2 / k!l!
    \leq \int_{\torus_M \times \torus_M} dx_1 \, dx_2 \phi(x_1) \conj{\phi}(x_2) \mathscr{K}_M(0;x_1 - x_2)^{k+l}.
  \end{equation*}
  Similarly,
  \begin{align*}
    &\expect [\abs{ \inp{Z_M^{:k,l:}(t_1)}{\phi}_M - \inp{Z_M^{:k,l:}(t_2)}{\phi}_M }^2/k! l!] \\
    &\leq \int_{(\R\times\R^2)^{k+l}} \vert \langle \prod_{j=1}^k K_M(t_1 - s_j, \cdot - y_j)
    \prod_{j=k+1}^{k+l} \conj{K_M}(t_1 - s_j, \cdot - y_j) \\
    &\qquad -\prod_{j=1}^k K_M(t_2 - s_j, \cdot - y_j)
    \prod_{j=k+1}^{k+l} \conj{K_M}(t_2 - s_j, \cdot - y_j), \,\phi \rangle_M \vert^2 d\bm{s} d\bm{y}\\
    &= \int_{(\R\times\R^2)^{k+l}} \abs{\inp{\prod_{j=1}^k K_M(t_1 - s_j, \cdot - y_j) \prod_{j=k+1}^{k+l}
    \conj{K_M}(t_1 - s_j, x - y_j)}{\phi}_M}^2 d\bm{s} d\bm{y} \\
    &+ \int_{(\R\times\R^2)^{k+l}} \abs{\inp{\prod_{j=1}^k K_M(t_2 - s_j, \cdot - y_j) \prod_{j=k+1}^{k+l}
    \conj{K_M}(t_2 - s_j, x - y_j)}{\phi}_M}^2 d\bm{s} d\bm{y} \\
    &-2 \Re B,
  \end{align*}
  where
  \begin{align*}
    B
    &=\int_{(\R\times\R^2)^{k+l}} \inp{\prod_{j=1}^k K_M(t_1 - s_j, \cdot - y_j) \prod_{j=k+1}^{k+l}
    \conj{K_M}(t_1 - s_j, x - y_j)}{\phi}_M \\
    &\hspace{4em} \times \inp{\prod_{j=1}^k \conj{K_M}(t_2 - s_j, \cdot - y_j) \prod_{j=k+1}^{k+l}
    K_M(t_2 - s_j, x - y_j)}{\conj{\phi}}_M d\bm{s} d\bm{y}\\
    &= \int_{\torus_M \times \torus_M} dx_1 dx_2 \phi(x_1) \conj{\phi}(x_2) \\
    &\hspace{4em}\times \left(\int_{\R \times \torus_M} K_M(t_1-s, x_1 - y) \conj{K_M}(t_2-s, x_2-y) dy\right)^k \\
    &\hspace{4em}\times \left(\int_{\R \times \torus_M} \conj{K_M}(t_1-s, x_1 - y) K_M(t_2-s, x_2-y) dy\right)^l \\
    &= \int_{\torus_M \times \torus_M} dx_1 dx_2 \phi(x_1) \conj{\phi}(x_2) \mathscr{K}_M(t_1 - t_2; x_1 - x_2)^k
    \mathscr{K}_M(t_2 - t_1; x_1 - x_2)^l.
  \end{align*}
  As a result,
  \begin{align*}
    \expect &[\abs{ \inp{Z_M^{:k,l:}(t_1)}{\phi}_M - \inp{Z_M^{:k,l:}(t_2)}{\phi}_M }^2/k! l!] \\
    &\leq 2 \Re \int_{\torus_M \times \torus_M} dx_1 \, dx_2 \phi(x_1) \conj{\phi}(x_2) \\
    &\quad \times \left[ \mathscr{K}_M(0;x_1 - x_2)^{k+l} - \mathscr{K}_{M}(t_1 - t_2; x_1 - x_2)^k \mathscr{K}_{M}(t_2 - t_1; x_1 - x_2)^l \right].
  \end{align*}
  Finally, similar calculation yields
  \begin{align*}
    \expect &[\abs{\inp{Z_M^{:k,l:}(t)}{\phi}_{\infty} - \inp{Z^{:k,l:}(t)}{\phi}_{\infty}}^2/k!l!] \\
    &\leq  \Re \int_{\R^2 \times \R^2} dx_1 \, dx_2 \phi(x_1) \conj{\phi}(x_2) \\
    &\quad  \times \left( \mathscr{K}_M(0; x_1 - x_2)^{k+l}
    + \mathscr{K}(0;x_1 - x_2)^{k+l}
     -  2\mathscr{K}_{M,\infty}(x_1, x_2)^k \mathscr{K}_{M,\infty}(x_2, x_1)^l \right).
  \end{align*}
  Thanks to Lemma \ref{lem:esitemate_of_K_M} and \ref{lem:estimate_of_K_M_infty}, the remainder of the proof is the same
  as \cite[\largethm 5.1]{MW2dim}.
\end{proof}
\begin{remark}\label{rem:continuity_of_OU}
  Theorem \ref{thm:main_estimate_of_OU} is a consequence of Lemma \ref{lem:modification_of_distribution}.
  Therefore, the following claim is obvious:
  given $M \in [1, \infty)$, $T \in (0, \infty)$ and $\alpha \in (0, 1)$, there exists $\theta \in (0, 1)$ such that
  \begin{equation*}
    \expect[ \sup_{0 \leq s \leq t} \norm{Z_M(s) - Z_M(0)}_{\B^{-\alpha, M}_{\infty, \infty}} ]
    \lesssim_{M, T, \alpha} t^{\theta}
  \end{equation*}
  for every $t \leq T$.
\end{remark}

\subsection{Convergence of mollified processes}\label{subsec:convergence_of_mollified_processes}
Take $\rho \in \S$ with $\int \rho = 1$
and set $\rho_{\delta} \defby \delta^{-4} \rho(\frac{t}{\delta^2}, \frac{x}{\delta})$.
We consider a stationary solution of
\begin{equation*}
  \partial_t Z_{M;\delta} = A Z_{M;\delta} + \xi_{M; \delta},
\end{equation*}
where $\xi_{M;\delta} \defby \xi_M * \rho_{\delta}$ is a mollified noise.
Applying a three dimensional version of Lemma \ref{lem:modification_of_distribution}, we can realize
$\xi_M$ as a distribution in $\B^{-3/2 - \epsilon, M}_{\infty, \infty}$ if $M < \infty$ or
as a distribution in $\B^{-3/2 - \epsilon, \sigma}_{p, p}$ for all $M \in [1, \infty]$ with $p \in [1, \infty)$.
Then $Z_{M;\delta}$ is given by
\begin{equation*}
  Z_{M;\delta}(t; x)= \int_{-\infty}^t \left(e^{(t-s)A} \xi_{M;\delta}(s, \cdot)\right)(x) \, ds.
\end{equation*}
Take $\psi_n \in C_c^{\infty}(\R; [0, 1])$ such that $\psi_n \equiv 1$ on $[1/n, n]$ and
$\supp(\psi_n) \subset [1/n^2, n^2]$. Then $\psi_n K = \psi_n(t) K(t, x) \in \S$ and
\begin{align*}
  Z_{M;\delta}(t; x) &= \lim_{n \to \infty} \int_{\R \times \R^2} \psi_n(t-s) K(t-s, x-y) \xi_{M;\delta}(s, y) ds dy \\
  &= \lim_{n \to \infty} [ (\psi_n K) * (\rho_{\delta} * \xi_M) ] (t, x) \\
  &= \lim_{n \to \infty} [ ((\psi_n K)*\phi_{\delta}) * \xi_M] (t, x) \\
  &= \lim_{n \to \infty} \int_{\R \times \R^2} (\psi_n K) * \rho_{\delta}(t-s, x-y) \xi_M(ds dy).
\end{align*}
Since $\lim_{n \to \infty} (\psi_n K) * \rho_{\delta} = K * \rho_{\delta}$ in $L^2(\R \times \R^2)$, we obtain
\begin{equation*}
  Z_{M;\delta}(t; x) = \int_{\R \times \R^2} \indic_{\torus_M}(x) K_{M;\delta}(t-s, x-y) \xi(ds dy)
\end{equation*}
where $K_{M;\delta} \defby K_M * \rho_{\delta}$.
Then Wick products $Z_{M;\delta}^{:k,l:}(t;x)$ are given by
\begin{multline*}
  Z_{M;\delta}^{:k, l:}(t;x) = \int_{\R \times \torus_M} \prod_{j=1}^k K_{M;\delta}(t-s_j, x-y_j)
  \prod_{j=k+1}^{k+l} \conj{K_{M;\delta}}(t-s_j, x-y_j) \\
  \xi(ds_1 dy_1) \cdots \xi(ds_k dy_k)
  \conj{\xi}(ds_{k+1} dy_{k+1}) \cdots \conj{\xi}(ds_{k+l}  dy_{k+l}).
\end{multline*}
By Corollary \ref{cor:multiple_integral_and_hermite_polynomial}, we have
$Z_{M;\delta}^{:k,l:}(t;x) = H_{k,l}(Z_{M;\delta}(t;x), c_{M;\delta})$
where
\begin{equation} \label{eq:def_of_renormalized_constant}
  c_{M;\delta} \defby \int_{\R \times \torus_M} \abs{K_{M;\delta}(s,y)}^2 \, ds dy
\end{equation}
In particular,
\begin{align*}
  Z^{:2,0:}_{M;\delta}(t; x) &= Z_{M;\delta}^2(t; x), \\
  Z^{:1, 1:}_{M; \delta}(t; x) &= \abs{Z_{M;\delta}(t;x)}^2 - c_{M;\delta}, \\
  Z^{:2, 1:}_{M; \delta}(t; x) &= \abs{Z_{M;\delta}(t;x)}^2 Z_{M;\delta}(t,x) - 2c_{M;\delta} Z_{M;\delta}(t;x).
\end{align*}
\begin{proposition}\label{prop:renormalization_constant}
  We have, as $\delta \to 0$,
  \begin{equation*}
    c_{M;\delta} = \frac{1}{4\pi \mu} \log_+ \delta^{-1} + O(1).
  \end{equation*}
\end{proposition}
\begin{proof}
  We have
  \begin{equation*}
    c_{M;\delta} = \int_{\R \times \R^2} \rho(s_1, y_1) \rho(s_2, y_2) \mathscr{K}_M
    (\delta^2(s_2 - s_1); \delta(y_1 - y_2)) ds_1 ds_2 dy_1 dy_2.
  \end{equation*}
  By \eqref{eq:estimate_of_K_M},
  \begin{multline*}
    \int_{\max\{\abs{s_1}, \abs{s_2}\} \geq \delta^{-1}} \rho(s_1, y_1) \rho(s_2, y_2)\\ \times \mathscr{K}_M
    (\delta^2(s_2 - s_1); \delta(y_1 - y_2)) ds_1 ds_2 dy_1 dy_2  = O(1).
  \end{multline*}
  and by \eqref{eq:esitimate_of_continuity_of_K_M},
  \begin{multline*}
    \int_{\abs{s_1}, \abs{s_2} \leq \delta^{-1}} \rho(s_1, y_1) \rho(s_2, y_2) \\ \times \abs{\mathscr{K}_M
    (\delta^2(s_2 - s_1); \delta(y_1 - y_2)) - \mathscr{K}_M(0; \delta(y_1 -y_2))} ds_1 ds_2 dy_1 dy_2 = O(1).
  \end{multline*}
  Therefore, it suffices to show that
  \begin{equation*}
    \mathscr{K}_M(0; x) = \frac{1}{4\pi \mu} \log_+ \abs{x}_M^{-1} + O(1).
  \end{equation*}
  We can assume $x \in \torus_M$. As shown in the proof of \eqref{eq:estimate_of_K_M},
  \begin{equation*}
    \sum_{y \in M\Z^2 \setminus \{0\}} \int_0^{\infty} \frac{e^{-s}}{8 \pi \mu s} \exp
    \left(-\frac{\abs{x-y}^2}{4\mu s} \right) ds = O(1).
  \end{equation*}
  Then it remains to observe that
  \begin{align*}
    \int_0^{\infty}\frac{e^{-s}}{8\pi\mu s} \exp\left(-\frac{\abs{x}^2}{4\mu s} \right)
    &= \int_0^1 \frac{1}{8 \pi \mu s} \exp \left( - \frac{\abs{x}^2}{4 \mu s} \right) ds + O(1) \\
    &= \int_{\frac{\abs{x}^2}{4\mu}}^{\infty} \frac{1}{8 \pi \mu r} e^{-r} dr + O(1) \\
    &= \int_{\min\{1, \frac{\abs{x}^2}{4\mu} \}}^1 \frac{1}{8 \pi \mu r} dr + O(1) \\
    &= \frac{1}{4 \pi \mu} \log_+ \abs{x}^{-1} + O(1). \qedhere
  \end{align*}
\end{proof}
\begin{theorem} \label{thm:convergence_of_mollified_processes}
  For every $M \in [1, \infty)$, $T \in (0, \infty)$, $\alpha \in (0, 1)$ and $p \in [1, \infty)$, we have
  \begin{align*}
    \lim_{\delta \to 0} \expect[ \sup_{0 \leq t \leq T} \norm{Z_{M;\delta}^{:k,l:}(t) - Z_{M}^{:k,l:}(t)}_{\B_{p,p}^{-\alpha, M}}^p ] &= 0, \\
    \lim_{\delta \to 0}  \expect[ \sup_{0\leq t \leq T} \norm{Z_{\delta}^{:k,l:}(t) - Z^{:k,l:}(t)}_{\pbesov_{p,p}^{-\alpha, \sigma}}^p] &= 0.
  \end{align*}
\end{theorem}
\begin{remark}
  By Proposition \ref{prop:besov_embedding}, for $\alpha, p \in (0, \infty)$,
  \begin{equation*}
    \lim_{\delta \to 0} \expect[ \sup_{0 \leq t \leq T} \norm{Z_{M;\delta}^{:k,l:} - Z_M^{:k,l:}(t)}_{\B^{-\alpha,M}_{\infty, \infty}}^p ] = 0.
  \end{equation*}
\end{remark}
\begin{proof}
  \emph{STEP 1.}
  We first show that for each $\lambda, \epsilon \in (0, 1)$
  there exists a constant $C = C(\lambda, \epsilon)$, independent of $\delta \in (0, 1)$, such that
  \begin{equation}\label{eq:difference_of_Z_M_delta}
    \expect [\abs{\inp{Z_{M;\delta}^{:k,l:}(t_1) - Z_{M;\delta}^{:k,l:}(t_2)}{\eta_j(x-\cdot)}_{\infty}}^2/k!l! ]
    \leq C \abs{t_1 - t_2}^{\lambda} 2^{2 j \lambda(1+\epsilon)}.
  \end{equation}
  We have
  \begin{align*}
    \expect [\vert &\inp{Z_{M;\delta}^{:k,l:}(t_1) - Z_{M;\delta}^{:k,l:}(t_2)}{\eta_j(x-\cdot)}\vert^2/k!l!]\hspace{15em} \\
    &\hspace{4em}\leq 2 \Re \int_{\R^2 \times \R^2} dx_1 dx_2 \eta_j(x_1) \eta_j(x_2) \\
    &\hspace{5em} \times \Big[ \left( \int_{\R \times \torus_M} K_{M;\delta}(s, x_1 - y) \conj{K_{M;\delta}}(s, x_2 - y) \, ds dy\right)^{k}\\
    &\hspace{5em}\times \left( \int_{\R \times \torus_M} K_{M;\delta}(s, x_2 - y) \conj{K_{M;\delta}}(s, x_1 - y) \, ds dy\right)^{l} \\
    &\hspace{5em}- \left( \int_{\R \times \torus_M} K_{M;\delta}(t_1 - s, x_1 - y) \conj{K_{M;\delta}}(t_2 - s, x_2 - y) \, ds dy \right)^k \\
    &\hspace{5em}\times \left( \int_{\R \times \torus_M} K_{M;\delta}(t_2 - s, x_2 - y) \conj{K_{M;\delta}}(t_1 - s, x_1 - y) \, ds dy\right)^l \Big] ,
  \end{align*}
  and
  \begin{align*}
    \int_{\R \times \torus_M} &K_{M;\delta}(t_1 - s, x_1 - y) K_{M;\delta}(t_2 - s, x_2 - y) \, ds dy \\
    &= \int_{(\R \times \R^2)^2} \rho_{\delta}(s_1, y_1) \rho_{\delta}(s_2, y_2) \\
    &\qquad \times \mathscr{K}_M(t_1 - t_2 - (s_1 - s_2); x_1 - x_2 - (y_1 - y_2))  ds_1 dy_1 ds_2 dy_2.
  \end{align*}
  By Lemma \ref{lem:esitemate_of_K_M} and Jensen inequality, we see that the left hand side of \eqref{eq:difference_of_Z_M_delta}
  is bounded by
  \begin{multline}\label{eq:bound_of_difference_of_Z_M_delta_1}
    C \abs{t_1 - t_2}^{\lambda} \int dx_1 dx_2 \,\abs{\eta_j(x_1)}\abs{\eta_j(x_2)}
    \left(\int \tilde{\rho}_{\delta}(x_1 - y_1) \tilde{\rho}_{\delta}(x_2 - y_2)
    \frac{dy_1 dy_2}{\abs{y_1 - y_2}_M^{2\lambda}} \right)\\
    \times \left(\int \tilde{\rho}_{\delta}(x_1 - y_1) \tilde{\rho}_{\delta}(x_2 - y_2)
    (1 + \log_+^{k+l-1} \abs{y_1 - y_2}_M^{-1}) \, dy_1 dy_2\right)
  \end{multline}
  where $\tilde{\rho}_{\delta}(x) \defby \int \rho_{\delta}(s, x) \, ds$.
  The inequality $\log_+^{k+l-1} \abs{y_1 - y_2}_M^{-1} \lesssim_{k,l, \epsilon_1} \abs{y_1 - y_2}_M^{- \epsilon_1}$
  for every $\epsilon_1 \in (0, 1)$ and
  Jensen's inequality implies that the integral in \eqref{eq:bound_of_difference_of_Z_M_delta_1} is bounded by
  \begin{equation} \label{eq:bound_of_difference_of_Z_M_delta_2}
    \int_{\R^2 \times \R^2} \tilde{\rho}_{\delta} * \abs{\eta_j}(x_1) \tilde{\rho}_{\delta} * \abs{\eta_j}(x_2)
    \frac{dx_1 dx_2}{\abs{x_1 - x_2}_M^{\lambda_1}},
  \end{equation}
  where $\lambda_1 \defby 2 \lambda (1 + \epsilon)$.
  Therefore, to prove \eqref{eq:difference_of_Z_M_delta}, it suffices to show that
  \eqref{eq:bound_of_difference_of_Z_M_delta_2} is bounded by $C 2^{j \lambda_1}$.
  We focus on the case $j \geq 0$, since the case $j = -1$ is the same as the case $j=0$.
  We note that $\tilde{\rho}_{\delta} * \abs{\eta_j}(x) = 2^{2j} \tilde{\rho}_{2^j \delta} * \abs{\eta_0}(2^j x)$.
  We proceed to check that the integral
  \begin{align*}
    \int \abs{\eta_0(x_1)} &\abs{\eta_0(x_2)} \frac{dx_1 dx_2}{\abs{x_1 - x_2 - y}^{\lambda_1}}
    = \int \abs{\eta_0(x_1 + x_2)} \abs{\eta_0(x_2)} \frac{dx_1 dx_2}{\abs{x_1 - y}^{\lambda_1}} \\
    &\lesssim \int_{\abs{x_1 - y} \leq 1} \frac{dx_1}{\abs{x_1 - y}^{\lambda_1}}
    + \int_{\abs{x_1 - y} \geq 1} \abs{\eta_0(x_1 + x_2)} \abs{\eta_0(x_2)} \, dx_1 dx_2 \\
    &\lesssim 1
  \end{align*}
  can be bounded uniformly for all $y \in \R^2$.
  Furthermore, if $\abs{x} \geq 2 \delta \supp (\tilde{\rho})$, we have, uniformly for $\delta$,
  \begin{equation*}
    \abs{\tilde{\rho}_{\delta} * \eta_j (x)}
    \lesssim_m \int \tilde{\rho}_{\delta}(y) \frac{2^{2j}}{(2^j \abs{x - y})^m} \, dy
    \lesssim_m \frac{1}{2^{(m-2)j} \abs{x}^m}.
  \end{equation*}
  Now we see that the integral \eqref{eq:bound_of_difference_of_Z_M_delta_2} in the region
  $\set{(x_1, x_2) \given \abs{x_1}, \abs{x_2} \leq M/4}$
  is bounded by
  \begin{align*}
    \int_{\R^2 \times \R^2} \rho_{2^j \delta} *& \abs{\eta_0} (x_1) \rho_{2^j \delta} * \abs{\eta_0}(x_2)
    \frac{dx_1 dx_2}{\abs{2^{-j}(x_1 - x_2)}^{\lambda_1}} \\
    &\leq 2^{j \lambda_1} \int dy_1 dy_2 \tilde{\rho}(y_1) \tilde{\rho}(y_2)
    \int \frac{\abs{\eta_0}(x_1) \abs{\eta_0}(x_2) \, dx_1 dx_2}{\abs{x_1 - x_2 - 2^j \delta (y_1 - y_2)}^{\lambda_1}} \\
    &\lesssim 2^{j \lambda_1}.
  \end{align*}
  The integral \eqref{eq:bound_of_difference_of_Z_M_delta_2} in the region
  $\set{(x_1, x_2) \given \abs{x_1} \geq M/4 \mbox{ and } \abs{x_2} \leq M/4}$
  is bounded by
  \begin{equation*}
    C \int_{\abs{x_1} \geq M/4, \,\,  \abs{x_2} \leq M/4} \frac{1}{(1+\abs{x_1}^2)^m}
    \rho_{\delta} * \abs{\eta_j}(x_2) \frac{dx_1 dx_2}{\abs{x_1 - x_2}_M^{\lambda_1}}.
  \end{equation*}
  When $\abs{y} \leq M/4$, we have a uniform estimate
  \begin{align*}
    \int_{\R^2} \frac{1}{(1+\abs{x}^2)^m} \frac{dx}{\abs{x-y}_M^{\lambda_1}}
    &= \sum_{z \in M \Z^2} \int_{y + z + \torus_M} \frac{1}{(1+\abs{x}^2)^m} \frac{dx}{\abs{x-y-z}_M^{\lambda_1}} \\
    &\lesssim 1 + \sum_{z \in M \Z^2 \setminus \{0\}} \frac{1}{\abs{z}^{2m}} \int_{\torus_M} \frac{dx}{\abs{x}^{\lambda_1}} \\
    &\lesssim 1,
  \end{align*}
  and the $L^1$-norm of $\rho_{\delta} * \abs{\eta_j}$ is uniformly bounded with respect to $\delta$ and $j$.
  The integral \eqref{eq:bound_of_difference_of_Z_M_delta_2} in the region
  $\set{(x_1, x_2) \given \abs{x_1} \geq M/4 \mbox{ and } \abs{x_2} \geq M/4}$
  is bounded by
  \begin{equation*}
    \int_{\R^2 \times \R^2} \frac{1}{(1+\abs{x_1}^2)^m} \frac{1}{(1+\abs{x_2}^2)^m}
    \frac{dx_1 dx_2}{\abs{x_1 - x_2}_M^{\lambda_1}} \lesssim 1.
  \end{equation*}
  Therefore, we conclude that \eqref{eq:bound_of_difference_of_Z_M_delta_2} is bounded by $C 2^{j \lambda_1}$.

  \emph{STEP 2.}
  We continue to evaluate
  \begin{equation}\label{eq:difference_between_Z_and_Z_delta}
    a(j, \delta) \defby \expect [\abs{ \inp{Z_{M;\delta}^{:k,l:}(t) - Z_{M}^{:k,l:}(t)}{\eta_j(x - \cdot)}}^2] ,
  \end{equation}
  which is bounded by $k!l!$ times
  \begin{multline*}
    \Re \int_{\R^2 \times \R^2} dx_1 dx_2 \eta_j(x_1) \eta_j(x_2) \\
    \times \Big[ \left( \mathscr{K}_{M; \delta}(x_1 - x_2) \right)^{k+l}
    + \mathscr{K}_M(0; x_1 - x_2)^{k+l} \\
    - 2 \left( \int \rho_{\delta}(s_1, y_1) \mathscr{K}_M(s_1; x_1 - x_2 - y_1) \, ds_1 dy_1 \right)^{k+l} \Big],
  \end{multline*}
  where
  \begin{equation*}
    \mathscr{K}_{M;\delta}(x) \defby \int \rho_{\delta}(s_1, y_1) \rho_{\delta}(s_2, y_2)
    \mathscr{K}_M(s_2 - s_1; x_1 - x_2 - (y_1 - y_2)) \, ds_1 ds_2 dy_1 dy_2 .
  \end{equation*}
  For fixed $x_1, x_2$ with $x_1 - x_2 \notin M\Z^2$,
  \begin{equation*}
    \lim_{\delta \to 0} \mathscr{K}_{M;\delta}(x_1 - x_2) = \mathscr{K}_M(0; x_1 - x_2).
  \end{equation*}
  \emph{STEP 1} shows that for every $n$,
  \begin{equation*}
    \sup_{\delta \in (0, 1)} \int \abs{\eta_j(x_1)} \abs{\eta_j(x_2)} \abs{\mathscr{K}_{M; \delta}(x_1 - x_2)}^n \, dx_1 dx_2
    < \infty.
  \end{equation*}
  Therefore,
  \begin{multline*}
    \lim_{\delta \to 0} \int \eta_j(x_1) \eta_j(x_2) \mathscr{K}_{M; \delta}(x_1 - x_2) \, dx_1 dx_2 \\
    = \int \eta_j(x_1) \eta_j(x_2) \mathscr{K}_M(0; x_1 - x_2) \, dx_1 dx_2 .
  \end{multline*}
  Furthermore, calculation similar to the proof of \emph{STEP 1} shows that for every $\epsilon \in (0, 1)$,
  $2^{- \epsilon j} a(j, \delta)$ (see \eqref{eq:difference_between_Z_and_Z_delta}) is
  uniformly bounded with respect to $j$ and $\delta$.
  Hence, for $\nu \in (0, 1)$,
  \begin{align*}
    &\expect[ \sum_{j \geq -1} 2^{-j \alpha p} \norm{ \inp{Z_{M;\delta}^{:k,l:}(t) - Z_{M;\delta}^{:k,l:}(s)
    - (Z^{:k,l:}_M(t) - Z^{:k,l:}_M(s))}{\eta_j(x - \cdot)}_{\infty} }_{L^p(\torus_M, dx)}^p ] \\
    &= \sum_{j \geq -1} 2^{-j \alpha p} \int_{\torus_M} \expect [
    \abs{\inp{Z_{M;\delta}^{:k,l:}(t) - Z_{M;\delta}^{:k,l:}(s) - (Z^{:k,l:}_M(t) - Z^{:k,l:}_M(s))}{\eta_j(x - \cdot)}_{\infty}}^p] \, dx \\
    &\lesssim \sum_{j \geq -1} 2^{-j \alpha p} \left( \abs{t-s}^{\lambda} 2^{j \lambda_1}\right)^{\frac{\nu p}{2}}
    a(j, \delta)^{\frac{(1-\nu)p}{2}} \\
    &= \abs{t-s}^{\frac{p\nu \lambda}{2}} \sum_{j \geq -1} 2^{-j (\alpha - \frac{\nu \lambda_1}{2}) p  } a(j, \delta)^{\frac{(1-\nu)p}{2}}.
  \end{align*}
  By Lebesgue's dominated convergence theorem,
  \begin{equation*}
      \lim_{\delta \to 0} \sum_{j \geq -1} 2^{-j (\alpha - \frac{\nu \lambda_1}{2}) p } a(j, \delta)^{\frac{(1-\nu)p}{2}} = 0.
  \end{equation*}
  Therefore, Kolmogorov continuity theorem finishes the proof.
\end{proof}
\begin{remark}\label{remark:convergence_of_mollified_processes}
  As the above proof suggests, time regularization is unnecessary.
  In \cite{TW18}, they use smooth approximations by
  \begin{equation*}
    \tilde{Z}_{M;\delta}(t; x) \defby \inp{Z_M(t)}{\tilde{\rho}_{\delta}(x - \cdot)}_{\infty},
  \end{equation*}
  where $\tilde{\rho} \in \S$ and $\tilde{\rho}_{\delta}(x) \defby \delta^{-2} \tilde{\rho}(\delta^{-1} x)$.
  Although space-time regularization $Z_{M;\delta}$ seems natural in view of renormalization,
  it is not $(\F_t)$-adapted. Therefore, space regularization is crucial to study
  stochastic properties of solutions of \eqref{eq:CGL}.
\end{remark}

\subsection{Nonstationary Ornstein-Uhlenbeck processes}\label{subsection:nonstationary_OU}
We have constructed Ornstein-Uhlenbeck processes starting from $t = -\infty$.
An advantage is that the processes are stationary in time and the renormalized constants are independent of time.
However, we have to consider nonstationary Ornstein-Uhlenbeck processes later in order to study probabilistic properties
of solutions. For the sake of simplicity, we restrict our analysis to the case $\torus_M$ with $M < \infty$.

To begin with, it is convenient to introduce an algebraic structure for Wick product.
Let $\mathcal{R} \defby \B^{\alpha, M}_{\infty, \infty}$ for some $\alpha \in (0, \infty)$.
$\mathcal{R}$ is a ring with the usual product, see Corollary \ref{cor:multiplicative_inequality}.
Furthermore, Corollary \ref{cor:multiplicative_inequality} allows us to multiply an element of $\mathcal{R}$ and
$Z_M^{:k,l:}(t)$. Therefore, the space
\begin{equation*}
  \mathcal{M}_t \defby  \operatorname{span}_{\mathcal{R}} \set{Z_M^{:k,l:}(t) \given k, l \in \N}
\end{equation*}
is a commutative $\mathcal{R}$-algebra with multiplication
\begin{equation*}
   Z_M^{:k_1, l_1:}(t) \wickproduct Z_M^{:k_2, l_2:}(t)  \defby Z_M^{:k_1 + k_2, l_1 + l_2:}(t).
\end{equation*}
We also have an operation of complex conjugate
\begin{equation*}
  \conj{Z_M^{:k,l:}(t)} = Z_M^{:l, k:}(t),
\end{equation*}
where the equality holds as distribution.
We write for $X \in \mathcal{M}_t$,
\begin{equation*}
  X^{:k,l:} \defby \underbrace{X \wickproduct \cdots \wickproduct X}_{k\, \mbox{ times}} \wickproduct
  \underbrace{\conj{X} \wickproduct \cdots \wickproduct \conj{X}}_{l\, \mbox{ times}}.
\end{equation*}

Now we set $Z_M(s, t) \defby Z_M(t) - e^{(t-s)A} Z_M(s) \in \mathcal{M}_t$ for $s < t$, where we regard
$e^{(t-s)A} Z_M(s) \in \mathcal{R}$. We see that for each $\phi \in \S$, almost surely,
\begin{equation*}
  \inp{Z_M(s, t)}{\phi} = \int_{\R \times \R^2} \indic_{[s, t]}(r) \indic_{\torus_M}(y)
  \inp{K_M(t-r, \cdot - y)}{\phi} \xi(dr dy).
\end{equation*}
Indeed, this can be proved by first showing the corresponding identities for mollified processes and then taking the limit.
See the proof of Proposition \ref{prop:independence_of_Z_M_s_t} below. We also set for $s < t$
\begin{align*}
  Z^{:k,l:}_M(s, t) &\defby (Z_M(s, t))^{:k, l:} \\ &= \sum_{\substack{0\leq i \leq k, \\ 0\leq j \leq l}}
  \binom{k}{i} \binom{l}{j} (-1)^{i+j} \left( e^{(t-s)A} Z_M(s) \right)^{:i,j:} Z_M^{:k-i, l-j:}(t).
\end{align*}
We have the identity
\begin{equation}
  Z_M^{:k,l:}(s, t+h) = \sum \binom{k}{i} \binom{l}{j} (e^{hA} Z_M(s, t))^{:i,j:} Z_M(t, t+h)^{:k-i,l-j:}
\end{equation}
for $s < t$ and $h \in (0, \infty)$. In fact,
\begin{align*}
  Z_M^{:k,l:}(s, t+h) &= \left( Z_M(t+h) - e^{(t+h-s)A} Z_M(s) \right)^{:k,l:} \\
  &=\left( Z_M(t, t+h) + e^{hA} Z_M(t) - e^{hA} e^{(t-s)A} Z_M(s) \right)^{:k,l:} \\
  &= \left(Z_M(t, t+h) + e^{hA} Z_M(s, t) \right)^{:k,l:} \\
  &= \sum \binom{k}{i} \binom{l}{j} \left( e^{hA} Z_M(s, t) \right)^{:i, j:} Z_M(t, t+h)^{:k-i, l-j:}.
\end{align*}
\begin{proposition}\label{prop:L^p_estimate_of_nonstationary_OU}
  Let $p \in [1, \infty)$, $\alpha, \alpha' \in (0, 1)$, $T \in (0, \infty)$ and $s \in \R$.
  Then for each $k, l \in \N$,
  \begin{equation*}
    \expect[ \sup_{0 \leq t \leq T} t^{(k+l-1)\alpha' p} \norm{Z^{:k,l:}_M(s, s+t)}_{\B^{-\alpha, M}_{\infty, \infty}}^p ]
    \lesssim_{p, \alpha, \alpha', k, l, T} 1.
  \end{equation*}
\end{proposition}
\begin{proof}
  Let $\contisp^{-\alpha} \defby \B^{-\alpha, M}_{\infty, \infty}$ and $V(t) \defby e^{tA} Z_M(s)$.
  We have to estimate each term of
  \begin{equation*}
    \sum_{i,j} \binom{k}{i} \binom{l}{j} (-1)^{i+j} V(t)^{:i,j:} Z_M(s+t)^{:k-i, l-j:}.
  \end{equation*}
  For $(i, j) = (k, l)$,
  \begin{equation*}
    \norm{V(t)^{:k,l:}}_{\contisp^{-\alpha}} \lesssim \norm{V(t)}_{\contisp^{2\alpha}}^{k+l-1} \norm{V(t)}_{\contisp^{-\alpha}}
    \lesssim t^{-\frac{3(k+l-1)\alpha}{2}}\norm{Z_M(s)}_{\contisp^{-\alpha}}^{k+l}.
  \end{equation*}
  Otherwise,
  \begin{align*}
    \norm{V(t)^{:i,j:} Z_M(s+t)^{:k-i,l-j:}}_{\contisp^{-\alpha}}
    &\lesssim \norm{V(t)}^{i+j}_{\contisp^{2\alpha}} \norm{Z_M(s+t)^{:k-i, l-j:}}_{\contisp^{-\alpha}}\\
    &\lesssim t^{-\frac{3(i+j)\alpha}{2}} \norm{Z_M(s)}_{\contisp^{-\alpha}}^{i+j} \norm{Z_M(s+t)^{:k-i, l-j:}}_{\contisp^{-\alpha}}.
  \end{align*}
  with $i + j \leq k + l -1$.
  Therefore, Theorem \ref{thm:main_estimate_of_OU} ends the proof.
\end{proof}
\begin{remark}\label{rem:continuity_of_nonstationary_OU}
  We have
  \begin{equation*}
    Z_M(0, t) = Z_M(t) - Z_M(0) + (1 - e^{tA}) Z_M(0).
  \end{equation*}
  According to Remark \ref{rem:continuity_of_OU} and Proposition \ref{prop:time_regularity_of_the_heat_flow_in_besov},
  we have, for some $\theta \in (0, 1)$,
  \begin{equation*}
    \expect[\sup_{0 < s \leq t} \norm{Z_M(0, s)}_{\contisp^{-\alpha}} ] \lesssim_{M, \alpha} t^{\theta},
    \quad t \in (0, 1).
  \end{equation*}
\end{remark}
\begin{proposition}\label{prop:independence_of_Z_M_s_t}
  For fixed $s \in \R$, $\set{Z_M^{:k,l:}(s, t) \given t \in (s, \infty), \, k, l \in \N}$ is independent of $\F_s$.
\end{proposition}
\begin{proof}
  According to Remark \ref{remark:convergence_of_mollified_processes},  $Z_M^{:k, l:}(s, t)$
  is the limit of
  \begin{multline*}
    \sum \binom{k}{i} \binom{l}{j} (-1)^{i+j} \left[e^{(t-s)A} Z_{M; \delta}(s) \right]^{:i,j:} \tilde{Z}_{M;\delta}^{:k-i, l-j:}(t) \\
    = H_{k, l}(\tilde{Z}_{M; \delta}(t) - e^{(t-s)A} \tilde{Z}_{M;\delta}(s), c_{M;\delta})
  \end{multline*}
  as $\delta \to 0$. Here $H_{k,l}$ is a complex Hermite polynomial and
  the equality results from Proposition \ref{prop:properties_of_hermite_polynomials}-(iii).
  It remains to notice that
  \begin{align*}
    Z_{M;\delta}(t;x) &- e^{(t-s)A} Z_{M;\delta}(s;x) \\
    &= \int \indic_{\torus_M}(y) [K_M(t-r, \cdot) * \rho_{\delta}](x-y) \xi(drdy) \\
    &\hspace{3em} - \int \indic_{\torus_M}(y) \left\{ e^{(t-s)A}[K_M(s-r, \cdot)*\rho_{\delta}] \right\} (x-y) \xi(drdy) \\
    &= \int \indic_{[s,t]}(r) (K_M(t-r, \cdot) * \rho_{\delta})(x-y) \xi(drdy),
  \end{align*}
  is independent of $\F_s$.
\end{proof}

%% file: solution_on_torus.tex
\section{Construction of solutions on the torus}\label{sec:solving_CGL}
\subsection{A strategy and the main theorem}\label{subsec:review_of_main_thm}
In this section, we construct solutions on the torus $\torus_M$ as in \cite[Section 6]{MW2dim}.
As the actual value of $M$ is irrelevant, we set $M = 1$.
As noted in the introduction, the equation \eqref{eq:CGL} is ill-defined due to the low regularity of the white noise $\xi$.
To overcome this difficulty, we employ a strategy first introduced in \cite{DD03}.
Namely, we decompose a solution $u = Z + Y$, where $Z$ is the Ornstein-Uhlenbeck process constructed
in Section \ref{section:she}. Then $Y$ formally solves the equation \eqref{eq:pde_of_Y} below.
We regard powers of $Z$ as Wick powers constructed in Section \ref{section:she}.
Then, in the light of Schauder's estimate (Proposition \ref{prop:smoothing_in_besov}),  the regularity of $Y$
should be $2-\epsilon$, and therefore all the products in \eqref{eq:pde_of_Y} are well-defined thanks to
Corollary \ref{cor:multiplicative_inequality}.
Now the idea from rough path theory comes into play. After ill-definedness has disappeared via probabilistic methods,
we will solve the equation in a deterministic framework.

In this section, we will write $L^2$ for $L^2(\torus)$ and et cetera.
We set $\contisp^{\alpha} \defby \B_{\infty, \infty}^{\alpha}$ and $\B^{\alpha}_{p} \defby \B_{p, \infty}^{\alpha}$.
For $\underline{Z} = (Z, Z^2, \abs{Z}^2, \abs{Z}^2 Z)$ in the set
\begin{equation*}
  C([0, T]; \contisp^{-\alpha})
  \times
  \big( C((0, T]; \contisp^{-\alpha}) \big)^3,
\end{equation*}
we write
\begin{equation*}
  \norm{\underline{Z}_t}_{\zset_{\alpha}}  \defby  \max \{
  \norm{Z(t)}_{\contisp^{-\alpha}},   t^{\alpha}\norm{Z^2(t)}_{\contisp^{-\alpha}},
  \,t^{\alpha} \norm{\abs{Z}^2}_{\contisp^{-\alpha}}, \,t^{\alpha} \norm{\abs{Z}^2 Z}_{\contisp^{-\alpha}} \}.
\end{equation*}
We emphasize that $Z^2$, $\abs{Z}^2$ and $\abs{Z}^2 Z$ are not functions of $Z$.
We denote by $\zset$ the set
\begin{equation*}
  \set{\underline{Z} \in C([0, \infty); \contisp^{-\alpha}) \times C((0, \infty); \contisp^{-\alpha})^3 \given
  \forall T, \alpha \in (0, \infty), \,\, \sup_{0 < t \leq T} \norm{\underline{Z}_t}_{\zset_{\alpha}} < \infty  }.
\end{equation*}
We set
\begin{multline*}
  \Psi(Y_t, \underline{Z}_t) \defby
  (1+\lambda)(Z_t + Y_t) \\ - \nu(\abs{Y_t}^2Y_t + 2 Z\abs{Y_t}^2 + \conj{Z_t}Y_t^2 + 2\abs{Z_t}^2 Y_t + Z_t^2 \conj{Y_t} + \abs{Z_t}^2 Z_t)
\end{multline*}
as long as the multiplications on the right hand side make sense.
\begin{definition}\label{def:definition_of_Y}
  For given $Y_0 \in \contisp^{-\alpha_0}$, $\underline{Z} \in \zset$ and $T \in (0, \infty)$,
  we say that $Y$ is a solution of
  \begin{equation}
    \label{eq:pde_of_Y}
    \left \{
    \begin{aligned}
      &\partial_t Y = A Y + \Psi(Y, \underline{Z}), \\
      &Y(0, \cdot) = Y_0
    \end{aligned} \right.
  \end{equation}
  over $[0, T]$ if the following conditions are satisfied:
  \begin{enumerate}[(i)]
    \item There exists $\alpha_1 \in (0, \infty)$ and $\gamma \in (0, \frac{1}{3})$ such that
    $Y \in C((0, T]; \contisp^{\alpha_1})$ and
  \begin{equation}\label{eq:assumption_on_singularity_at_zero}
    \sup_{0 < t \leq T} t^{\gamma} \norm{Y_t}_{\contisp^{\alpha_1}} < \infty.
  \end{equation}
    \item For $t \in (0, T]$ we have
  \begin{equation}\label{eq:def_of_solution_Y}
    Y_t = e^{tA} Y_0 + \int_0^t e^{(t-s)A} \Psi(Y_s, \underline{Z}_s) \, ds .
  \end{equation}
  \end{enumerate}
  We let $S^T(Y_0, \underline{Z})$ be the set of solutions of \eqref{eq:pde_of_Y}.
\end{definition}
\begin{remark}
  The technical assumption \eqref{eq:assumption_on_singularity_at_zero} is needed to ensure
  well-definedness of the integral in \eqref{eq:def_of_solution_Y} and uniqueness of solutions.
  It will be revealed that a solution satisfies \eqref{eq:assumption_on_singularity_at_zero} for
  any $\alpha_1$, $\gamma$ satisfying
  \begin{equation}\label{eq:condition_of_alpha_1_and_gamma}
    \frac{\alpha_0 + \alpha_1}{2} < \gamma, \qquad \frac{\alpha_1}{2} + 2 \gamma < 1.
  \end{equation}
\end{remark}
The main theorem of this section is the following;
\begin{theorem}\label{thm:global_wellposedness}
  For all $\alpha_0 \in (0, \frac{2}{3})$, $T > 0$, $Z \in \zset$ and $Y_0 \in \contisp^{-\alpha_0}$,
  there exists exactly one solution of \eqref{eq:pde_of_Y} over $[0, T]$.
\end{theorem}
The proof of Theorem \ref{thm:global_wellposedness} is given at the end of Subsection \ref{subsec:bootstrap}.
Until Proposition \ref{prop:continuity_parameters},
we work in a deterministic framework.
Afterwards, we derive some results of solutions of the SCGL (Theorem \ref{thm:convergence_of_smooth_cgl}
and Theorem \ref{thm:coming_down_from_infinity}), where we work in a
probabilistic framework.

\subsection{The local well-posedness}\label{subsec:local_wellposedness}
We first show local well-posedness of \eqref{eq:pde_of_Y}.
\begin{theorem}\label{thm:local_wellposedness}
  Let $\alpha_0 \in (0, \frac{2}{3})$ and $K \in [1, \infty)$.
  Then, there exist $\alpha = \alpha(\alpha_0)$ and $T^* = T^*(\alpha, K) \in (0, \infty)$ such that
  for every $Y_0 \in \contisp^{-\alpha_0}$ and $\underline{Z} \in \zset$
  with $\norm{Y_0}_{\contisp^{-\alpha_0}} + \sup_{0<t\leq T^*} \norm{\underline{Z}_t}_{\zset_{\alpha}} \leq K$,
  the set $S^{T^*}(Y_0, \underline{Z})$ is a singleton.
\end{theorem}
\begin{proof}
  \emph{STEP 1.} We first construct a solution of \eqref{eq:pde_of_Y}.
  We choose
  $\alpha, \alpha_1 \in (0, \infty)$ and  $\gamma \in (0, \frac{1}{3})$ satisfying
  \begin{equation}\label{eq:condition_of_alpha_and_gamma}
    \alpha < \alpha_1, \quad \frac{\alpha_0 + \alpha_1}{2} < \gamma, \quad \frac{\alpha+ \alpha_1}{2} + 2 \gamma < 1.
  \end{equation}
  For $T^* \leq  1$, we set $\vertiii{Y}_{T^*} \defby \sup_{0 < t \leq T^*} t^{\gamma} \norm{Y_t}_{\contisp^{\alpha_1}} $
  and
  \begin{equation*}
  B_{T^*} \defby \set{Y \in C((0, T^*]; \contisp^{\alpha_1}) \given \vertiii{Y}_{T^*} \leq 1}.
  \end{equation*}
  We define
  \begin{equation*}
    \mathcal{M}_{T^*} Y(t) \defby  e^{tA} Y_0 + \int_0^t e^{(t-s)A} \Psi(Y_s, \underline{Z}_s) \, ds,
    \quad 0 < t \leq T^*.
  \end{equation*}
  We show that for small $T^*$ the map $\mathcal{M}_{T^*}: B_{T^*} \to B_{T^*}$ is a contraction.
  Proposition \ref{prop:smoothing_in_besov} bounds the first term by
  \begin{equation*}
    \norm{e^{tA} Y_0}_{\contisp^{\alpha_1}} \lesssim t^{-\frac{\alpha_0 + \alpha_1}{2}} \norm{Y_0}_{\contisp^{-\alpha_0}}.
  \end{equation*}
  By \eqref{eq:condition_of_alpha_and_gamma} the norm $\vertiii{e^{\cdot A} Y_0}_{T^*}$ becomes small for small $T^*$. The map
  $t \mapsto e^{tA} Y_0 \in \contisp^{\alpha}$ is continuous on $(0, T^*]$.

  For the second term, we note that
  $\norm{\Psi(Y_s, \underline{Z}_s)}_{\contisp^{-\alpha}}$ $\lesssim s^{-3\gamma}$
  by Corollary \ref{cor:multiplicative_inequality}. Hence we have by Proposition \ref{prop:smoothing_in_besov}
  \begin{equation*}
    \norm{ \int_0^t e^{(t-s)A} \Psi(Y_s, \underline{Z}_s) \, ds }_{\contisp^{\alpha_1}}
    \lesssim \int_0^t (t-s)^{- \frac{\alpha + \alpha_1}{2}} s^{-3 \gamma} \, ds
    \lesssim t^{1 - \frac{\alpha + \alpha_1}{2} - 3\gamma}.
  \end{equation*}
  By \eqref{eq:condition_of_alpha_and_gamma},
  $\vertiii{\cdot}_{T^*}$-norm of the second term becomes small for small $T^*$.
  To prove the continuity of the map $t \mapsto \mathcal{M}_{T^*} Y (t)$,
  we take $\alpha'_1$ sufficiently close to but greater than $\alpha_1$. For $0 < t < t'$,
  \begin{align*}
    \MoveEqLeft
    \norm{ \int_0^t e^{(t-s)A} \Psi(Y_s, \underline{Z}_s) \, ds
    - \int_0^{t'} e^{(t'-s)A} \Psi(Y_s, \underline{Z}_s) \, ds }_{\contisp^{\alpha_1}} \\
    &\leq \int_0^t \norm{(e^{(t-s)A} - e^{(t'-s)A}) \Psi(Y_s, \underline{Z}_s)}_{\contisp^{\alpha_1}} \, ds
    + \int_t^{t'} \norm{e^{(t'-s)A} \Psi(Y_s, \underline{Z}_s) }_{\contisp^{\alpha_1}}\, ds  \\
    & \lesssim \int_0^t \abs{t' - t}^{\frac{\alpha'_1 - \alpha_1}{2}} \norm{e^{(t-s)A} \Psi(Y_s, \underline{Z}_s)}_{\contisp^{\alpha'_1}} \, ds
    +\int_t^{t'} (t' - s)^{- \frac{\alpha+\alpha_1}{2}} s^{-3\gamma} \, ds \\
    &\lesssim \abs{t' - t}^{\frac{\alpha'_1 - \alpha_1}{2}} t^{1- \frac{\alpha + \alpha_1'}{2} - 3 \gamma}
    + (t')^{1 - \frac{3\alpha}{2} - 3 \gamma} \int_{t/t'}^1 (1-s)^{- \frac{\alpha+\alpha_1}{2}} s^{-3\gamma} \, ds.
  \end{align*}
   In the second inequality, we used Proposition \ref{prop:smoothing_in_besov} and Proposition \ref{prop:time_regularity_of_the_heat_flow_in_besov}.
  Thus, the map $(0, T^*] \ni t \mapsto M_{T^*} Y(t) \in \contisp^{\alpha_1}$ is continuous.
  Finally we prove that the map $\mathcal{M}_{T^*}$ is a contraction for small $T^*$.
  Since
  \begin{equation*}
    \norm{\Psi(Y_s, \underline{Z}_s) - \Psi(Y_s', \underline{Z}_s)}_{\contisp^{-\alpha}}
    \lesssim \norm{Y_s - Y_s'}_{\contisp^{\alpha_1}} s^{-2 \gamma },
  \end{equation*}
  we have
  \begin{align*}
    \norm{\mathcal{M}_{T^*} Y(t) - \mathcal{M}_{T^*} Y'(t)}_{\contisp^{\alpha_1}}
    &\lesssim \int_0^t (t-s)^{-\frac{\alpha_1 + \alpha}{2}} s^{-2 \gamma} \norm{Y_s - Y_s'}_{\contisp^{\alpha_1}} ds \\
    &\lesssim t^{1 - \frac{\alpha + \alpha_1}{2} - 3 \gamma} \norm{Y_s - Y_s'}.
  \end{align*}
  Therefore, $\mathcal{M}_{T^*}$ is a contraction for small $T^*$.
  Then there is a unique fixed point of the map $\mathcal{M}_{T^*}$, which is a solution of \eqref{eq:pde_of_Y}.

  \emph{STEP 2.} We show uniqueness of solutions. We let $T^*$ be the time constructed in \emph{STEP 1} and let
  $Y, Y' \in S^{T^*}(Y_0, \underline{Z})$.
  Set
  \begin{equation*}
    T^{**} \defby \sup \set{t \in [0, T^*] \given \forall s \leq t, Y_s = Y'_s}.
  \end{equation*}
  Assume $T^{**} < T^*$. The argument in \emph{STEP 1} implies that,
  possibly by taking smaller $\alpha_1 \in (0, 1)$ and larger $\gamma \in (0, \frac{1}{3})$,
  there exists $\delta \in (0, T^* - T^{**})$
  such that $(Y_{T^{**} + t})_{t \leq \delta}$, $(Y'_{T^{**} + t})_{t \leq \delta} \in B_{\delta}$ and
  $\tilde{\mathcal{M}}: B_{\delta} \to B_{\delta}$ is a contraction, where
  \begin{equation*}
    \tilde{\mathcal{M}} y (t) \defby e^{tA} Y_{T^{**}} + \int_0^t e^{(t-s)A} \Psi(y_s, \underline{Z}_{T^{**} + s}) ds
    \qquad (y \in B_{\delta}).
  \end{equation*}
  Since both $(Y_{T^{**} + t})_{t \leq \delta}$ and $(Y'_{T^{**} + t})_{t \leq \delta}$
  are fixed points of $\tilde{\mathcal{M}}$,
  $Y_{T^{**} + t} = Y'_{T^{**} + t}$ for $t \leq \delta$, which contradicts the definition of $T^{**}$.
  Therefore $T^{**} = T^*$ and $Y = Y'$.
\end{proof}

\begin{remark}
  We have $S^T(Y_0, \underline{Z}) \subset C([0, T]; \contisp^{-\alpha_0}) \cap
  C((0, T]; \contisp^{\beta})$ for every $\beta \in (1, 2)$.
  Indeed, take $Y \in S^T(Y_0, \underline{Z})$ and $\alpha \in (0, 2 - \beta)$. As shown in \emph{STEP 1} above,
  we have $\norm{\Psi(Y_s, \underline{Z}_s)}_{\contisp^{-\alpha}}$ $\lesssim s^{-3\gamma}$
  and hence
  \begin{equation}
    \norm{\int_0^t e^{(t-s)A} \Psi(Y_s, \underline{Z}_s) \, ds}_{\contisp^{\beta}}
    \lesssim \int_0^t (t-s)^{ - \frac{\alpha + \beta}{2}} s^{- 3 \gamma}
    \, ds \lesssim t^{1  - \frac{\alpha + \beta}{2} - 3\gamma}.
  \end{equation}
  Therefore $Y_t$ is $\contisp^{\beta}$-valued.
  As the continuity can be proved as in \emph{STEP 1} above,
  we see that $Y \in C((0, T]; \contisp^{\beta})$. Similarly we can show $Y \in C([0, T]; \contisp^{-\alpha_0})$.
\end{remark}
\begin{remark}\label{remark:form_of_T_star}
  $T^*$ can be of the form $T^* = C(\alpha_0, \alpha_1, \alpha, \gamma) K^{-\theta}$ for some $\theta \in (0, \infty)$.
\end{remark}

\subsection{A priori $L^p$ estimates}\label{subsec:a_priori_L_p}
We now explain our strategy to prove the global well-posedness of \eqref{eq:pde_of_Y}.
Uniqueness of solutions can be proved as in \emph{STEP 2} in the proof of Theorem \ref{thm:local_wellposedness}.
Therefore, the problem is to construct a global solution of \eqref{eq:pde_of_Y}.
Our strategy is to apply the fixed point argument in the proof of Theorem \ref{thm:local_wellposedness} repeatedly.

By Theorem \ref{thm:local_wellposedness}, there exists $Y \in S^{T^*}(Y_0, \underline{Z})$ for small $T^* \in (0, \infty)$.
We move to find $Y_{T^* + \cdot}$ satisfying
\begin{equation*}
  Y_{T^* + t} = e^{tA} Y_{T^*} + \int_0^t e^{(t-s)A} \Psi(Y_{T^*+s}, \underline{Z}_{T^*+s}) ds.
\end{equation*}
Replacing $(Y_0, \underline{Z})$ by $(Y_{T^*}, \underline{Z}_{T^*+\cdot})$,
Theorem \ref{thm:local_wellposedness} allows us to find
\begin{equation*}
  Y_{T^* + \cdot} \in S^{T^{**}}(Y_{T^*}, \underline{Z}_{T^*+ \cdot})
\end{equation*}
for some $T^{**} \in (0,\infty)$. Thus, we extended a solution of \eqref{eq:pde_of_Y} over $[0, T^*]$ to a solution of \eqref{eq:pde_of_Y}
over $[0, T^* + T^{**}]$. We then continue this process.

However, in order to construct a solution over $[0, T]$ for given $T \in (0, \infty)$, we have to choose $T^*, T^{**}, \ldots$
uniformly away from $0$. According to Theorem \ref{thm:local_wellposedness},
for fixed $\underline{Z}$, $T^*$ depends on the $\norm{\cdot}_{\contisp^{-\alpha_0}}$-norm of the initial
condition. Therefore, if we know a priori that every solution $Y$ of \eqref{eq:pde_of_Y} over $[0, \tilde{T}]$ with
$\tilde{T} \leq T$ satisfies
\begin{equation}\label{eq:a_priori_estimate_in_C_minus_alpha_zero}
  \sup_{0 \leq t \leq \tilde{T}} \norm{Y_t}_{\contisp^{-\alpha_0}} \lesssim_{T, Y_0, \underline{Z}} 1,
\end{equation}
then we can take $T^*, T^{**}, \ldots$ uniformly away from $0$.

The aim of this subsection is to prove a priori $L^p$ estimate for $p \in [2, p(\mu))$, where
\begin{equation}\label{eq:def_of_p_mu}
  p(\mu) \defby 2(1 + \mu^2 + \mu \sqrt{1 + \mu^2}).
\end{equation}
The proof is in the spirit of \cite[Section 6]{MW2dim} and \cite[Section 3]{TW18}.
However, with their methods, we only obtain a prior $L^p$ estimate with $p < p(\mu)$ due to the presence of the
dispersion $i \Delta$.

We stress that this a priori $L^p$ estimate is insufficient to construct a global solution of \eqref{eq:pde_of_Y}
for small $\mu$.
Indeed, when $\mu > \frac{1}{2\sqrt{2}}$, we can take $\alpha_0 \in (0, \frac{2}{3})$ and $p \in [2, p(\mu))$ such that
the embedding $L^p \hookrightarrow \contisp^{-\alpha_0}$ holds.
Then we obtain the estimate \eqref{eq:a_priori_estimate_in_C_minus_alpha_zero}, which leads to construction of a global solution.
In contrast, when $\mu \leq \frac{1}{2\sqrt{2}}$, such embedding does not hold, and hence a priori $L^p$ estimate for
$p < p(\mu)$ is insufficient. This problem will be addressed in Subsection \ref{subsec:bootstrap}.

In this subsection, we fix $\alpha_0 \in (0, \frac{2}{3})$, $\beta \in (0, 2)$, $Y_0 \in \contisp^{-\alpha_0}$
and $\underline{Z} \in \zset_{\alpha}$.
\begin{proposition}\label{prop:Holder_continuity_of_Y}
  Let $Y \in S^T(Y_0, \underline{Z})$, $0 < t_0 \leq T$ and $\delta \in (0, \beta)$. Then
  the map $Y: [t_0, T] \to \contisp^{\delta}$ is
  $\frac{\beta - \delta}{2}$-H\"older continuous.
\end{proposition}
\begin{proof}
  Set $\Psi_s \defby \Psi(Y_s, \underline{Z}_s)$. As $Y$ is a solution of \eqref{eq:pde_of_Y}, we have
  \begin{equation*}
    Y_t = e^{(t- t_0) A} Y_{t_0} + \int_{t_0}^t e^{(t-s)A} \Psi_s \, ds.
  \end{equation*}
  For $t < t'$, the first term can be estimated by
  \begin{equation*}
    \norm{e^{(t'-t_0)A} Y_{t_0} - e^{(t-t_0) A} Y_{t_0}}_{\contisp^{\delta}}
    \lesssim (t' - t)^{\frac{\beta- \delta}{2}} \norm{e^{(t-t_0) A} Y_{t_0}}_{\contisp^{\beta}}
    \lesssim (t'- t)^{\frac{\beta- \delta}{2}} \norm{Y_{t_0}}_{\contisp^{\beta}},
  \end{equation*}
  where we used Proposition \ref{prop:time_regularity_of_the_heat_flow_in_besov}.
  For the second term, since $Y \in C([t_0, T], \contisp^{\beta})$, we have the estimate
  \begin{equation*}
    \norm{\Psi_s}_{\contisp^{-\alpha}} \lesssim_Y s^{-\alpha} \qquad s \in [t_0, T].
  \end{equation*}
  Then we have, by Proposition \ref{prop:time_regularity_of_the_heat_flow_in_besov} and
  Proposition \ref{prop:smoothing_in_besov},
  \begin{align*}
    \MoveEqLeft
    \norm{\int_{t_0}^{t'} e^{(t'-s)A} \Psi_s \, ds - \int_{t_0}^t e^{(t-s)A} \Psi_s \, ds}_{\contisp^{\delta}}  \\
    &\leq \int_{t_0}^t \norm{e^{(t'-s)A} \Psi_s - e^{(t-s)A} \Psi_s}_{\contisp^{\delta}} \, ds
    + \int_t^{t'} \norm{e^{(t'-s)A} \Psi_s}_{\contisp^{\delta}} \, ds \\
    &\lesssim \int_{t_0}^t (t' - t)^{\frac{\beta - \delta}{2}} \norm{e^{(t-s)A} \Psi_s}_{\contisp^{\beta}} \, ds
    + \int_t^{t'} \norm{e^{(t'-s)A} \Psi_s}_{\contisp^{\delta}} \, ds \\
    &\lesssim \int_{t_0}^t (t'-t)^{\frac{\beta - \delta}{2}} (t-s)^{- \frac{\alpha + \beta}{2}} s^{-\alpha} \, ds
    + \int_t^{t'} (t' -s)^{- \frac{\alpha + \delta}{2}} s^{-\alpha} \, ds \\
    &\lesssim (t'-t)^{\frac{\beta - \delta}{2}} t^{1-\frac{3 \alpha}{2} - \frac{\beta}{2}}
    + (t'-t)^{1 - \frac{3 \alpha}{2} - \frac{\delta}{2}}.
  \end{align*}
  By taking sufficiently small $\alpha = \alpha(\beta, \delta) > 0$, we conclude the proof.
\end{proof}

\begin{proposition}\label{prop:mild_solution}
  Let $Y \in S^T(Y_0, \underline{Z})$. Then for $\phi \in \contisp^1$ we have the identity
  \begin{equation*}
    \inp{Y_t}{\phi} - \inp{Y_{t_0}}{\phi} =
    \int_{t_0}^t \left\{ - (i + \mu) \inp{\nabla Y_s}{\nabla \phi} - \inp{Y_s}{\phi} +
    \inp{\Psi(Y_s, \underline{Z}_s)}{\phi} \right\} \, ds
  \end{equation*}
  for $0 < t_0 \leq t \leq T$.
\end{proposition}
\begin{proof}
  Set $\Psi_s \defby \Psi(Y_s, \underline{Z}_s)$.
  We begin by showing that for $\phi \in C^{\infty}(\torus)$ and $t \leq T$,
  \begin{equation}
    \label{eq:weak_sol_against_smooth}
    \inp{Y_t}{\phi} - \inp{Y_{t_0}}{\phi} = \int_{t_0}^t \left(\inp{Y_s}{A \phi} + \inp{\Psi_s}{\phi}\right)  ds.
  \end{equation}
  Since $Y$ is a mild solution of \eqref{eq:pde_of_Y}, we have
  \begin{equation*}
    \int_{t_0}^t \inp{Y_s}{A\phi} \, ds = \int_{t_0}^t \inp*{e^{(s-t_0) A} Y_{t_0} +
     \int_{t_0}^s e^{(s-u) A} \Psi_u \, du}{A \phi} \, ds.
  \end{equation*}
  Note that for $s > 0$,
  \begin{equation*}
    \inp{e^{(s-t_0)A} Y_{t_0}}{A\phi} = \inp{Ae^{(s-t_0)A} Y_{t_0}}{\phi} = \frac{d}{ds} \inp{e^{(s-t_0)A} Y_{t_0}}{\phi}.
  \end{equation*}
  Therefore
  \begin{equation*}
    \int_{t_0}^t \inp{e^{(s-t_0)A}Y_{t_0}}{A \phi} \, ds = \inp{e^{(t-t_0)A}Y_{t_0}}{\phi} - \inp{Y_{t_0}}{\phi}.
  \end{equation*}
  Similarly, we obtain
  \begin{equation*}
    \int_{t_0}^t \int_{t_0}^s A e^{(s-u) A} \Psi_u \, du \, ds
    = \int_{t_0}^t \int_u^t A e^{(s-u) A} \Psi_u \, ds \, du
    = \int_{t_0}^t (e^{(t-u) A} - 1) \Psi_u \, du.
  \end{equation*}
  Thus, we get
  \begin{equation*}
    \int_{t_0}^t \inp{Y_s}{A \phi} \, ds = \inp{e^{(t-t_0)A} Y_{t_0}}{\phi} - \inp{Y_{t_0}}{\phi}
    + \int_{t_0}^t \inp{e^{(t-s)A} \Psi_s}{\phi} \, ds - \int_{t_0}^t \inp{\Psi_s}{\phi} \, ds,
  \end{equation*}
  which proves \eqref{eq:weak_sol_against_smooth}.

  Now integration by parts implies
  \begin{equation*}
    \inp{Y_t}{\phi} - \inp{Y_{t_0}}{\phi} = \int_{t_0}^t \left\{-(i+\mu) \inp{\nabla Y_s}{\nabla \phi}
    - \inp{Y_s}{\phi} + \inp{\Psi_s}{\phi} \right\} \, ds,
  \end{equation*}
  and, as a result, the density argument finishes the proof.
\end{proof}
\begin{proposition}\label{prop:L^p_identity}
  Let $Y \in S^T(Y_0, \underline{Z})$ and $p \in [2, \infty)$. Then we have the identity
  \begin{multline*}
    \frac{1}{p} ( \norm{Y_t}_{L^p}^p - \norm{Y_{t_0}}_{L^p}^p) \\= \Re \Big[ \int_{t_0}^t \{
     -(i+\mu) \inp{\nabla Y_s}{\nabla(\conj{Y_s} \abs{Y_s}^{p-2})}
    - \norm{Y_s}_{L^p}^p + \inp{\Psi(Y_s, \underline{Z}_s)}{\conj{Y_s} \abs{Y_s}^{p-2}} \} \, ds \Big]
  \end{multline*}
  for $0 < t_0 \leq t \leq T$.
\end{proposition}
\begin{proof}
  Let $\underline{\mathbbm{t}} = (t_0 < t_1 < \cdots < t_n = t)$ be a partition of $[0, t]$. By Proposition \ref{prop:mild_solution},
  we have $\norm{Y_t}_{L^p}^p - \norm{Y_{t_0}}_{L^p}^p - \mathscr{S}(\underline{\mathbbm{t}}) =
  \mathscr{T}(\underline{\mathbbm{t}})$,
  where
  \begin{equation*}
    \mathscr{S}(\underline{\mathbbm{t}}) \defby \sum_{i=0}^{n-1} \inp{Y_{t_{i+1}}}{\conj{Y_{t_{i+1}}} \abs{Y_{t_{i+1}}}^{p-2} -
    \conj{Y_{t_i}} \abs{Y_{t_i}}^{p-2}},
  \end{equation*}
  and
  \begin{multline*}
    \mathscr{T}(\underline{\mathbbm{t}}) \defby \sum_{i=0}^{n-1} \int_{t_i}^{t_{i+1}} \{- (i + \mu) \inp{\nabla Y_s}{\nabla(
    \conj{Y_{t_i}} \abs{Y_{t_i}}^{p-2})} \\
    - \inp{Y_s}{\conj{Y}_{t_i} \abs{Y_{t_i}}^{p-2}} + \inp{\Psi_s}{\conj{Y_{t_i}} \abs{Y_{t_i}}^{p-2}} \} ds .
  \end{multline*}
  We analyze the limits of $\mathscr{S}(\underline{\mathbbm{t}})$ and $\mathscr{T}(\underline{\mathbbm{t}})$
  as the mesh size tends to zero.
  Lemma \ref{lem:holder_fractional_power} implies
  \begin{equation*}
    \sum_{i=0}^{n-1} \int_{t_i}^{t_{i+1}} \inp{\nabla Y_s}{\nabla(\conj{Y_{t_i}} \abs{Y_{t_i}}^{p-2})} \, ds \to
    \int_{t_0}^t \inp{\nabla Y_s}{\nabla (\conj{Y_s} \abs{Y_s}^{p-2})} \, ds .
  \end{equation*}
  and
  \begin{multline*}
    \sum_{i=0}^{n-1} \int_{t_i}^{t_{i+1}} (-\inp{Y_s}{\conj{Y}_{t_i}\abs{Y_{t_i}}^{p-2}}
    + \inp{\Psi_s}{\conj{Y_{t_i}} \abs{Y_{t_i}}^{p-2}})  ds \\
    \to \int_{t_0}^t (- \norm{Y_s}_{L^p}^p + \inp{\Psi_s}{\conj{Y_s} \abs{Y_s}^{p-2}})  ds .
  \end{multline*}

  It remains to prove $\mathscr{S}(\underline{\mathbbm{t}}) \to \frac{p-1}{p} (\norm{Y_t}_{L^p}^p - \norm{Y_{t_0}}_{L^p}^p)$.
  Discrete version of integration by parts implies
  \begin{equation*}
    \mathscr{S}(\underline{\mathbbm{t}}) = \norm{Y_t}_{L^p}^p - \norm{Y_{t_0}}_{L^p}^p
    - \sum_{i=0}^{n-1} \inp{\conj{Y_{t_i}} \abs{Y_{t_i}}^{p-2}}{Y_{t_{i+1}} - Y_{t_i}}.
  \end{equation*}
  When $\abs{x} \leq R$, we have
  \begin{equation*}
    \abs{x+h}^p = \abs{x}^p + p \Re(\conj{x} \abs{x}^{p-2} h) + O_R( \abs{h}^2 ) \quad \mbox{as } h \to 0.
  \end{equation*}
  Therefore,
  \begin{equation*}
    \abs{Y_t}^p - \abs{Y_{t_0}}^p =  \sum_{i=0}^{n-1} \left\{p \Re[ \conj{Y_{t_i}} \abs{Y_{t_i}}^{p-2} (Y_{t_{i+1}} - Y_{t_i})]
     + O(\norm{Y_{t_{i+1}} - Y_{t_i}}_{L^{\infty}}^2) \right\}.
  \end{equation*}
  Then, in view of Proposition \ref{prop:Holder_continuity_of_Y}, we obtain
  \begin{equation*}
    \Re \left[  \sum_{i=0}^{n-1} \inp{Y_{t_i} \abs{Y_{t_i}}^{p-2}}{Y_{t_{i+1}} - Y_{t_i}} \right]
    = \frac{1}{p} ( \norm{Y_t}_{L^p}^p - \norm{Y_{t_0}}_{L^p}^p ) + o(1). \qedhere
  \end{equation*}
\end{proof}
We recall the notation \eqref{eq:def_of_p_mu}.
\begin{proposition}\label{prop:a_priori_L^p_estimate}
  Assume $p \in [2,  p(\mu) )$.
  Then there exist $b \in [1, \infty)$ and $c \in (0, \infty)$ such that
  for each $\alpha \in (0, \frac{1}{2})$ we can find a constant $C \in (0, \infty)$ such that
  for $T \in (0, \infty)$, $0 < t_0 \leq t \leq T$ and $Y \in S^T(Y_0, \underline{Z})$,
  \begin{equation}
    \label{eq:L^p_a_priori_estimate}
    \begin{multlined}
    \norm{Y_t}_{L^p}^p - \norm{Y_{t_0}}_{L^p}^p
    + c \int_{t_0}^t \norm{Y_s}_{L^{p+2}}^{p+2}  ds + c \int_{t_0}^t \norm{\abs{\nabla Y_s}^2 \abs{Y_s}^{p-2}}_{L^1} ds  \\
    \leq C \int_{t_0}^t (1 + s^{-\alpha}\norm{\underline{Z}_s}_{\zset_{\alpha}})^{b}  ds.
    \end{multlined}
  \end{equation}
\end{proposition}
\begin{proof}
  First we note that
  \begin{equation*}
    \nabla ( \conj{Y_s} \abs{Y_s}^{p-2} )
    = \frac{p}{2} \abs{Y_s}^{p-2} \nabla \conj{Y_s} + \frac{p-2}{2} \abs{Y_s}^{p-4} \conj{Y_s}^{2} \nabla Y_s,
  \end{equation*}
  and thus,
  \begin{align*}
    \Re & \inp{\nabla Y_s}{\nabla ( \conj{Y_s} \abs{Y_s}^{p-2} )}\\
    &= -\frac{\mu p}{2} \inp{1}{\abs{\nabla Y_s}^2 \abs{Y_s}^{p-2}}
    + \Re \frac{p-2}{2} (i + \mu) \inp{\nabla{Y_s} \cdot \nabla{Y_s}}{\abs{Y_s}^{p-4} \conj{Y_s}^2} \\
    &\leq - \delta \inp{1}{\abs{\nabla Y_s}^2 \abs{Y_s}^{p-2}},
  \end{align*}
  where $\delta \defby \frac{\mu p}{2} - \frac{p-2}{2} \sqrt{\mu^2 + 1} > 0$.
  Therefore, by Proposition \ref{prop:L^p_identity},
  \begin{align}
    \label{eq:long_estimate_of_L^p_Y}
    \MoveEqLeft
    \frac{1}{p}( \norm{Y_t}_{L^p}^p - \norm{Y_{t_0}}_{L^p}^p) + \delta \int_{t_0}^t \inp{1}{\abs{\nabla Y_s}^2 \abs{Y_s}^{p-2}} \, ds
    + \Re (\nu) \int_{t_0}^t \inp{1}{\abs{Y_s}^{p+2}}  ds \nonumber \\
    \begin{split}
      \lesssim_{\nu}& \int_{t_0}^t \big\{ \inp{1}{\abs{Y_s}^p} + \abs{\inp{Z_s}{\conj{Y_s}\abs{Y_s}^{p-2}}}
      + \abs{\inp{Z_s}{\conj{Y_s}\abs{Y_s}^p}} + \abs{\inp{\conj{Z_s}}{Y_s^3 \abs{Y_s}^{p-2}}} \\
      &+ \inp{\abs{Z_s}^2}{\abs{Y_s}^p} + \abs{\inp{Z_s^2}{\conj{Y_s}^2 \abs{Y_s}^{p-2}}}
      + \abs{\inp{\abs{Z_s}^2\conj{Z_s}}{\conj{Y_s} \abs{Y_s}^{p-2}}} \big\} ds.
    \end{split}
  \end{align}
  Each term in the integrand is of the form $\abs{\inp{\xi_s}{\eta_s}}$ with
  $\xi \in \set{Z, \conj{Z}, Z^2, \abs{Z}^2, \abs{Z}^2Z}$, $\eta = Y_s^a \conj{Y_s}^b$ and $2 \leq a+b \leq p+1$.
  Set $A_s \defby \norm{\abs{Y_s}^{p-2} \abs{\nabla Y_s}^2}_{L^1}$ and $B_s \defby
  \norm{Y_s}_{L^{p+2}}^{p+2}$. We show that $\abs{\inp{\xi_s}{\eta_s}}$
  is controlled by $A_s$ and $B_s$.
  Proposition \ref{prop:duality_in_besov} implies
  $\abs{\inp{\xi_s}{\eta_s}}$ $\lesssim \norm{\xi_s}_{\contisp^{-\alpha}} \norm{\eta_s}_{\B_{1,1}^{\alpha}}$.
  $\norm{\xi_s}_{\contisp^{-\alpha}}$ is bounded by $s^{-\alpha} \norm{\underline{Z}_s}_{\zset_{\alpha}}$.

  Now we estimate $\norm{\eta_s}_{\besov_{1,1}^{1/2}}$.
  By Proposition \ref{prop:besov_vs_sobolev}, we get
  \begin{equation*}
    \norm{\eta_s}_{\B_{1,1}^{1/2}} \lesssim
    \norm{\eta_s}_{L^1}^{\frac{1}{2}} \norm{\nabla \eta_s}_{L^1}^{\frac{1}{2}} + \norm{\eta_s}_{L^1}.
  \end{equation*}
  By Cauchy-Schwarz inequality, we obtain
  \begin{equation*}
    \norm{\nabla \eta_s}_{L^1}^2 \lesssim_{a, b} \norm{ \abs{Y_s}^{a+b-1} \abs{\nabla Y_s}}_{L^1}^2
    \leq A_s B_s^{\frac{2(a+b) - p}{p+2}}.
  \end{equation*}
  Therefore,
  \begin{equation*}
    \norm{\eta_s}_{\B_{1,1}^{1/2}} \lesssim A_s^{\frac{1}{4}} B_s^{\frac{1}{2} \frac{a+b}{p+2} +
    \frac{1}{4} \frac{2(a+b) - p}{p+2}} + B_s^{\frac{a+b}{p+2}}.
  \end{equation*}
  Since
  \begin{equation*}
    a_1 \defby \frac{1}{4} + \frac{1}{2} \frac{a+b}{p+2} + \frac{1}{4} \frac{2(a+b) - p}{p+2} < 1,
  \end{equation*}
  we have $\norm{\eta_s}_{\B_{1,1}^{1/2}}
  \lesssim A_s^{a_1} + B_s^{a_1} + B_s^{a_2}$ with $a_2 \defby \frac{a+b}{p+2}$.

  For $\epsilon > 0$, $\gamma \in (0, 1)$ and $x \geq 0$, Young's inequality for products implies
  \begin{equation*}
    c x^{\gamma} \leq (1-\gamma)(c \epsilon^{-\gamma})^{\frac{1}{1-\gamma}} + \gamma \epsilon x.
  \end{equation*}
  Using this, we can find a constant $C = C(\epsilon) \in (0, \infty)$ such that
  \begin{equation*}
    \abs{\inp{\xi_s}{\eta_s}} \leq \epsilon (A_s + B_s)
    + C \left( s^{-\frac{\alpha}{1-a_1}} \norm{\underline{Z}_s}_{\zset_{\alpha}}^\frac{1}{1- a_1} +
    s^{-\frac{\alpha}{1-a_2}} \norm{ \underline{Z}_s}_{\zset_{\alpha}}^\frac{1}{1- a_2} \right ).
  \end{equation*}
  Note that $a_1$ and $a_2$ are uniformly away from $1$ when $p \in [2, p(\mu))$.
  Therefore, the above argument implies that the right hand side of  \eqref{eq:long_estimate_of_L^p_Y} is bounded by
  \begin{equation*}
    \epsilon \int_{t_0}^t (A_s + B_s) \, ds + C \int_{t_0}^t (1+ s^{-\alpha} \norm{\underline{Z}_s}_{\zset_{\alpha}})^b ds
  \end{equation*}
  for some $b \in [1, \infty)$, independent of $p$ and $\alpha$.
  Now the claim of the proposition easily follows.
\end{proof}

\begin{corollary}\label{cor:strong_L^p_estimate_independent_of_initial}
  Let $p \in [2, p(\mu))$ and $\alpha \in (0, 1)$ sufficiently small.
  Then there exist positive constants $b = b(p)$ and $C = C(\alpha, p)$ such that
  for $0 < t \leq T$ and $Y \in S^T (Y_0, \underline{Z})$,
  \begin{equation*}
    \norm{Y_t}_{L^p}^p \leq C \max \{t^{-\frac{p}{2}},
    (1 + t^{-\alpha} \sup_{0 < s \leq T} \norm{\underline{Z}_s}_{\zset_{\alpha}})^b \}.
  \end{equation*}
\end{corollary}
\begin{proof}
  By proposition \ref{prop:a_priori_L^p_estimate} we have
  \begin{equation*}
    \partial_t \norm{Y_t}_{L^p}^p + c_1 \norm{Y_t}_{L^{p+2}}^{p+2}
    \leq c_2 (1 + t^{-\alpha} \norm{\underline{Z}_t}_{\zset_{\alpha}})^b,
  \end{equation*}
  and hence
  \begin{equation*}
    \partial_s \norm{Y_s}_{L^p}^p + c_1 (\norm{Y_s}_{L^p}^p)^{1 + \frac{2}{p}}
    \lesssim c_2 (1 + t^{-\alpha} \sup_{0< r \leq T} \norm{\underline{Z}_r}_{\zset_{\alpha}})^b,
    \quad s \in [2^{-1}t, t].
  \end{equation*}
  Then Lemma \ref{lem:ode_comparison} below shows that
  \begin{equation*}
    \norm{Y_t}_{L^p}^p \lesssim \max \{t^{-\frac{p}{2}},  (1 + t^{-\alpha} \sup_{0 < s \leq T}
    \norm{\underline{Z}_s}_{\zset_{\alpha}})^{\frac{pb}{p+2}} \}. \qedhere
  \end{equation*}
\end{proof}
\begin{lemma}\label{lem:ode_comparison}
  Suppose that a differentiable function $f:[0, T] \to [0, \infty)$ satisfies
  \begin{equation*}
    \frac{df}{dt} + c_1 f^{1+\lambda} \leq c_2
  \end{equation*}
  with $\lambda$, $c_1, c_2 \in (0, \infty)$. Then we have
  \begin{equation*}
    f(t) \lesssim_{\lambda} \max \left\{(c_1 t)^{-\frac{1}{\lambda}},  ( c_1^{-1} c_2 )^{\frac{1}{1+\lambda}} \right\}.
  \end{equation*}
\end{lemma}
\begin{proof}
  The proof is the same as \cite[Lemma 3.8]{TW18}.
  Let $t \in (0, T]$. First assume there exists $s \in [0, t]$ such that $\frac{c_1}{2} f(s)^{1+\lambda} \leq c_2$.
  Since $f(r) \geq (\frac{2c_2}{c_1})^{\frac{1}{1+\lambda}}$ implies $\frac{df}{dr}(r) \leq 0$, we see that
  $f \leq (\frac{2c_2}{c_1})^{\frac{1}{1+\lambda}}$ on $[s, T]$.

  Now assume that $f \geq (\frac{2c_2}{c_1})^{\frac{1}{1+\lambda}}$ on $[0, t]$. Then we have
  \begin{equation*}
    \frac{df}{ds} + \frac{c_1}{2} f^{1+\lambda} \leq 0 \qquad \mbox{on } [0, t],
  \end{equation*}
  and hence
  \begin{equation*}
    \frac{1}{-\lambda}(f(t)^{-\lambda} - f(0)^{-\lambda}) \leq - \frac{c_1 t}{2}.
  \end{equation*}
  This yields $f(t) \leq (\frac{2}{c_1\lambda t})^{\frac{1}{\lambda}}$.
\end{proof}

\subsection{A priori estimate by bootstrap arguments}\label{subsec:bootstrap}
As explained, a priori $L^p$ estimate obtained in the previous subsection is insufficient to construct a global solution.
The aim of this subsection is to upgrade a priori estimate by taking advantage of the smoothing effect of the semigroup $e^{tA}$.

We set
\begin{equation}\label{eq:z_norm_alpha_T}
  \znorm{\underline{Z}}_{\alpha, T} \defby \sup_{0 < t \leq T} \norm{\underline{Z}(t)}_{\zset_{\alpha}}.
\end{equation}
\begin{lemma}\label{lem:a_priori_estimate_1}
  There exist $p \in (2, \infty)$, $\gamma \in (0, 1)$ and $\kappa \in (0, \infty)$
  such that for $0<t_0 < t_1 \leq T$,
  $Y \in S^{t_1}(Y_0, \underline{Z})$ and sufficiently small $\alpha \in (0, 1)$,
  \begin{equation*}
    \int_{t_0}^{t_1} \norm{Y_t}_{\B_{3p}^{\gamma}}^3 \, dt
    \lesssim_{p, \alpha, \gamma, T}
    (t_0^{-1} + \znorm{\underline{Z}}_{\alpha, T})^{\kappa}.
  \end{equation*}
\end{lemma}
\begin{proof}
  According to Proposition \ref{prop:a_priori_L^p_estimate}, we can take $p_0 \in (2, 3)$ so that
  the inequality \eqref{eq:L^p_a_priori_estimate} holds with $p = p_0$. Set $p_1 \defby (p_0 + 2)/3$ and take $q$ such that
  \begin{equation*}
    1 + \frac{1}{3} = \frac{1}{p_1} + \frac{1}{q}.
  \end{equation*}
  Since
  \begin{equation*}
    \frac{1}{p_1} - \frac{1}{q} = \frac{2}{p_1} - \frac{4}{3} < \frac{2}{4/3} - \frac{4}{3} = \frac{1}{6},
  \end{equation*}
  we can take $p \in (2, 3)$ and $\gamma \in (0, 1)$ satisfying
  \begin{equation*}
    q \left[ \left( \frac{1}{p_1} - \frac{1}{3p} \right) + \frac{\gamma}{2} \right] < 1.
  \end{equation*}
  We additionally suppose $\gamma$ is so small that we can find $\alpha \in (0, 1)$ such that
  \begin{equation*}
    \frac{2 + \alpha + \gamma}{2} - \frac{1}{3p} < 1, \quad \frac{3(\alpha + \gamma)}{2} < 1.
  \end{equation*}

  A solution $Y \in S^{t_1}(Y_0, \underline{Z})$ satisfies
  \begin{equation}\label{eq:Y_is_a_mild_sol_from_t_0}
    \begin{multlined}
    Y_t = e^{(t-t_0)A} Y_{t_0} +  \int_{t_0}^t e^{(t-s)A} \big[-\nu( \abs{Y_s}^2 Y_s \\ + 2Z_s \abs{Y_s}^2 + \conj{Z}_s Y_s^2
    +2\abs{Z}_s^2 Y_s + Z^2_s \conj{Y}_s + \abs{Z}_s^2 Z_s ) + (1+\lambda) (Y_s + Z_s) \big] ds.
    \end{multlined}
  \end{equation}
  We evaluate the $L^3([t_0, {t_1}]; \B^{\gamma}_{3p})$-norm of each term of the right hand side.
  We begin to estimate the integral in \eqref{eq:Y_is_a_mild_sol_from_t_0}. The key is to use Young's convolution inequality.
  For the cubic term, we evaluate
  \begin{align*}
    \norm{e^{(t-s)A} (\abs{Y_s}^2 Y_s)}_{\B^{\gamma}_{3p}}
    &\lesssim \norm{e^{(t-s)A} (\abs{Y_s}^2 Y_s)}_{\B^{\gamma + 2(\frac{1}{p_1} - \frac{1}{3p})}_{p_1}}\\
    &\lesssim (t-s)^{-(\frac{1}{p_1} - \frac{1}{3p}) - \frac{\gamma}{2}} \norm{\abs{Y_s}^{3p_1}}_{L^1}^{\frac{1}{p_1}}.
  \end{align*}
  We applied Proposition \ref{prop:besov_embedding} in the first inequality
  and applied Propositions \ref{lem:embedding_between_besov_spaces} and \ref{prop:smoothing_in_besov} in the second inequality.
  Therefore, by Young's convolution inequality,
  \begin{align*}
    \MoveEqLeft
    \int_{t_0}^{{t_1}} \left( \int_{t_0}^t \norm{e^{(t-s)A} (\abs{Y_s}^2 Y_s)}_{\B^{\gamma}_{3p}} ds \right)^3 dt \hspace{15em}\\
    &\lesssim \int_{t_0}^{t_1} \left( \int_{t_0}^{t_1} \indic_{\{0\leq t-s \leq t_1-t_0\}}
     (t-s)^{-(\frac{1}{p_1} - \frac{1}{3p}) - \frac{\gamma}{2}}
    \norm{Y_s^{3p_1}}_{L^1}^{\frac{1}{p_1}} ds \right)^3 dt \\
    &\leq \left( \int_{0}^{t_1-t_0} t^{-q\left[ \frac{1}{p_1} - \frac{1}{3p} + \frac{\gamma}{2} \right]} dt \right)^{\frac{3}{q}}
    \left( \int_{t_0}^{t_1} \norm{Y_t^{3p_1}}_{L^1} dt \right)^{\frac{3}{p_1}} \\
    &\lesssim (t_1 - t_0)^{\frac{3}{q} - 3(\frac{1}{p_1} - \frac{1}{3p} + \frac{\gamma}{2})}
    \left( 1 + \norm{Y_{t_0}}_{L^{p_0}}^{p_0} + \znorm{\underline{Z}}_{\alpha, T}^b \right)^{\frac{3}{p_1}}
  \end{align*}
  if $\alpha < b^{-1}$. In the last inequality, we applied the inequality \eqref{eq:L^p_a_priori_estimate}.

  We move to the next term.
  \begin{align*}
    \norm{e^{(t-s)A} (Z_s \abs{Y_s}^2)}_{\B^{\gamma}_{3p}}
    &\lesssim \norm{e^{(t-s)A} (Z_s \abs{Y_s}^2)}_{\B_1^{\gamma + 2(1 - \frac{1}{3p})}} \\
    &\lesssim (t-s)^{-\frac{2+\alpha+\gamma}{2} + \frac{1}{3p}} \norm{Z_s \abs{Y_s}^2}_{\B^{-\alpha}_1} \\
    &\lesssim (t-s)^{-\frac{2+\alpha+\gamma}{2} + \frac{1}{3p}} \norm{Z_s}_{\contisp^{-\alpha}} \norm{\abs{Y_s}^2}_{\B^{2/3}_{1,1}}.
  \end{align*}
  We applied Proposition \ref{prop:besov_embedding} to the first inequality, applied Proposition \ref{prop:smoothing_in_besov}
  to the second inequality and applied Corollary \ref{cor:multiplicative_inequality}-(ii) and
  Lemma \ref{lem:embedding_between_besov_spaces}-(iv) to the third inequality.
  By Proposition \ref{prop:besov_vs_sobolev} and \eqref{eq:L^p_a_priori_estimate}, for $\epsilon \in (0, 1)$,
  \begin{equation}\label{eq:estimate_of_Y^2_by_sobolev}
    \norm{\abs{Y_s}^2}_{\B^{\epsilon}_{1, 1}} \lesssim \norm{Y_s^2}_{L^1}^{1 - \epsilon} \norm{Y_s \nabla \conj{Y_s}}_{L^1}^{\epsilon}
    + \norm{Y_s^2}_{L^1}
    \lesssim B^{1 - \frac{\epsilon}{2}} \norm{\nabla Y_s}_{L^2}^{\epsilon} + B,
  \end{equation}
  where
  \begin{equation*}
    B \defby 1 + \norm{Y_{t_0}^2}_{L^1} + t_0^{-\alpha b} \znorm{\underline{Z}}_{\alpha, T}^b.
  \end{equation*}
  Setting $\epsilon = \frac{2}{3}$ in \eqref{eq:estimate_of_Y^2_by_sobolev} and using Young's convolution inequality, we obtain
  \begin{align*}
    \MoveEqLeft
    \int_{t_0}^{t_1} \left( \int_{t_0}^t \norm{e^{(t-s)A} (Z_s \abs{Y_s}^2)}_{\B^{\gamma}_{3p}} ds \right)^3 dt \\
    &\lesssim \left( \int_{0}^{t_1 - t_0} t^{- \frac{2 + \alpha + \gamma}{2} + \frac{1}{3p}} dt \right)^3
    \int_{t_0}^{t_1} (B^{5} \norm{\nabla Y_t}_{L^2}^2 + B^6) dt \\
    &\lesssim (t_1 - t_0)^{3(\frac{1}{3p} - \frac{\alpha + \gamma}{2})} B^6.
  \end{align*}
  The term $\conj{Z}_s Y_s^2$ can be handled similarly.

  For the term $\abs{Z_s}^2 Y_s$, we can similarly estimate
  \begin{align*}
    \norm{e^{(t-s)A} (\abs{Z_s}^2 Y_s)}_{\B^{\gamma}_{3p}}
    &\lesssim (t-s)^{-\frac{2+\alpha+\gamma}{2} + \frac{1}{3p}} \norm{\abs{Z_s}^2}_{\contisp^{-\alpha}} \norm{Y_s}_{\B^{1/3}_{1,1}} \\
    &\lesssim (t-s)^{-\frac{2+\alpha+\gamma}{2} + \frac{1}{3p}} \norm{\abs{Z_s}^2}_{\contisp^{-\alpha}}
    (B^{\frac{1}{3}} \norm{\nabla Y_s}_{L^2}^{\frac{1}{3}} + B^{\frac{1}{2}}).
  \end{align*}
  By Young's convolution inequality and Cauchy-Schwarz inequality,
  \begin{align*}
    \MoveEqLeft
    \int_{t_0}^{t_1} \left( \int_{t_0}^t \norm{e^{(t-s)A} (\abs{Z_s}^2 Y_s)}_{\B^{\gamma}_{3p}} ds \right)^3 dt \\
    &\lesssim \left( \int_{0}^{t_1 - t_0} t^{-\frac{2+\alpha+\gamma}{2} + \frac{1}{3p}} dt \right)^3
    \int_{t_0}^{t_1} \norm{\abs{Z_t}^2}^3_{\contisp^{-\alpha}} (B \norm{\nabla Y_t}_{L^2} + B^{\frac{3}{2}} )dt \\
    &\lesssim (t_1 - t_0)^{\frac{1}{p} - \frac{3(\alpha + \gamma)}{2}}
    \left(\int_{t_0}^{t_1} \norm{\abs{Z_s}^2}_{\contisp^{-\alpha}}^6 ds \right)^{\frac{1}{2}}
    \left[ B \left( \int_{t_0}^{t_1} \norm{Y_t}_{L^2}^2 dt \right)^{\frac{1}{2}} + B^{\frac{3}{2}} \right] \\
    &\lesssim (t_1 - t_0)^{\frac{1}{p} - \frac{3(\alpha+\gamma)}{2} + \frac{1 - 6 \alpha}{2}} B^{\frac{9}{2}}
  \end{align*}
  if $\alpha < \frac{1}{6}$.
  Estimates of the remaining terms in the integrand in \eqref{eq:Y_is_a_mild_sol_from_t_0} are similar.

  Now we estimate the term $e^{(t-t_0)A} Y_{t_0}$. Note that we have
  \begin{equation*}
    e^{(t-t_0)A}Y_{t_0} = e^{(t- \frac{t_0}{2}) A} Y_{\frac{t_0}{2}}
    + \int_{\frac{t_0}{2}}^{t_0} e^{(t-s)A} \Psi(Y_s, \underline{Z}_s) ds.
  \end{equation*}
  The above argument shows that
  \begin{equation*}
    \int_{t_0}^{t_1} \left( \int_{\frac{t_0}{2}}^{t_0} \norm{e^{(t - s)A}
    \Psi(Y_s, \underline{Z}_s)}_{\B^{\gamma}_{3p}} ds \right)^3 dt
    \lesssim ( 1 + t_0^{-1} + \norm{Y_{\frac{t_0}{2}}}_{L^{p_0}} + \znorm{\underline{Z}}_{\alpha, T} )^{\kappa}.
  \end{equation*}
  By Proposition \ref{prop:besov_embedding} and Proposition \ref{prop:smoothing_in_besov},
  \begin{equation*}
  \norm{e^{(t-\frac{t_0}{2})A} Y_{\frac{t_0}{2}}}_{\B^{\gamma}_{3p}}
  \lesssim \norm{e^{(t-\frac{t_0}{2})A} Y_{\frac{t_0}{2}}}_{\B^{\gamma + 2(\frac{1}{p_0} - \frac{1}{3p})}_{p_0}}
  \lesssim (t - \frac{t_0}{2})^{-\frac{\gamma}{2} - \frac{1}{p_0} + \frac{1}{3p}} \norm{Y_{\frac{t_0}{2}}}_{L^{p_0}},
  \end{equation*}
  and hence,
  \begin{equation*}
    \int_{t_0}^{t_1} \norm{e^{(t-t_0)A} Y_{t_0}}^3_{\B^{\gamma}_{3p}} ds
    \lesssim (t_1 - t_0) t_0^{-3(\frac{\gamma}{2} + \frac{1}{p_0} - \frac{1}{3p})} \norm{Y_{\frac{t_0}{2}}}^3_{L^{p_0}}.
  \end{equation*}

  Finally, recall that by Corollary \ref{cor:strong_L^p_estimate_independent_of_initial}
  \begin{equation*}
  \norm{Y_{t_0}}_{L^{p_0}}^{p_0} \lesssim  t_0^{-\frac{p_0}{2}} + t_0^{-\alpha b}
  \znorm{\underline{Z}}_{\alpha, T}^b.
  \qedhere
  \end{equation*}
\end{proof}

\begin{lemma}\label{lem:a_priori_estimate_2}
  There exist $p \in (2, \infty)$, $\gamma \in (0, 1)$ and $\kappa \in (0, \infty)$
  such that for  $0<t_0 < t_1 \leq T$,
  $Y \in S^{t_1}(Y_0, \underline{Z})$ and sufficiently small $\alpha \in (0, 1)$,
  \begin{equation*}
    \sup_{t_0 \leq t \leq t_1} \norm{Y_t}_{\B_{p}^{\gamma}}
    \lesssim_{p, \gamma, \alpha, T}
    (t_0^{-1} +
    \znorm{\underline{Z}}_{\alpha, T})^{\kappa}.
  \end{equation*}
\end{lemma}
\begin{proof}
  We evaluate each term on the right hand side of \eqref{eq:Y_is_a_mild_sol_from_t_0}.
  Take $p, \gamma, \kappa$ in Lemma \ref{lem:a_priori_estimate_1}.
  We have, by Proposition \ref{prop:smoothing_in_besov} and Corollary \ref{cor:multiplicative_inequality}-(i),
  \begin{equation*}
    \norm{e^{(t-s)A} (\abs{Y_s}^2 Y_s)}_{\B^{\gamma}_p} \lesssim \norm{\abs{Y_s}^2 Y_s}_{\B^{\gamma}_p}
    \lesssim \norm{Y_s}^3_{\B_{3p}^{\gamma}},
  \end{equation*}
  and thus by Lemma \ref{lem:a_priori_estimate_1}
  \begin{equation*}
    \sup_{t_0 \leq t \leq t_1} \int_{t_0}^{t} \norm{e^{(t-s)A} (\abs{Y_s}^2 Y_s)}_{\B^{\gamma}_{p}} ds
    \lesssim ( t_0^{-1} +  \znorm{\underline{Z}}_{\alpha, T} )^{\kappa}.
  \end{equation*}
  For $\epsilon \in (\alpha, 1)$, referring to \eqref{eq:estimate_of_Y^2_by_sobolev}, we observe
  \begin{equation*}
    \norm{e^{(t-s)A} Z_s \abs{Y_s}^2}_{\B^{\gamma}_p} \lesssim
    (t-s)^{-\frac{2+\alpha+\gamma}{2} + \frac{1}{p}} \norm{Z_s}_{\contisp^{-\alpha}}
    \left( \norm{Y_s^2}_{L^1}^{1-\frac{\epsilon}{2}} \norm{\nabla{Y_s}}_{L^2}^{\epsilon} + \norm{Y_s^2}_{L^1} \right).
  \end{equation*}
  For sufficiently small $\epsilon$ (hence $\alpha$ have to be small), H\"older's inequality, \eqref{eq:L^p_a_priori_estimate}
  and Corollary \ref{cor:strong_L^p_estimate_independent_of_initial} imply
  \begin{equation*}
    \sup_{t_0 \leq t \leq t_1} \int_{t_0}^{t} (t-s)^{-\frac{2+\alpha+\gamma}{2} + \frac{1}{p}} \norm{\nabla Y_s}_{L^2}^{\epsilon} ds
    \lesssim (t_0^{-1} +  \znorm{\underline{Z}}_{\alpha, T})^{\kappa_1}
  \end{equation*}
  for some $\kappa_1 \in [1, \infty)$.
  The other terms in the integrand can be similarly evaluated.
  The term $e^{(t-t_0)A} Y_{t_0}$ can be handled as in the proof of Lemma \ref{lem:a_priori_estimate_1}.
\end{proof}

\begin{lemma}\label{lem:a_priori_estimate_3}
  Let $0 < t_0 < t_1 \leq T$. Assume that there exist $C \in [1, \infty)$, $p \in (2, \infty)$ and $\gamma \in (0, 2)$ such that
  for every $Y \in S^{t_1}(Y_0, \underline{Z})$,
  \begin{equation*}
     \int_{\frac{t_0}{2}}^{t_1} \norm{Y_t}_{\B^{\gamma}_{3p}}^3 dt + \sup_{\frac{t_0}{2} \leq t \leq t_1} \norm{Y_t}_{\B^{\gamma}_p} \leq C.
  \end{equation*}
  In addition, assume that $p, \alpha, \gamma$ and $\epsilon \in (0, \frac{5}{3} - \frac{10}{3p})$ satisfy
  \begin{equation}\label{eq:assumption_on_p_alpha_gamma_epsilon}
    \frac{\alpha + \gamma + \epsilon}{2} <   \frac{5}{6} - \frac{2}{3p} .
  \end{equation}
  Then there exists $\kappa \in (0, \infty)$ such that for $Y \in S^{t_1}(Y_0, \underline{Z})$,
  \begin{equation*}
    \int_{t_0}^{t_1} \norm{Y_t}^3_{\B^{\gamma + \epsilon}_{3p}} dt + \sup_{t_0 \leq t \leq t_1} \norm{Y_t}_{\B_p^{\gamma + \epsilon}}
    \lesssim_{T, p, \alpha, \gamma, \epsilon} (C + t_0^{-1} +
     \znorm{\underline{Z}}_{\alpha, T})^{\kappa}.
  \end{equation*}
\end{lemma}
\begin{proof}
  Again we evaluate each term of the right hand side of \eqref{eq:Y_is_a_mild_sol_from_t_0}.
  We begin with the estimate of the integral. By Proposition \ref{prop:besov_embedding} and Proposition \ref{prop:smoothing_in_besov},
  \begin{align*}
    \norm{e^{(t-s)A} (\abs{Y_s}^2 Y_s)}_{\B^{\gamma + \epsilon}_{3p}}
    &\lesssim \norm{e^{(t-s)A} (\abs{Y_s}^2 Y_s)}_{\B^{\gamma+ \epsilon + \frac{10}{3p}}_{p/2}} \\
    &\lesssim (t-s)^{-\frac{\epsilon}{2} - \frac{5}{3p}} \norm{\abs{Y_s}^2 Y_s}_{\B^{\gamma}_{p/2}} \\
    &\lesssim (t-s)^{-\frac{\epsilon}{2} - \frac{5}{3p}} \norm{Y_s}_{\B^{\gamma}_p} \norm{Y_s^2}_{\B^{\gamma}_p}.
  \end{align*}
  Corollary \ref{cor:multiplicative_inequality} and Proposition \ref{prop:interpolation} imply
  \begin{equation}\label{eq:estimate_of_Y_s^2}
    \norm{Y_s^2}_{\B^{\gamma}_p} \lesssim \norm{Y_s}_{\B^{\gamma}_{2p}}^2
    \leq \norm{Y_s}_{\B^{\gamma}_{3p}}^{\frac{3}{2}} \norm{Y_s}_{\B^{\gamma}_{p}}^{\frac{1}{2}}
    \leq C^{\frac{1}{2}} \norm{\abs{Y_s}}_{\B^{\gamma}_{3p}}^{\frac{3}{2}}.
  \end{equation}
  Therefore,
  \begin{equation*}
    \norm{e^{(t-s)A} (\abs{Y_s}^2 Y_s)}_{\B^{\gamma + \epsilon}_{3p}}
    \lesssim (t-s)^{-\frac{\epsilon}{2} - \frac{5}{3p}} C^{\frac{3}{2}} \norm{Y_s}_{\B^{\gamma}_{3p}}^{\frac{3}{2}}.
  \end{equation*}
  Noting $1 + \frac{1}{3} = \frac{5}{6} + \frac{1}{2}$, Young's convolution inequality yields
  \begin{align*}
    \int_{t_0}^{t_1} &\left( \int_{t_0}^t \norm{e^{(t-s)A} (\abs{Y_s}^2 Y_s)}_{\B^{\gamma+\epsilon}_{3p}} ds \right)^3 dt \\
    &\lesssim C^\frac{9}{2} \left( \int_{t_0}^{t_1} (t_1 - s)^{-\frac{6}{5}( \frac{\epsilon}{2} + \frac{5}{3p})} ds \right)^{\frac{5}{2}}
    \left(\int_{t_0}^{t_1} \norm{Y_s}^3_{\B^{\gamma}_{3p}} ds \right)^{\frac{3}{2}}
    \lesssim C^6.
  \end{align*}
  Here we used the fact that $\epsilon < \frac{5}{3} - \frac{10}{3p}$.
  For the term $Z_s \abs{Y_s}^2$,
  \begin{align*}
    \norm{e^{(t-s)A} (Z_s \abs{Y_s}^2)}_{\B^{\gamma+\epsilon}_{3p}}
    &\lesssim \norm{e^{(t-s)A} (Z_s \abs{Y_s}^2)}_{\B^{\gamma + \epsilon + \frac{4}{3p}}_{p}}\\
    &\lesssim (t-s)^{-\frac{\alpha+\gamma+\epsilon}{2} - \frac{2}{3p}} \norm{Z_s}_{\contisp^{-\alpha}} \norm{\abs{Y_s}^2}_{\B^{\gamma}_p}.
  \end{align*}
  Young's convolution inequality and \eqref{eq:estimate_of_Y_s^2}  yield
  \begin{align*}
    \MoveEqLeft
    \int_{t_0}^{t_1}\left( \int_{t_0}^t \norm{e^{(t-s)A} (Z_s \abs{Y_s}^2)}_{\B^{\gamma + \epsilon}_{3p}} ds \right)^3 dt \\
    &\lesssim \left( \int_{t_0}^{t_1} (t_1 - s)^{-\frac{6}{5}(\frac{\alpha+\gamma + \epsilon}{2} + \frac{2}{3p})} ds \right)^{\frac{5}{2}}
    \left( \int_{t_0}^{t_1} C \norm{\underline{Z}_s}_{\zset_{\alpha}}^2 \norm{Y_s}_{\B^{\gamma}_{3p}}^3 ds \right)^{\frac{3}{2}}\\
    &\lesssim C^3  \znorm{\underline{Z}}_{\alpha, T}^3.
  \end{align*}
  Here we used the assumption \eqref{eq:assumption_on_p_alpha_gamma_epsilon}.
  We continue to estimate
  \begin{align*}
    \norm{e^{(t-s)A} (\abs{Z_s}^2 Y_s)}_{\B^{\gamma + \epsilon}_{3p}}
    &\lesssim \norm{e^{(t-s)A} (\abs{Z_s}^2 Y_s)}_{\B_p^{\gamma + \epsilon + \frac{4}{3p}}} \\
    &\lesssim (t-s)^{-\frac{\alpha + \gamma + \epsilon}{2} - \frac{2}{3p}} \norm{\abs{Z_s}^2}_{\contisp^{-\alpha}} \norm{Y_s}_{\B^{\gamma}_p} \\
    &\lesssim C (t-s)^{-\frac{\alpha+\gamma+\epsilon}{2} - \frac{2}{3p}} s^{-\alpha} \norm{\underline{Z}_s}_{\zset_{\alpha}} .
  \end{align*}
  and
  \begin{multline*}
    \int_{t_0}^{t_1} \left( \int_{t_0}^t \norm{e^{(t-s)A} (\abs{Z_s}^2 Y_s)}_{\B^{\gamma + \epsilon}_{3p}} ds \right)^3 dt \\
    \lesssim \left( \int_{t_0}^{t_1} (t_1 - s)^{-\frac{\alpha + \gamma + \epsilon}{2} - \frac{2}{3p}} ds \right)^3
    \int_{t_0}^{t_1} \left(C s^{-\alpha} \norm{\underline{Z}_s}_{\zset_{\alpha}}  \right)^3 ds
    \lesssim C^3  \znorm{\underline{Z}}_{\alpha, T}^3.
  \end{multline*}
  Estimates on the other terms in the integrand are similar. The term $e^{(t-t_0)A}Y_{t_0}$ can be estimated
  as in the proof of Lemma \ref{lem:a_priori_estimate_1}.

  We next evaluate $\sup$-norm of $Y_t$. Again the strategy is identical to Lemma \ref{lem:a_priori_estimate_2}.
  Since
  \begin{equation*}
    \norm{e^{(t-s)A} (\abs{Y_s}^2 Y_s)}_{\B^{\gamma + \epsilon}_p} \lesssim \norm{Y_s}_{\B^{\gamma + \epsilon}_{3p}}^3,
  \end{equation*}
  the first part of the proof shows
  \begin{equation*}
    \sup_{t_0 \leq t \leq t_1} \int_{t_0}^t \norm{e^{(t-s)A} (\abs{Y_s}^2 Y_s)}_{\B^{\gamma + \epsilon}_p} ds
    \lesssim (C + t_0^{-1} + \norm{\underline{Z}}_{\alpha, T})^{\kappa}.
  \end{equation*}
  Next, we have
  \begin{equation*}
    \norm{e^{(t-s)A} (Z_s \abs{Y_s}^2)}_{\B^{\gamma + \epsilon}_p}
    \lesssim (t-s)^{-\frac{\alpha + \gamma + \epsilon}{2} - \frac{1}{p}} \norm{\underline{Z}_s}_{\zset_{\alpha}} \norm{Y_s}_{\B^{\gamma}_p}^2.
  \end{equation*}
  We note that $\frac{\alpha+\gamma+\epsilon}{2} + \frac{1}{p} < 1$ since $\frac{5}{6} - \frac{2}{3p} < 1 - \frac{1}{p}$.
  We have
  \begin{equation*}
    \norm{e^{(t-s)A} (\abs{Z_s}^2 Y_s)}_{\B^{\gamma + \epsilon}_p}
    \lesssim (t-s)^{-\frac{\alpha + \gamma + \epsilon}{2}} s^{-\alpha} \norm{\underline{Z}_s}_{\zset_{\alpha}} \norm{Y_s}_{\B^{\gamma}_p}
  \end{equation*}
  and the other terms in the integrand can be similarly estimated. The term $e^{(t-t_0)A}Y_{t_0}$ can be handled as in
  the proof of Lemma \ref{lem:a_priori_estimate_1}.
\end{proof}

\begin{theorem}\label{thm:a_priori_estimate_4}
  Let $\beta \in (0, 2)$.
  Then for all sufficiently small $\alpha \in (0, 1)$, there exists $\kappa  = \kappa(\beta) \in (0, \infty)$ such that
  for $0 < t_0 < t_1 \leq T$ and $Y \in S^{t_1}(Y_0, \underline{Z})$,
  \begin{equation*}
    \sup_{t_0 \leq t \leq t_1} \norm{Y_t}_{\contisp^{\beta}} \lesssim_{\beta, \alpha, T} (t_0^{-1} +
    \znorm{\underline{Z}}_{\alpha, T})^{\kappa}.
  \end{equation*}
\end{theorem}
\begin{proof}
  Set $B(t_0) \defby t_0^{-1} +  \znorm{\underline{Z}}_{\alpha, T}$.
  According to Lemma \ref{lem:a_priori_estimate_1} and Lemma \ref{lem:a_priori_estimate_2},
  we can find $p \in (2, \infty)$, $\epsilon \in (0, 1)$ and $\kappa_1 \in (0, \infty)$ such that
  \begin{equation*}
    \int_{t_0}^{t_1} \norm{Y_t}_{\B^{\epsilon}_{3p}}^3 dt + \sup_{t_0 \leq t \leq t_1}
    \norm{Y_t}_{\B^{\epsilon}_p} \lesssim B(t_0)^{\kappa_1}.
  \end{equation*}
  Then, using Lemma \ref{lem:a_priori_estimate_3} repeatedly, we can find
  \begin{equation*}
     \gamma > \frac{5}{3} - \frac{4}{3p} - 2\alpha
  \end{equation*}
  and $\kappa_2 \in (0, \infty)$ such that
  \begin{equation*}
    \sup_{t_0 \leq t \leq t_1} \norm{Y_t}_{\B^{\gamma}_p} \lesssim B(t_0)^{\kappa_2}.
  \end{equation*}
  Assuming $\alpha$ is sufficiently small, we can suppose $\gamma > 1$.
  Then Proposition \ref{prop:besov_embedding} implies
  \begin{equation*}
    \sup_{t_0 \leq t \leq t_1} \norm{Y_t}_{\contisp^{\delta}} \lesssim
    \sup_{t_0 \leq t \leq t_1} \norm{Y_t}_{\B^{\delta + \frac{2}{p}}_p} \lesssim B(t_0)^{\kappa_2}
  \end{equation*}
  provided $\delta + \frac{2}{p} < \gamma$.
  Now recall that we have
  \begin{equation*}
    Y_t = e^{(t-\frac{t_0}{2})A} Y_{\frac{t_0}{2}} + \int_{\frac{t_0}{2}}^t e^{(t-s)A} \Psi(Y_s, \underline{Z}_s) ds.
  \end{equation*}
  Proposition \ref{prop:besov_embedding} and Corollary \ref{cor:strong_L^p_estimate_independent_of_initial} imply
  \begin{equation*}
  \sup_{t_0 \leq t \leq t_1} \norm{e^{(t-\frac{t_0}{2})A} Y_{\frac{t_0}{2}}}_{\contisp^{\beta}} \lesssim B(t_0)^{\kappa_3}
  \end{equation*}
  for some $\kappa_3 \in (0, \infty)$. If $\alpha < \delta$, we have
  $\norm{\Psi(Y_s, \underline{Z}_s)}_{\contisp^{-\alpha}} \lesssim B(t_0)^{\kappa_4} s^{-\alpha}$ for
  $s \in [\frac{t_0}{2}, t_1]$, and hence
  \begin{equation*}
    \int_{\frac{t_0}{2}}^t \norm{e^{(t-s)A} \Psi(Y_s, \underline{Z}_s)}_{\contisp^{\beta}} ds
    \lesssim t^{1 - \frac{\beta + 2\alpha}{2}} B(t_0)^{\kappa_4}. \qedhere
  \end{equation*}
\end{proof}

\begin{proof}[Proof of Theorem \ref{thm:global_wellposedness}]
  Fix $T \in (0, \infty)$. Set
  \begin{equation*}
    K \defby  \znorm{\underline{Z}}_{\alpha, T}
    + \sup \set{ \sup_{0 \leq t \leq t_1} \norm{Y_t}_{\contisp^{-\alpha_0}} \given t_1 \in (0, T], Y \in S^{t_1}(Y_0,
    \underline{Z})}.
  \end{equation*}
  On the one hand, by Theorem \ref{thm:local_wellposedness}, there exists $T^* \in (0, T]$ such that
  $\# S^{T^*}(Y_0, \underline{Z}) = 1$ and for $Y \in S^{T^*}(Y_0, \underline{Z})$,
  \begin{equation*}
    \sup_{0 \leq t \leq T^*} \norm{Y_t}_{\contisp^{-\alpha_0}} < \infty.
  \end{equation*}
  On the other hand, by Theorem \ref{thm:a_priori_estimate_4},
  \begin{equation*}
    \sup \set{ \sup_{T^* \leq t \leq t_1} \norm{Y_t}_{\contisp^{-\alpha_0}} \given t_1 \in [T^*, T], Y \in S^{t_1}(Y_0,
    \underline{Z})} < \infty.
  \end{equation*}
  Consequently, $K < \infty$.
  Therefore, we can repeatedly use the fixed point argument in the proof of Theorem \ref{thm:local_wellposedness}
  to construct a global solution over $[0, T]$.
  Uniqueness can be shown as in \emph{STEP 2} in the proof of Theorem \ref{thm:local_wellposedness}.
\end{proof}
\begin{remark}\label{rem:SCGL_in_the_plane}
  We believe that the method of \cite{MW2dim} enables us to construct a solution of the two-dimensional SCGL in the plane.
  However, there is a technical obstacle. Indeed, as in \cite{MW2dim}, in order to construct a solution in the plane,
  we have to carry out our analysis in weighted Besov spaces $\hat{\B}_{p, q}^{\alpha, \sigma}$
  (see Definition \ref{def:weighted_Besov}).
  Then, we have to change weights $\sigma$ when applying Besov embeddings for weighted Besov spaces
  (\cite[\largeprop 2]{MW2dim}). Since Besov embeddings (Proposition \ref{prop:besov_embedding}-(ii)) played a crucial role
  in our approach to the two-dimensional SCGL on the torus, the combination of the method of \cite{MW2dim} and our approach
  yields construction of a solution in the plane taking values in $\hat{\B}_{p, \infty}^{-\alpha, \sigma}$ only for large
  $\sigma \in (2, \infty)$.
\end{remark}

\subsection{Continuity with respect to input data and coming down from infinity}\label{subsec:continuity_and_coming_down}
We begin to prove that the solution of \eqref{eq:pde_of_Y} depends continuously on the input data.
\begin{proposition}\label{prop:continuity_parameters}
  Let $\alpha_0 \in (0, \frac{2}{3})$, $\beta \in (0, 2)$, $T \in (0, \infty)$ and $R \in [1, \infty)$.
  Assume $\alpha \in (0, 1)$ is so small that Theorem \ref{thm:a_priori_estimate_4} holds.
  Suppose that $\alpha_1 \in (0, \infty)$ and $\gamma \in (0, \frac{1}{3})$ satisfy \eqref{eq:condition_of_alpha_and_gamma}
  and that $Y_0^{(1)}, Y_0^{(2)} \in \contisp^{-\alpha_0}$ and $\underline{Z}^{(1)}, \underline{Z}^{(2)} \in \zset$
  satisfy
  \begin{equation*}
    \norm{Y_0^{(1)}}_{\contisp^{-\alpha_0}} + \norm{Y_0^{(2)}}_{\contisp^{-\alpha_0}}
    + \sup_{0 < t \leq T} \norm{\underline{Z}_t^{(1)}}_{\zset_{\alpha}} +
    \sup_{0 < t \leq T} \norm{\underline{Z}_t^{(2)}}_{\zset_{\alpha}} \leq R.
  \end{equation*}
  Let $Y^{(i)} \in S^T(Y_0^{(i)}, \underline{Z}^{(i)})$.

  Then there exist positive constants $\kappa = \kappa(\alpha_0)$, $\kappa' = \kappa'(\alpha_0,
  \alpha_1, \gamma)$ and
  $\kappa'' = \kappa''(\beta)$ such that
  \begin{align*}
    &\sup_{0 \leq t \leq T}  \norm{Y_t^{(1)}- Y_t^{(2)}}_{\contisp^{-\alpha_0}}
    \lesssim_{\alpha_0, \alpha, T}
    R^{\kappa} D, \\
    &\sup_{0 < t \leq T} t^{\gamma} \norm{Y_t^{(1)}- Y_t^{(2)}}_{\contisp^{\alpha_1}} \lesssim_{\alpha_0, \alpha_1, \alpha, \gamma, T}
    R^{\kappa'} D, \\
    &\sup_{t_0 < t \leq T}  \norm{Y_t^{(1)}- Y_t^{(2)}}_{\contisp^{\beta}} \lesssim_{\alpha_1, \alpha, \gamma, \beta, t_0, T}
    R^{\kappa''} D,
  \end{align*}
  where $D \defby \norm{Y_0^{(1)} - Y_0^{(2)}}_{\contisp^{-\alpha_0}} + \sup_{0 \leq t \leq T} \norm{\underline{Z}_t^{(1)}
  - \underline{Z}_t^{(2)}}_{\zset_{\alpha}}$.
\end{proposition}
\begin{proof}
  We only prove the bound on $\contisp^{\alpha_1}$-norm.
  The other bounds can be proved similarly.
  By Remark \ref{remark:form_of_T_star} and Theorem \ref{thm:a_priori_estimate_4}, we have
  \begin{equation*}
    \sup_{0 < t \leq T} t^{\gamma} \norm{Y_t}_{\contisp^{\alpha_1}} \lesssim R^{\kappa_1}.
  \end{equation*}
  Therefore,
  \begin{multline*}
    \norm{\Psi(Y_s^{(1)}, \underline{Z}_s^{(1)}) - \Psi(Y_s^{(2)}, \underline{Z}_s^{(2)})}_{\contisp^{\alpha_1}}
    \\ \lesssim R^{\kappa_2} s^{- 3 \gamma} \left( s^{\gamma} \norm{Y_s^{(1)} - Y_s^{(2)}}_{\contisp^{\alpha_1}} +
    \norm{\underline{Z}_s^{(1)} - \underline{Z}_s^{(2)}}_{\zset_{\alpha}}  \right).
  \end{multline*}
  By applying the above estimate to
  \begin{equation*}
    Y^{(1)}_t - Y^{(2)}_t = e^{tA}(Y^{(1)}_0 - Y^{(2)}_0)
    + \int_0^t e^{(t-s)A}\{ \Psi(Y_s^{(1)}, \underline{Z}_s^{(1)}) - \Psi(Y_s^{(2)}, \underline{Z}_s^{(2)}) \} ds,
  \end{equation*}
  we obtain
  \begin{multline*}
    t^{\gamma} \norm{Y_t^{(1)} - Y_t^{(2)}}_{\contisp^{\alpha_1}}
    \leq c_1  \norm{Y_0^{(1)} - Y_0^{(2)}}_{\contisp^{-\alpha_0}} \\
    + c_2  R^{\kappa_2} \sup_{0<s \leq t} \norm{\underline{Z}_s^{(1)} - \underline{Z}_s^{(2)}}_{\zset_{\alpha}}
    + c_3 t^{\gamma'} R^{\kappa_3} \sup_{0 < s \leq t} s^{\gamma} \norm{Y_s^{(1)} - Y_s^{(2)}}_{\contisp^{\alpha_1}}.
  \end{multline*}
  Setting $n \defby \lfloor (2 c_3 R^{\kappa_3})^{\frac{1}{\gamma'}} T \rfloor + 1$, we obtain
  \begin{equation*}
    \sup_{0 < t \leq T/n} t^{\gamma} \norm{Y_t^{(1)} - Y_t^{(2)}}_{\contisp^{\alpha_1}}
    \lesssim \norm{Y_0^{(1)} - Y_0^{(2)}} + R^{\kappa_2}
    \sup_{0 < t \leq T} \norm{\underline{Z}_t^{(1)} - \underline{Z}_t^{(2)}}_{\zset_{\alpha}}.
  \end{equation*}
  Repeating this process on $[\frac{T}{n}, \frac{2T}{n}], \ldots, [\frac{(n-1)T}{n}, T]$, we obtain the desired inequality.
\end{proof}

From now on, we derive some results of solutions of the SCGL \eqref{eq:CGL}
from the results obtained before. Therefore, we work in a probabilistic framework.
\begin{definition}\label{def:def_of_CGL}
  Let $\alpha_0 \in (0, \frac{2}{3})$, $u_0 \in \contisp^{-\alpha_0}$ and
  \begin{equation*}
    \underline{Z}_t \defby (Z^{:1,0:}(0,t), Z^{:2,0:}(0,t), Z^{:1,1:}(0,t), Z^{:2, 1:}(0,t)),
  \end{equation*}
  where $Z^{:k,l:}(0,t)$ is the nonstationary Ornstein-Uhlenbeck process constructed in
  \ref{subsection:nonstationary_OU}. Let $Y \in S^T(u_0, \underline{Z})$ for every $T \in (0, \infty)$.
  We call $u(t; u_0) \defby Z^{:1,0:}(0,t) + Y_t$ the solution of \eqref{eq:CGL}.
  We simply write $u(t)$ when the initial value is not important.
\end{definition}
By Proposition \ref{prop:continuity_parameters}, we can prove a slight generalization of Theorem \ref{thm:convergence_of_smooth_cgl}.
\begin{corollary}\label{cor:convergence_of_smoothed_solution}
  Assume
  \begin{enumerate}[(i)]
    \item $\rho \in \S(\R \times \R^2)$ and $\rho_{\delta}(t, x) \defby {\delta}^{-4}\rho(\frac{t}{\delta^2}, \frac{x}{\delta})$ and
    \item $u_{0;n} \to u_0$ in $\contisp^{-\alpha_0}$.
  \end{enumerate}
  Set $\xi_{\delta} \defby \xi_1 * \rho_{\delta}$ and let $u_{\delta; n}$ be the solution of
  \begin{equation*}
    \left \{ \begin{aligned}
    &\partial_t u_{\delta; n} = (i + \mu) \Delta u_{\delta; n} - \nu \left(\abs{u_{\delta; n}}^2 u_{\delta;n}
    - 2 c_{1; \delta} u_{\delta;n} \right) + \lambda u_{\delta; n} + \xi_{\delta}, \\
    &u_{\delta; n}(0, \cdot) = u_{0;n},
  \end{aligned} \right.
  \end{equation*}
  where $c_{1; \delta}$ is the renormalization constant given in \eqref{eq:def_of_renormalized_constant}.
  Then, for every $p \in [1, \infty)$, $T \in (0, \infty)$, $\alpha \in (0, \infty)$ and $\gamma \in (\frac{\alpha_0}{2}, \infty)$,
  \begin{align*}
    &\lim_{\delta \to 0,\, n \to \infty} \expect[ \sup_{0 \leq t \leq T} \norm{u(t;u_0) - u_{\delta; n}(t)}_{\contisp^{-\alpha_0}}^p ] = 0, \\
    &\lim_{\delta \to 0, \, n \to \infty} \expect[ \sup_{0 < t \leq T} t^{p\gamma}\norm{u(t;u_0) - u_{\delta; n}(t)}_{\contisp^{-\alpha}}^p ] = 0.
  \end{align*}
\end{corollary}
\begin{proof}
  Let $Z^{:k,l:}(0, t) \defby Z_M^{:k,l:}(0, t)$ be the nonstationary Ornstein-Uhlenbeck process constructed in
  \ref{subsection:nonstationary_OU} and
  \begin{equation*}
    Z^{(\delta)}(0, t) \defby \int_0^t e^{(t-s)A} \xi_{\delta}(s) ds.
  \end{equation*}
  Set
  \begin{align*}
    \underline{Z}(0, t) &\defby (Z(0,t), Z^{:2,0:}(0,t), Z^{:1,1:}(0,t), Z^{:2,1:}(0, t)), \\
    \underline{Z}^{(\delta)}(0, t) &\defby (H_{k,l}(Z^{(\delta)}(t), c_{M;\delta}))_{(k,l) =
    (1,0), (2,0), (1,1), (2,1)}.
  \end{align*}
  Theorem \ref{thm:convergence_of_mollified_processes} implies that
  \begin{equation*}
    \lim_{\delta \to 0} \expect\left[ \sup_{0 < t \leq T} \norm{\underline{Z}(0,t) -
    \underline{Z}^{(\delta)}(0, t)}_{\zset_{\alpha'}}^p \right] = 0.
  \end{equation*}
  Therefore, the claim follows from Proposition \ref{prop:continuity_parameters}.
\end{proof}
\begin{remark}\label{remark:convergence_of_smoothed_solution}
  A similar result holds for approximations which mollify the noise $\xi$
  only with respect to the spatial variable. See Remark \ref{remark:convergence_of_mollified_processes}.
\end{remark}
Finally, we prove an estimate for the solution of \eqref{eq:CGL} which is uniform with respect to the initial condition
(Theorem \ref{thm:coming_down_from_infinity}).
This surprising estimate, called ``coming down from infinity" in \cite{MW3dim}, is due to the damping
from the nonlinear term.
\begin{lemma}\label{lem:solution_started_at_t}
  Let $Z = Z_M$ be the Ornstein-Uhlenbeck process and set
  \begin{equation*}
    \underline{Z}(s, t) \defby (Z(s,t), Z^{:2, 0:}(s, t), Z^{:1, 1:}(s, t), Z^{:2, 1:}(s, t)).
  \end{equation*}
  Let $u = Z + Y$ be the solution of \eqref{eq:CGL}.
  Then
  \begin{equation*}
    u(t+\cdot) - Z(t, t+\cdot) \in S^T(u(t), \underline{Z}(t, t+\cdot)) \quad \mbox{for every $T \in (0, \infty)$.}
  \end{equation*}
\end{lemma}
\begin{proof}
  Set $Y_{t, t+h} \defby e^{hA}Z(t) + Y_{t+h}$ so that $u(t+h) = Z(t, t+h) + Y_{t, t+h}$.
  We have
  \begin{equation*}
    Y_{t, t+h} = e^{hA} u(t) + \int_0^h e^{(h-r)A} \Psi(Y_{t+r}, \underline{Z}_{t+r}) dr.
  \end{equation*}
  Note that $\Psi(Y_{t+r}, \underline{Z}_{t+r})
  = -\nu(Y_{t+r} + Z(t+r))^{:2,1:} + (\lambda+1)(Y_{t+r} + Z(t+r))$ and
  \begin{equation*}
    Y_{t+r} + Z(t + r) = Y_{t+r} + Z(t, t+r) + e^{rA} Z(t) = Y_{t, t+r} + Z(t, t+r). \qedhere
  \end{equation*}
\end{proof}
\begin{proof}[Proof of Theorem \ref{thm:coming_down_from_infinity}]
  Set $\tilde{Y}(t- \frac{t_0}{2}, t) \defby u(t; u_0) - Z(t-\frac{t_0}{2}, t)$.
  Theorem \ref{thm:a_priori_estimate_4} and Lemma \ref{lem:solution_started_at_t}
  imply that
  \begin{equation*}
    \norm{\tilde{Y}(t-\frac{t_0}{2}, t)}_{\contisp^{-\alpha}}
    \lesssim (t_0^{-1} + \sup_{t-\frac{t_0}{2} < s \leq t} \norm{\underline{Z}(t-\frac{t_0}{2}, s)}_{\zset_{\alpha}})^{\kappa}.
  \end{equation*}
  Since $\set{\underline{Z}(t-\frac{t_0}{2}, s) \given s \in (t-\frac{t_0}{2}, t]}$ has the same law as
  $\set{\underline{Z}(0, s) \given s \in (0, \frac{t_0}{2}]}$, it remains to apply Theorem \ref{thm:main_estimate_of_OU}.
\end{proof}

%% file: strong_Feller_property.tex
\section{Strong Feller property}\label{sec:strong_Feller}
In this section, we again work on the fixed torus $\torus_M$.
As the actual value of $M$ is not important, we set $M = 1$.
We fix $\alpha_0 \in (0, \frac{2}{3})$ and $\alpha_1, \gamma$ satisfying \eqref{eq:condition_of_alpha_1_and_gamma}.
We fix small $\alpha \in (0, 1)$ and hence we write $\znorm{\underline{Z}}_T \defby \znorm{\underline{Z}}_{\alpha, T}$,
see \eqref{eq:z_norm_alpha_T}.
The aim of this section is to prove Theorem \ref{thm:strong_Feller},
the proof of which is given at the end of Subsection \ref{subsec:holder_continuity}.
This section follows \cite[Section 5]{TW18}.

\subsection{Markov property}\label{subsec:Markov_property}
\begin{proposition}\label{prop:markov_process}
  Let $u = u(\cdot; u_0)$ be the solution of \eqref{eq:CGL}. Then,
  $u$ is a Markov process on $\contisp^{-\alpha_0}$ with filtration $\{\F_t\}$ given in Definition
  \ref{def:definition_of_white_noise_and_filtration}-(ii).
  Furthermore, the Markov process $u$ is Feller.
\end{proposition}
\begin{proof}
  Let $t \in [0, \infty)$ and $h \in (0, \infty)$.
  Set $Y(t, t+h) \defby u(t+h) - Z(t, t+h)$.
  By Lemma \ref{lem:solution_started_at_t}, we have $Y(t, t+\cdot) \in S^T(u(t), \underline{Z}(t, t+\cdot))$.
  Since $ \underline{Z}(t, t+\cdot) $ is independent of $\F_t$ by Proposition \ref{prop:independence_of_Z_M_s_t} and
  the solution of the shifted equation \eqref{eq:pde_of_Y} is measurable with respect to the initial value
  and the driver by Proposition \ref{prop:continuity_parameters}, we see that for any bounded measurable function
  $\Phi: \contisp^{-\alpha_0} \to \R$,
  \begin{equation*}
    \expect[ \Phi(u(t+h)) \vert \F_t ] = \expect[\Phi(u(h); v)]\vert_{v = u(t)}.
  \end{equation*}
  Therefore, $u$ is a Markov process.

  To prove that $u$ is Feller, take $\Phi \in C_b(\contisp^{-\alpha_0}; \R)$ and $t \in (0, \infty)$.
  We need to show that the map
  \begin{equation*}
    \contisp^{-\alpha_0} \ni v \mapsto \expect[ \Phi(u(t); v)] \in \R
  \end{equation*}
  is continuous, but this follows from Proposition \ref{prop:continuity_parameters} and
  Lebesgue's dominated convergence theorem.
\end{proof}

\subsection{Approximation by a system of SDEs}\label{subsec:approximation_by_SDEs}
In the next subsection, we derive the most crucial formula,
called Bismut-Elworthy -Li formula, to prove the strong Feller property of $u$.
This subsection serves as preparation for the setting.

Heuristically, the proof of Bismut-Elworthy-Li formula consists of the following three steps;
\begin{enumerate}[(i)]
  \item We perturb a parameter, in our case the noise, from $\xi$ to $\xi^{\delta} \defby \xi + \delta \partial_t w$
  with $\delta \in (0, 1)$ and $w$ a Cameron-Martin path.
  \item Let $u^{\delta}$ be the solution of
  \begin{equation*}
    \left \{
    \begin{aligned}
      &\partial_t u^{\delta} = (i + \mu)\Delta u^{\delta} - \nu \abs{u^{\delta}}^2 u^{\delta} + \lambda u^{\delta}
       + \xi^{\delta}, \\
      &u^{\delta}(0, \cdot) = u_0.
    \end{aligned} \right.
  \end{equation*}
  In the light of Girsanov's theorem, we construct a probability measure $\P^{\delta}$ under which $u^{\delta}$
  has the same law as the original solution $u$ has
  under $\P$.
  \item Then, we have $\left.\frac{d}{d\delta}\right\rvert_{\delta = 0} \expect^{\P^{\delta}}[ \Phi(u^{\delta}) ] = 0$.
  We compute the left hand side to obtain the formula.
\end{enumerate}

As in \cite{TW18}, however, we consider a finite dimensional approximation of the equation \eqref{eq:CGL}, in order to
avoid the language of Malliavin calculus.
More precisely, we will work on the vector space spanned by $\{ e_m \}_{\abs{m} \leq n}$.
However, the nonlinear operation $u \mapsto \abs{u}^2 u$ is not closed in this space.

Therefore, we introduce cutoff operators.
We take $\rho \in \S(\R^2)$ such that $\F \rho = \indic_{B(0, \frac{1}{2})}$
on $B(0, \frac{1}{2}) \cup (\R^2 \setminus B(0, 1))$.
Then set $\rho^{(n)}(x) \defby n^2 \rho(nx)$. Notice that $\F \rho^{(n)}$ has its support in $B(0, n)$.
We define an operator $\Pi_n: \S'(\torus) \to C^{\infty}(\torus)$ by
\begin{equation*}
  \Pi_n( \sum_{m \in \Z^2} a_m e_m ) \defby \sum_{m \in \Z^2} [\F \rho^{(n)}](m) a_m e_m.
\end{equation*}
\begin{lemma}\label{lem:properties_of_Pi}
  The operator $\Pi_n$ defined above satisfies the following:
  \begin{enumerate}[(i)]
    \item There exists a constant $C \in (0, \infty)$ such that
    $\norm{\Pi_n}_{\contisp^{\lambda} \to \contisp^{\lambda}} \leq C$ for every $n \in \N$ and $\lambda \in \R$.
    \item There exists a constant $C \in (0, \infty)$ such that
    $\norm{\Pi_n - \operatorname{id}}_{\contisp^{\lambda} \to \contisp^{\lambda - \delta}} \leq C 2^{-n \delta}$
    for every $n \in \N$, $\lambda \in \R$ and $\delta \in (0, 1)$.
  \end{enumerate}
\end{lemma}
\begin{proof}
  Recall the notations from Subsection \ref{subsection:notation_for_Besov}.
  Let $f \in \contisp^{\lambda}$. Since $\Pi_n$ is a Fourier multiplier, we have
  $\delta_k \Pi_n f = \Pi_n \delta_k f$. Therefore
  \begin{equation*}
    \norm{\Pi_n f}_{\contisp^{\lambda}} = \sup_{k \geq -1} 2^{k \lambda} \norm{\delta_k \Pi_n f}_{L^{\infty}} \leq
    \norm{\Pi_n}_{L^{\infty} \to L^{\infty}} \norm{f}_{\contisp^{\lambda}}.
  \end{equation*}
  As $\norm{\Pi_n}_{L^{\infty} \to L^{\infty}} \leq \norm{\rho^{(n)}}_{L^1(\R^2)} = \norm{\rho}_{L^1(\R^2)} < \infty$
  by Young's convolution inequality, we end the proof of (i).

  To prove (ii), we observe that there exists a constant $n_0 \in \N$ such that $k + n_0 \leq n$ implies
  $\delta_k (\Pi_n - \operatorname{id}) = 0$. Therefore
  \begin{equation*}
    \norm{\delta_k(\Pi_n f - f)}_{L^{\infty}} \lesssim \norm{\delta_k f}_{L^{\infty}} \indic_{\{k>n-n_0\}}.
  \end{equation*}
  Thus, for $f \in \contisp^{\lambda}$,
  \begin{equation*}
  \norm{\Pi_n f - f}_{\contisp^{\lambda - \delta}}
  = \sup_{k \geq -1} 2^{k (\lambda - \delta)} \norm{\delta_k (\Pi_n f - f)}_{L^{\infty}}
  \lesssim 2^{-(n-n_0)\delta} \norm{f}_{\contisp^{\lambda}}. \qedhere
  \end{equation*}
\end{proof}

Heuristically, we approximate the noise $\xi$ by
\begin{equation*}
  \xi^{(n)}(t, x) = \frac{\partial}{\partial t} \sum_{\abs{m} \leq n} a_m^{(n)} W_m(t) e_m(x),
\end{equation*}
where $\{W_m\}_{m \in \Z^2}$ are i.i.d. complex Brownian motions, i.e.
$W_m(t) = \beta_{m, 1}(t) + i \beta_{m, 2}(t)$ where $\beta_{m, 1}$ and $\beta_{m, 2}$ are independent $\R$-valued
Brownian motions with
\begin{equation*}
  \expect[ \beta_1(t)^2 ] = \expect[ \beta_2(t)^2 ] = \frac{t}{2}.
\end{equation*}
We consider a system of SDEs
\begin{equation}\label{eq:approximate_sdes}
  \left \{
  \begin{aligned}
    &d u^{(n)} = \{ (i+\mu)\Delta u^{(n)} - \nu \Pi_n [H_{2, 1}(u^{(n)}, c_n)] + \lambda u^{(n)} \} dt  \\
    &\hspace{18em}+\sum_{\abs{m} \leq n} a_m^{(n)} e_m dW_m(t), \\
    &u^{(n)}(0, \cdot) = \Pi_n u_0,
  \end{aligned} \right.
\end{equation}
where $H_{2,1}$ is a complex Hermite polynomial and $c_n \in \R$ will be specified soon.

Now we check that this system of SDEs \eqref{eq:approximate_sdes} indeed approximates the original SPDE \eqref{eq:CGL}
for appropriately chosen $a_m^{(n)}$ and $c_n$.
Take $\rho$ and $\rho^{(n)}$ as above.
Recall that $\inp{\cdot}{\cdot}_1$ is the inner product of $L^2(\torus_1)$ and that
$\inp{\cdot}{\cdot}_{\infty}$ is the inner product of $L^2(\R^2)$.
We set $Z^{(n)}(0, t; x) \defby \inp{Z(0,t)}{\rho^{(n)}(x-\cdot)}_{\infty}$.
Then we have
\begin{equation*}
  Z^{(n)}(0, t) = \sum_{\abs{m} \leq n}[\F \rho^{(n)}](m) \inp{Z(0,t)}{e_{-m}}_1 e_m.
\end{equation*}
If we set
\begin{equation*}
  W_m(t) \defby \int \indic_{[0, t] \times \torus_1}(s, y) e_{-m}(y) \xi(ds dy)
\end{equation*}
and replace it by a continuous modification,
$\{W_m\}_{m \in \Z^2}$ are i.i.d. complex Brownian motions and we have almost surely
(see \eqref{eq:def_of_kerel_K_M} for the definition of $K_1$)
\begin{align*}
  \inp{Z(0, t)}{e_{-m}}_1 &= \int \indic_{[0,t] \times \torus_1}(s, y) \inp{K_1(t-s, \cdot - y)}{e_{-m}}_1 \xi(dsdy) \\
  &= \int \indic_{[0,t]\times \torus_1}(s, y) e_{-m}(y) e^{-(t-s)[(i+\mu)4\pi^2 \abs{m}^2 + 1]} \xi(dsdy) \\
  &=\int_0^t e^{-(t-s)[(i+\mu)4\pi^2 \abs{m}^2 + 1]} dW_m(s).
\end{align*}
We can obtain the last equality by approximating the function
\begin{equation*}
s \mapsto e^{-(t-s)[(i+\mu)4\pi^2 \abs{m}^2 + 1]}
\end{equation*}
by step functions. Therefore $Z^{(n)}(0, t)$ is the solution of
\begin{equation}\label{eq:approximate_Z}
  \left \{
  \begin{aligned}
    &\partial_t Z^{(n)}(0, t) = A Z^{(n)}(0, t) dt  + \sum_{\abs{m} \leq n} a^{(n)}_m e_m dW_m(t),\\
    &Z^{(n)}(0,0) = 0,
  \end{aligned} \right.
\end{equation}
where $a^{(n)}_m \defby [\F \rho^{(n)}](m)$.

Next we set (see \eqref{eq:def_of_renormalized_constant})
\begin{equation*}
  c_n \defby \int_{\R \times \torus_1} \abs{\inp{K_1(s, \cdot - y)}{\rho^{(n)}}_{\infty}}^2 ds dy
\end{equation*}
and
\begin{equation}\label{eq:def_of_approximate_underline_z}
  \underline{Z}^{(n)}(t) \defby
  ( H_{k,l}(Z^{(n)}(0,t), c_n) )_{(k,l) = (1, 0), (2, 0), (1, 1), (2, 1)}.
\end{equation}
Let $Y^{(n)}$ be the solution of
\begin{equation}\label{eq:pde_of_approximate_Y}
  \left \{
  \begin{aligned}
    &\partial_t Y^{(n)} = AY^{(n)} + \Pi_n\Psi(Y^{(n)}, \underline{Z}^{(n)}), \\
    &Y^{(n)}(0) = \Pi_n u_0,
  \end{aligned} \right.
\end{equation}
where $\Psi(Y^{(n)}, \underline{Z}^{(n)}) \defby H_{2, 1}(Y^{(n)} + Z^{(n)}(0, \cdot))$.
According to Lemma \ref{lem:global_wellposedness_of_approximate_Y}, the explosion time of $Y^{(n)}$ goes to infinity in probability.
Furthermore, Lemma \ref{lem:global_wellposedness_of_approximate_Y}
also shows that if we take $\tilde{Y}^{(n)} \in S^{T}(\Pi_n u_0, \underline{Z}^{(n)})$
(see Definition \ref{def:definition_of_Y} for $S^T$),
$\tilde{Y}^{(n)} - Y^{(n)}$ converges to $0$ in $C([0, T]; \contisp^{-\alpha_0})$.

Finally, we set $u^{(n)} \defby Z^{(n)}(0, \cdot) + Y^{(n)}$. Then $u^{(n)}$ solves \eqref{eq:approximate_sdes}.
Since $Z^{(n)}(0, t) + \tilde{Y}^{(n)}$ converges to the original solution $u$ of \eqref{eq:CGL}
by Corollary \ref{cor:convergence_of_smoothed_solution},
$u^{(n)}$ converges to $u$ as well.
We set $d(n) \defby \# \set{m \in \Z^2 \given \abs{m} \leq n}$.
\begin{lemma}\label{lem:global_wellposedness_of_approximate_Y}\leavevmode
  \begin{enumerate}[(i)]
    \item Let $T, R \in (0, \infty)$ and $(w_m)_{\abs{m} \leq n} \in C([0, T]; \C^{d(n)})$
    with $w(0) = 0$.
    We set
    \begin{equation*}
      Z^{(n)}(0, t; x) \defby \sum_{\abs{m} \leq n} a_m^{(n)} e_m(x) \int_0^t
      e^{-(t-s)[4\pi^2(i+\mu) \abs{m}^2 + 1]} dw_m(s)
    \end{equation*}
    where the integral is in the sense of Riemann-Stiltjes. We define $\underline{Z}^{(n)}$ by
    \eqref{eq:def_of_approximate_underline_z}.
    Furthermore, let $Y^{(n)}$ be the solution of
    the system of the ordinary differential equations \eqref{eq:pde_of_approximate_Y}
    and $\sigma^{(n)}$ be the explosion time of $Y^{(n)}$.
    Then, there exists $N = N(T, R)$ such that $n \geq N$ and
    $\norm{u_0}_{\contisp^{-\alpha_0}} + \znorm{\underline{Z}^{(n)}}_T \leq R$ implies $\sigma^{(n)} \geq T$.
    \item Let $W = (W_m)_{m \in \Z^2}$ be i.i.d. $\C$-valued Brownian motions. Define
    $Z^{(n)}$, $Y^{(n)}$ and $\sigma^{(n)}$ as above for $w = W$.
    Let $\tilde{Y}^{(n)} \in S^T(\Pi_n u_0, \underline{Z}^{(n)})$.
    Then,
    \begin{align*}
      &\lim_{n \to \infty} \P(\sigma^{(n)} \leq T) = 0 \quad \mbox{and}\\
      &\lim_{n \to \infty} \P(\sup_{0 \leq t \leq T} \norm{Y^{(n)}(t) - \tilde{Y}^{(n)}(t)}_{\contisp^{\alpha_1}} \geq \epsilon,
      \sigma^{(n)} > T) = 0
    \end{align*}
    for every $T \in (0, \infty)$ and $\epsilon \in (0, 1)$.
  \end{enumerate}
\end{lemma}
\begin{proof}
  Set $\tau^{(n)} \defby \inf \set{t > 0 \given t^{\gamma} \norm{Y^{(n)}(t) - \tilde{Y}^{(n)}(t)}_{\contisp^{\alpha_1}} \geq 1}$.
  We have $\sigma^{(n)} > \tau^{(n)}$ and for $t \leq \tau^{(n)}$
  \begin{equation}\label{eq:difference_Y_n_and_tilde_Y_n}
    Y^{(n)}(t) - \tilde{Y}^{(n)}(t) = \int_0^t e^{(t-s)A} \{ \Pi_n \Psi(Y^{(n)}(s), \underline{Z}^{(n)}(s))
    - \Psi(\tilde{Y}^{(n)}(s), \underline{Z}^{(n)}(s)) \} ds.
  \end{equation}
  Set
  \begin{equation*}
    K(n) \defby 1 + \sup_{0 < t \leq T}\left[ t^{3\gamma} \norm{\tilde{Y}^{(n)}(t)}^3_{\contisp^{\alpha_1}} +
     \norm{\underline{Z}^{(n)}(t)}_{\zset_{\alpha}}  \right].
  \end{equation*}
  Proposition \ref{prop:continuity_parameters} implies $K(n) \lesssim R^{\kappa}$ for some $\kappa \in (0, \infty)$.
  We have
  \begin{multline*}
    \norm{\Pi_n \Psi(Y^{(n)}(s), \underline{Z}^{(n)}(s)) - \Psi(\tilde{Y}^{(n)}(s),
    \underline{Z}^{(n)}(s))}_{\contisp^{-2\alpha}} \\
    \leq
    \norm{\Pi_n \Psi(Y^{(n)}(s), \underline{Z}^{(n)}(s)) -
    \Pi_n\Psi(\tilde{Y}^{(n)}(s), \underline{Z}^{(n)}(s))}_{\contisp^{-2\alpha}}\\
    +
    \norm{(\Pi_n - 1) \Psi(\tilde{Y}^{(n)}(s), \underline{Z}^{(n)}(s))}_{\contisp^{-2\alpha}}.
  \end{multline*}
  Suppose $s \leq \min \{\tau^{(n)}, T\}$ so that $\norm{Y^{(n)}(s) - \tilde{Y}^{(n)}(s)}_{\contisp^{\alpha_1}} \leq 1$.
  By Lemma \ref{lem:properties_of_Pi}-(i), the first term is bounded by
  \begin{equation*}
    C K(n) s^{-2\gamma} \norm{Y^{(n)}(s) - \tilde{Y}^{(n)}(s)}_{\contisp^{\alpha_1}}.
  \end{equation*}
  By Lemma \ref{lem:properties_of_Pi}-(ii), the second term is bounded by $\epsilon(n) K(n) s^{-3\gamma}$ where
  $\epsilon(n)$ is deterministic and converges to $0$ as $n \to \infty$.
  Therefore, applying Proposition \ref{prop:smoothing_in_besov} to \eqref{eq:difference_Y_n_and_tilde_Y_n},
  we obtain
  \begin{multline*}
    \sup_{0<s\leq t} \norm{Y^{(n)}(s) - \tilde{Y}^{(n)}(s)}_{\contisp^{\alpha_1}}\\
    \leq C_1 R^{\kappa} \epsilon(n) t^{\kappa_1} + C_2 R^{\kappa} t^{\kappa_2}
    \sup_{0<s\leq t} \norm{Y^{(n)}(s) - \tilde{Y}^{(n)}(s)}_{\contisp^{\alpha_1}}.
  \end{multline*}

  Set $m \defby \lfloor (2C_2 R^{\kappa})^{\frac{1}{\kappa_2}} T \rfloor + 1$ so that
  $C_2 R^{\kappa} (\frac{T}{m})^{\kappa_2} \leq \frac{1}{2}$.
  Then we have
  \begin{equation*}
    \sup_{0 \leq t \leq \min\{T_1, \tau^{(n)} \}} \norm{Y^{(n)} (t) - \tilde{Y}^{(n)}(t)}_{\contisp^{\alpha_1}}
    \leq C_3 \epsilon(n),
  \end{equation*}
  where $T_1 \defby \frac{T}{m}$ and $C_3 \defby 2 C_1 R^{\kappa} T_1^{\kappa_1}$.
  Note that $C_3 \epsilon(n) < 1$ implies $T_1 < \tau^{(n)}$. In this case, we have
  \begin{multline*}
    Y^{(n)}(T_1 + t) - \tilde{Y}^{(n)}(T_1 + t) =
    e^{tA} (Y^{(n)}(T_1) - \tilde{Y}^{(n)}(T_1)) \\
    + \int_0^t e^{(t-s)A} \{ \Pi_n \Psi(Y^{(n)}(T_1 + s), \underline{Z}^{(n)}(T_1+s))
    - \Psi(\tilde{Y}^{(n)}(T_1+s), \underline{Z}^{(n)}(T_1+s)) \} ds.
  \end{multline*}
  Repeating the above argument, we obtain
  \begin{equation*}
    \sup_{T_1 \leq t \leq \min\{2T_1, \tau^{(n)}\}} \norm{Y^{(n)}(t) - \tilde{Y}^{(n)}(t)}_{\contisp^{\alpha_1}}
    \leq (c+1) C_3 \epsilon(n),
  \end{equation*}
  where $c \defby \sup_{t \geq 0} \norm{e^{tA}}_{\contisp^{\alpha_1} \to \contisp^{\alpha_1}}$.
  Continuing, we have
  \begin{equation*}
    \sup_{\min\{(k-1)T_1, \tau^{(n)}\} \leq t \leq \min\{kT_1, \tau^{(n)}\}}
    \norm{Y^{(n)}(t) - \tilde{Y}^{(n)}(t)}_{\contisp^{\alpha_1}}
    \leq [(k-1)c + 1] C_3 \epsilon(n).
  \end{equation*}
  In particular, $\tau^{(n)} > T$ if $[(m-1)c + 1] C_3 \epsilon(n) < 1$.
  Therefore we end the proof of (i).

  Now we move to the proof of (ii). Since the above implies
  \begin{equation*}
    \limsup_{n \to \infty} \P(\sigma^{(n)} \leq T) \leq \P(\znorm{\underline{Z}}_T \geq R)
  \end{equation*}
  for every $T, R \in (0, \infty)$, we see that $\sigma^{(n)} \to \infty$ in probability.
  The convergence of $Y^{(n)} - \tilde{Y}^{(n)}$ is similar.
\end{proof}

\subsection{Bismut-Elworthy-Li formula}\label{subsec:bismut_elworthy_li_formula}
We fix $T, R \in (0, \infty)$ and take $N(T, 2R+1)$ in Lemma \ref{lem:global_wellposedness_of_approximate_Y}-(i).
In this subsection,
we will solely work on the equation \eqref{eq:approximate_sdes} for a fixed $n \geq N(T, 2R+1)$ and hence
we will omit indices for $n$ until the end of this subsection.
\begin{remark}
  In this subsection, every vector space is regarded as real. This is ultimately because we want
  the function $z \mapsto \conj{z}$ to be differentiable. For $x = (x_i), y=(y_i) \in \C^{d(n)}$,
  we set
  \begin{equation*}
  \realinp{x}{y} \defby \sum \{\Re(x_i) \Re(y_i) + \Im(x_i) \Im(y_i)\}.
  \end{equation*}
\end{remark}
The first task is to prove the Fr\'echet differentiability of $u$ with respect to
the initial value and the noise $(W_m)_{\abs{m} \leq n}$.
We view that the initial value belongs to the space $\contisp^{-\alpha_0}$ and that
the noise belongs to the space
\begin{equation*}
  \noisesp \defby \set{ (W_m)_{\abs{m} \leq n} \in C([0, T]; \C^{d(n)}) \given W_m(0) = 0 \mbox{ for all } m }.
\end{equation*}
We denote by $\initdiff_h$ and by $\noisediff_w$ the Fr\'echet derivative with respect to the initial value in the direction
of $h$ and
with respect to the noise in the direction of $w$ respectively.

Let $Z(0, \cdot)$ be the solution of \eqref{eq:approximate_Z}, or
\begin{equation}\label{eq:inp_of_Z_and_em}
  \inp{Z(0, t)}{e_{-m}}_1 = \indic_{\{\abs{m} \leq n\}}a_m \int_0^t e^{-(t-s)[4\pi^2(i+\mu) \abs{m}^2 + 1]} dW_m(s).
\end{equation}
Since the map $s \mapsto e^{-(t-s)[4\pi^2(i+\mu) \abs{m}^2 + 1]}$ is smooth, the integral is regarded as a
Riemann-Stieltjes integral. Therefore, it makes sense for every fixed $(W_m)_{\abs{m} \leq n} \in \noisesp$, and
the map
\begin{equation*}
  \noisesp \ni (W_m)_{\abs{m} \leq n} \mapsto Z(0, t) = \sum_{\abs{m} \leq n} \inp{Z(0, t)}{e_{-m}}_1 e_m
  \in \contisp^1
\end{equation*}
is linear and in particular Fr\'echet differentiable.
\begin{proposition}\label{prop:frechet_differentiability}
  For given $W = (W_m)_{\abs{m} \leq n}$, we
  define $Z = Z(W)$ by
  \begin{equation*}
    Z(0, t) \defby \sum_{\abs{m} \leq n} \inp{Z(0,t)}{e_{-m}}_1 e_m.
  \end{equation*}
  where  $\inp{Z(0,t)}{e_{-m}}_1$ is given by \eqref{eq:inp_of_Z_and_em}.
  Let $Y$ be the solution of \eqref{eq:pde_of_approximate_Y}.
  Set $u(t) \defby Z(0, t) + Y(t)$.
  If we set
  \begin{equation*}
    \mathcal{U} \defby \set{ (u_0, W) \in \contisp^{-\alpha_0} \times \noisesp
    \given \norm{u_0}_{\contisp^{-\alpha_0}} + \znorm{\underline{Z}(W)}_T \leq R},
  \end{equation*}
  where $\underline{Z}(W)$ is defined as in \eqref{eq:def_of_approximate_underline_z},
  then the map
  \begin{equation*}
    \mathcal{U} \ni (u_0, (W_m)_{\abs{m} \leq n} ) \mapsto u \in C([0, T]; \contisp^{-\alpha_0})
  \end{equation*}
  is Fr\'echet differentiable. Moreover, we have
  \begin{align}
    \label{eq:identity_of_initdiff}
    \initdiff_h u(t) ={}& \int_0^t e^{(t-s)A} \Pi_n \Psi'(Y(s), \underline{Z}(s)) ( \initdiff_h u(s) ) ds
     + e^{tA} \Pi_n h,\\
    \label{eq:identity_of_noisediff}
    \begin{split}
    \noisediff_w u(t) ={}&
    \int_0^t e^{(t-s)A} \Pi_n \Psi'(Y(s), \underline{Z}(s)) (\noisediff_w u(s)) ds  \\
    &\hspace{3em} + \sum_{\abs{m} \leq n} a_m \int_0^t e^{-(t-s)[4\pi^2(i+\mu) \abs{m}^2 + 1]} e_m dw_m(s),
    \end{split}
  \end{align}
  where
  \begin{equation*}
    \Psi'(Y(s), \underline{Z}(s))(\zeta) \defby
    \partial_z H_{2,1}(Y(s) + Z(0, s)) \zeta + \partial_{\conj{z}} H_{2,1}(Y(s) + Z(0, s)) \conj{\zeta}.
  \end{equation*}
\end{proposition}
\begin{remark}
  We restrict the domain to $\mathcal{U}$ in order to ensure the non-explosion of the solution $u$.
\end{remark}
\begin{proof}
  We have already checked that $\noisesp \ni (W_m) \mapsto Z(0, \cdot) \in C([0,T]; \contisp^1)$ is Fr\'echet differentiable.
  To prove the Fr\'echet differentiability of $Y$, we note that $Y = Y(u_0, Z(0, \cdot))$ is the zero of
  \begin{equation*}
    F(y; u_0, Z(0, \cdot))(t) \defby y_t - e^{tA} \Pi_n u_0 - \int_0^t e^{(t-s)A} \Pi_n
    \Psi(Y(s), \underline{Z}(s)) ds.
  \end{equation*}
  If we suppose that $y \in C([0, T]; \contisp^{-\alpha_0})$, $u_0 \in \contisp^{-\alpha_0}$ and $Z(0, \cdot) \in C([0, T]; \contisp^1)$,
  $F$ is Fr\'echet differentiable with respect to $y$, $u_0$ and $Z(0, \cdot)$.
  We want to apply the implicit function theorem for general Banach spaces. To do so, we need to check that
  the derivative $\partial_y F$ is nondegenerate.

  Lemma \ref{lem:global_wellposedness_of_approximate_Y} implies
  \begin{equation*}
    \sup_{0 \leq t \leq T} \norm{Y(u_0, Z(0, \cdot))(t)}_{\contisp^{-\alpha_0}} \lesssim_R 1.
  \end{equation*}
  Therefore, there exists a constant $C = C(R) \in (0, \infty)$ such that
  \begin{equation*}
    \sup_{0 \leq s \leq t} \norm{ \partial_y \int_0^s e^{(s-r)A} \Pi_n \Psi(y(r),
     \underline{Z}(r)) dr }_{\contisp^{-\alpha_0}
    \to \contisp^{-\alpha_0}} \vert_{y = Y(u_0, Z(0, \cdot))} \leq C t^{\kappa}.
  \end{equation*}
  If we set $T_1 \defby \min \{ (2C)^{-\frac{1}{\kappa}}, T \}$ and view that $F$ maps from
  $C([0, T_1]; \contisp^{-\alpha_0}) \times \contisp^{-\alpha_0} \times C([0, T]; \contisp^1)$ to
  $C([0, T_1]; \contisp^{-\alpha_0})$,
  the derivative $\partial_y F$ is invertiable at $y = Y(u_0, Z(0, \cdot))$.
  The implicit function theorem implies that the solution $Y \in C([0, T_1];$ $\contisp^{-\alpha_0})$ is
  differentiable around $(u_0, Z(0, \cdot))$. We can repeat this process by changing $F$ to
  \begin{multline*}
    F_1(y; u_0, Z(0, \cdot))(t) \\ \defby
    y_t - e^{tA} Y(u_0, Z(0, \cdot))(T_1) - \int_0^t e^{(t-s)A} \Pi_n
    \Psi(y(s), \underline{Z}(T_1 + s)) ds,
  \end{multline*}
  and see that $Y(T_1 + \cdot) \in C([0, T_1]; \contisp^{-\alpha_0})$ is differentiable. Continuing,
  we observe that $Y \in C([0, T]; \contisp^{-\alpha_0})$ is differentiable around $(u_0, Z(0, \cdot))$.
  Therefore, $Y$ is Fr\'echet differentiable with respect to $u_0$ and $Z(0, \cdot)$, and hence
  with respect to $u_0$ and $W = (W_m)$.
  Now the claim of differentiability follows.

  Finally, we prove \eqref{eq:identity_of_initdiff} and \eqref{eq:identity_of_noisediff}.
  As $Z(0, \cdot)$ is independent of $u_0$, we have $\initdiff u = \initdiff Y$. Then, \eqref{eq:identity_of_initdiff}
  follows by differentiating both sides of
  \begin{equation}\label{eq:mild_form_of_Y_n}
    Y(t) = e^{tA} \Pi_n u_0 + \int_0^t e^{(t-s)A} \Pi_n \Psi(Y_s, \underline{Z}(s)) ds
  \end{equation}
  with respect to the initial value.
  As $(W_m) \mapsto Z(0, \cdot)$ is linear, it is easy to see that
  $\noisediff_w Z(0, \cdot)$ equals to the second term of the right hand side of \eqref{eq:identity_of_noisediff}.
  Then, \eqref{eq:identity_of_noisediff} follows by differentiating both sides of
  \eqref{eq:mild_form_of_Y_n} with respect to the noise.
\end{proof}
We set
\begin{equation}\label{eq:def_of_tau}
  \tau_i \defby \inf\set{t \given \znorm{\underline{Z}}_t \geq i} \quad \mbox{for } i = 1, 2.
\end{equation}

\begin{lemma}\label{lem:exstence_of_T_star}
  Let $u$ and $Y$ be as above.
  Assume $\norm{x}_{\contisp^{-\alpha_0}} \leq R$. Then, there exists $\kappa \in (0, \infty)$ such that,
  if we set $T^* \defby (1+R)^{-\kappa}$, we have
  \begin{equation*}
    \sup_{t \leq \min\{T^*, \tau_2\}} t^{\gamma} \norm{Y(t;x)}_{\contisp^{\alpha_1}} \leq 1
    \quad \mbox{and} \quad \sup_{t \leq \min\{T^*, \tau_2\}} t^{\gamma}
    \norm{\initdiff_h u(t;x)}_{\contisp^{\alpha_1}} \leq  \norm{h}_{\contisp^{-\alpha_0}}
  \end{equation*}
  for every $h \in \contisp^{-\alpha_0}$.
\end{lemma}
\begin{proof}
  By exploiting the integral equations for $Y(t;x)$ and $\initdiff_h u(t;x)$,
  the estimates can be obtained as in the proof of
  Theorem \ref{thm:local_wellposedness} and Lemma \ref{lem:global_wellposedness_of_approximate_Y}.
\end{proof}
Let $\chi: \R \to [0, 1]$ be a smooth function such that $\chi(x) = \indic_{[0, 1]}(\abs{x})$
provided $\abs{x} \leq 1$ or $\abs{x} \geq 2$.
Set
\begin{equation*}
  Qw(t) \defby \sum_{\abs{m} \leq n} a_m \int_0^t e^{(t-s)A} e_m dw_m(s).
\end{equation*}
\begin{theorem}\label{thm:bismut_elworthy_li_formula}
  Let $w \in C([0, T]; \C^{d(n)})$ with $\partial_s w \in L^2([0, T]; \C^{d(n)})$ and with $\partial_s w$ adapted.
  Suppose that there exists a deterministic constant $C \in (0, \infty)$ such that
  $\norm{\partial_s w}_{L^2([0, T]; \C^{d(n)}} \leq C$ almost surely.
  In addition, let $W = (W_m)_{\abs{m} \leq n}$ be a complex Brownian motion and define
  $Z(0, \cdot), Y$ and $u$ as before.
  Finally, assume $\norm{u_0}_{\contisp^{-\alpha_0}} \leq R$.
  Then, we have the identity
  \begin{multline*}
    \expect[ D\Phi(u(t))(\noisediff_w u(t)) \chi(\znorm{\underline{Z}}_t) ]
    \\=- \expect[ \Phi(u(t)) \partial_{+} \chi(\znorm{\underline{Z}}_t)(w) ]
    + \expect \Big[ \Phi(u(t))  \chi(\znorm{\underline{Z}}_t)
    \int_0^t \realinp{\partial_s w(s)}{d W(s)} \Big],
  \end{multline*}
  where
  \begin{equation}\label{eq:one_sided_derivative_of_chi}
    \partial_{+} \chi (\znorm{\underline{Z}}_t)(w) \defby \chi'(\znorm{\underline{Z}}_t)
    \lim_{\delta \to 0+} \frac{\znorm{T^{\delta}\underline{Z}}_t - \znorm{\underline{Z}}_t}{\delta},
  \end{equation}
  \begin{equation*}
    T^{\delta}\underline{Z} \defby (H_{k,l}(Z(0, \cdot) + \delta Qw, c_n))_{(k,l) = (1,0), (2,0), (1, 1), (2, 1)}.
  \end{equation*}
\end{theorem}
\begin{remark}
  The existence of the limit in \eqref{eq:one_sided_derivative_of_chi} is
  guaranteed in \cite[Appendix D]{DPZ14}.
\end{remark}
\begin{proof}
  Set $\dot{w} \defby \partial_s w$. Let $Y^{\delta}$ be the solution of
  \eqref{eq:pde_of_approximate_Y} with $\underline{Z}^{(n)}$ replaced by $T^{\delta}\underline{Z}$.
  Furthermore, we set
  \begin{equation*}
    B^{\delta}(t) \defby - \delta \int_0^t \realinp{\dot{w}(s)}{dW(s)},
  \end{equation*}
  \begin{equation*}
    A^{\delta}(t) \defby \exp \left( B^{\delta}(t) - \frac{\delta^2}{4} \int_0^t \abs{\dot{w}(s)}^2 ds \right).
  \end{equation*}

  Since the assumption of $w$ implies
  \begin{equation*}
    \expect\left[ \exp\left( \frac{\delta^2}{4} \int_0^t \abs{\dot{w}(s)}^2 ds \right)\right] < \infty,
  \end{equation*}
  Novikov's condition is satisfied.
  Thus, if we define a probability measure $\P^{\delta}$ by $d \P^{\delta} = A^{\delta}(t) d\P$,
  the Girsanov theorem implies that $W^{\delta} \defby W + \delta w$ has the same law under $\P^{\delta}$ as $W$ has under $\P$.
  Therefore, we obtain
  \begin{equation}\label{eq:derivative_of_shifted_equation}
    \left. \frac{\partial}{\partial \delta} \right\rvert_{\delta = 0+}
    \expect^{\P}[ \Phi(u^{\delta}(t)) \chi(\znorm{T^{\delta} \underline{Z}}_t) A^{\delta}(t) ] = 0.
  \end{equation}
  Since we have
  \begin{align*}
    &\left. \frac{\partial}{\partial \delta}\right\rvert_{\delta=0}
    \Phi(u^{\delta}(t)) = D\Phi(u(t))(\noisediff_w u(t))  \\
    &\left. \frac{\partial}{\partial \delta}\right\rvert_{\delta=0} A^{\delta}(t) =
    - \int_0^t \realinp{\dot{w}(s)}{dW(s)},
  \end{align*}
  it remains to verify that we can interchange the differentiation and the expectation in \eqref{eq:derivative_of_shifted_equation}.

  To this end, we will prove
  \begin{equation}\label{eq:bound_of_difference_of_L^p_norm}
    \sup_{\delta \in (0, 1)} \expect \left[ \abs*{\frac{\Phi(u^{\delta}(t)) \chi(\znorm{T^{\delta} \underline{Z}}_t)
    A^{\delta}(t) - \Phi(u(t)) \chi(\znorm{\underline{Z}}_t)}{\delta}}^p \, \right] < \infty
  \end{equation}
  for $p \in (1, \infty)$.
  We estimate $\Phi(u^{\delta}(t)) - \Phi(u(t))$, $\chi(\znorm{T^{\delta}\underline{Z}}_t) - \chi(\znorm{\underline{Z}}_t)$ and
  $A^{\delta}(t) - 1$ separately.

  To begin with, we estimate
  \begin{align*}
    \abs{\Phi(u^{\delta}(t)) - \Phi(u(t))} &= \abs*{ \int_0^{\delta} D\Phi(u^{\lambda}(t)) ( \noisediff_w u^{\lambda}(t)) d\lambda } \\
    &\leq \norm{D\Phi}_{L^{\infty}} \int_0^{\delta}
    \norm{\noisediff_w u^{\lambda}(t)}_{\contisp^{-\alpha_0}} d\lambda.
  \end{align*}
  If $\chi({\znorm{\underline{Z}}_t}) > 0$, applying Gr\"onwall's inequality to \eqref{eq:identity_of_noisediff},
  $\norm{\noisediff_w u^{\lambda}(t)}_{\contisp^{-\alpha_0}}$
  is deterministically bounded, and therefore
  \begin{equation*}
    \abs{\Phi(u^{\delta}(t)) - \Phi(u(t))} \lesssim \delta.
  \end{equation*}

  Next, we observe that
  \begin{multline*}
    \abs{ \chi(\znorm{T^{\delta} \underline{Z}}_t) - \chi(\znorm{\underline{Z}}_t) } \\
    = \abs*{ \int_0^1 D\chi(\lambda \znorm{T^{\delta}\underline{Z}}_t + (1-\lambda) \znorm{\underline{Z}}_t)
    (\znorm{T^{\delta} \underline{Z}}_t - \znorm{\underline{Z}}_t) d\lambda} \lesssim \delta \znorm{Z}_t.
  \end{multline*}
  Similarly, we have
  \begin{equation*}
    \abs{A^{\delta}(t) - 1} \leq \int_0^{\delta} \abs{\partial_{\lambda} A^{\lambda}(t)} d\lambda.
  \end{equation*}
  We note that
  \begin{equation*}
    \partial_{\lambda}A^{\lambda}(t) = -A^{\lambda}(t)
    \Big( \int_0^t \realinp{\dot{w}(s)}{dW(s)} + \frac{\lambda}{2}\int_0^t \abs{v(s)}^2 ds \Big )
  \end{equation*}
  and that
  \begin{align*}
    &\expect \left[ \abs*{\frac{\Phi(u^{\delta}(t)) \chi(\znorm{T^{\delta} \underline{Z}}_t)
    (A^{\delta}(t)-1)}{\delta}}^p \right] \\
    &\lesssim_{\Phi, \chi, p}\expect \left[ \frac{1}{\delta} \int_0^{\delta}
    \abs*{ A^{\lambda}(t) \Big( \int_0^t \realinp{\dot{w}(s)}{dW(s)} + \frac{\lambda}{2}
    \int_0^t \abs{\dot{w}(s)}^2 ds \Big)}^p d\lambda \right] \\
    &\lesssim_p \sup_{\lambda \in (0, 1)} \expect\left[ A^{\lambda}(t)^{2p} \right] + \expect \left[ \Big \lvert
    \int_0^t \realinp{\dot{w}(s)}{dW(s)} \Big\rvert^{2p} \right]
    + \expect \left[ \Big( \int_0^t \abs{\dot{w}(s)}^2 ds \Big)^{2p} \right].
  \end{align*}
  Burkholder-Davis-Gundy inequalities and the assumption of $w$ imply
  \begin{equation*}
    \expect \left[ \Big \lvert
    \int_0^t \realinp{\dot{w}(s)}{dW(s)} \Big\rvert^{2p} \right]
    + \expect \left[ \Big( \int_0^t \abs{\dot{w}(s)}^2 ds \Big)^{2p} \right] < \infty.
  \end{equation*}

  Therefore, to prove \eqref{eq:bound_of_difference_of_L^p_norm}, it suffices to prove
  \begin{equation}\label{eq:sup_bound_of_martingale}
    \sup_{\lambda \in (0, 1)} \expect[ A^{\lambda}(t)^{2p} ]  < \infty.
  \end{equation}
  We have
  \begin{equation*}
    A^{\lambda}(t)^{2p} = A^{2p \lambda}(t) \exp \left( \frac{\lambda^2(4p^2-1)}{4} \int_0^t \abs{\dot{w}(s)}^2 ds \right).
  \end{equation*}
  Since $\int_0^t \abs{\dot{w}(s)}^2 ds$ is deterministically bounded, the proof is complete once we notice that
  $\expect[ A^{2p \lambda}(t) ] = 1$.
\end{proof}
\begin{proposition}
  Let $T^*$ be the time constructed in Lemma \ref{lem:exstence_of_T_star}.
  Then there exist constants $C \in (0, \infty)$ and $\theta \in (0, 1)$ such that
  \begin{equation}\label{eq:bound_on_P_t_x_minus_P_t_y}
    \abs{P_t \Phi(v_0) - P_t \Phi(v_1)} \leq \frac{C}{t^{\theta}} \norm{\Phi}_{L^{\infty}} \norm{v_0 - v_1}_{\contisp^{-\alpha_0}}
    + 2 \norm{\Phi}_{L^\infty} \P(t \geq \tau_1)
  \end{equation}
  for $v_0, v_1 \in \contisp^{-\alpha_0}$ with $\norm{v_0}_{\contisp^{-\alpha_0}},
  \norm{v_1}_{\contisp^{-\alpha_0}} \leq R$, $\Phi \in C^1_b(\contisp^{-\alpha_0})$ and $t \leq T^*$.
\end{proposition}
\begin{proof}
  We first note that the right hand side of \eqref{eq:bound_on_P_t_x_minus_P_t_y} is bounded by the sum of
  \begin{align*}
    I_1 &\defby \abs{ \expect[ \{\Phi(u(t;v_0)) - \Phi(u(t;v_1))\} \chi(\znorm{\underline{Z}}_t)]}, \\
    I_2 &\defby \abs{ \expect[ \{\Phi(u(t;v_0)) - \Phi(u(t;v_1))\} (1 - \chi(\znorm{\underline{Z}}_t))]}.
  \end{align*}
  Since $I_2 \leq 2 \norm{\Phi}_{L^{\infty}} \P(t \geq \tau_1)$, it suffices to show that
  \begin{equation*}
    I_1 \leq \frac{C}{t^{\theta}} \norm{\Phi}_{L^{\infty}} \norm{v_0-v_1}_{\contisp^{-\alpha_0}}.
  \end{equation*}

  The mean value theorem implies
  \begin{equation*}
    I_1 = \abs{ \expect[ \{\int_0^1 D\Phi(u(t;v_{\lambda})) \initdiff_{y-x} u(t; v_{\lambda}) d\lambda\} \chi(\znorm{
    \underline{Z}}_t) ] }
  \end{equation*}
  with $v_{\lambda} \defby (1-\lambda)v_0  + \lambda v_1$.
  Let $w$ be the element of $C([0, T]; \C^{d(n)})$ such that $w(0) = 0$ and
  \begin{equation*}
    \partial_s w^{\lambda}(s) = \initdiff_h u(s; v_{\lambda}) \indic_{\{s \leq \tau_2\}}
  \end{equation*}
  for given $h \in \contisp^{-\alpha_0}$.
  Then we have $\norm{\partial_s w^{\lambda}}_{L^2([0, t]; \C^{d(n)})} \leq C$
  for some deterministic constant $C \in (0, \infty)$ by Lemma \ref{lem:exstence_of_T_star}.

  Furthermore, we have $\noisediff_w u(t; v_{\lambda}) = t \initdiff_h u(t;v_{\lambda})$ if $t \leq \tau_2$.
  Indeed,
  letting $Y^{\lambda}$ be the solution of \eqref{eq:pde_of_approximate_Y} with the initial value $v_{\lambda}$, we observe that
  \begin{align*}
    \frac{\partial}{\partial t} (t \initdiff_h u(t;v_{\lambda}))
    &= J_{0, t}^{\lambda}h + t [ A \initdiff_h u(t;v_{\lambda}) + \Pi_n \Psi'(Y^{\lambda}(t), \underline{Z}(t)) J_{0,t}^{\lambda}h ] \\
    &= A(t \initdiff_h u(t;v_{\lambda})) + \Pi_n \Psi'(Y^{\lambda}(t), \underline{Z}(t))(t \initdiff_h u(t;v_{\lambda})) + \initdiff_h u(t;v_{\lambda}).
  \end{align*}
  Therefore, we obtain
  \begin{multline*}
    t \initdiff_h u(t;v_{\lambda}) \\= \int_0^t e^{(t-s)A} \Pi_n \Psi'(Y^{\lambda}(s), \underline{Z}(s))(s
    \initdiff_h u(s; v_{\lambda})) ds
    + \int_0^t e^{(t-s)A}  \partial_s w(s) ds.
  \end{multline*}
  In view of \eqref{eq:identity_of_noisediff}, both $\noisediff_w u(t; v_{\lambda})$ and $t \initdiff_h u(t; v_{\lambda})$
  satisfy the same integral equation.

  Theorem \ref{thm:bismut_elworthy_li_formula} now implies
  \begin{align*}
    &\expect[ D\Phi(u(t;v_{\lambda})) \initdiff_h u(t;v_{\lambda}) \chi(\znorm{\underline{Z}}_t) ]
    = \frac{1}{t} \expect[ D\Phi(u(t;v_{\lambda})) \noisediff_w u(t; v_{\lambda}) \chi(\znorm{\underline{Z}}_t) ]\\
    &= \frac{1}{t} \Big\{ - \expect[ \Phi(u(t;v_{\lambda})) \partial_+ \chi(\znorm{\underline{Z}}_t) (w) ] \\
    &\hspace{10em} + \expect[ \Phi(u(t;v_{\lambda})) \chi(\znorm{\underline{Z}}_t) \int_0^t
    \realinp{\initdiff_h u(s; v_{\lambda})}{dW(s)} ] \Big\}.
  \end{align*}
  We observe
  \begin{equation*}
    \abs{\partial_+ \chi(\znorm{\underline{Z}}_t)(w)} \\
    \lesssim \abs{\chi'(\znorm{\underline{Z}}_t)} \cdot \znorm{\underline{Z}}_t
    \sup_{0 \leq s \leq t} \norm{ \int_0^s e^{(s-r)A} \initdiff_{h} u(r; v_{\lambda})}_{\contisp^{\alpha_1} dr }
  \end{equation*}
  and, thanks to Lemma \ref{lem:exstence_of_T_star},
  \begin{equation*}
    \sup_{0 \leq s \leq t} \norm{ \int_0^s e^{(s-r)A} \initdiff_{h} u(r; v_{\lambda})}_{\contisp^{\alpha_1} dr}
    \lesssim t^{1 - \gamma} \norm{h}_{\contisp^{-\alpha_0}}.
  \end{equation*}
  Therefore, we obtain
  \begin{equation*}
    \abs{\expect[ \Phi(u(t;v_{\lambda})) \partial_+ \chi(\znorm{\underline{Z}}_t) (w) ]} \lesssim
    t^{1 - \gamma} \norm{\Phi}_{L^{\infty}} \norm{h}_{\contisp^{-\alpha_0}}.
  \end{equation*}

  On the other hand, by It\^o isometry and Lemma \ref{lem:exstence_of_T_star},
  \begin{align*}
    \expect[ \Big( \int_0^t \chi(\znorm{\underline{Z}}_t) \norm{\initdiff_h u(s; v_{\lambda})}_{\contisp^{-\alpha}}
    \Big)^2 ]
    &= \int_0^t \expect[ \chi(\znorm{\underline{Z}}_t)^2 \norm{\initdiff_h u(s; v_{\lambda})}_{L^2(\torus)}^2 ] ds \\
    &\lesssim t^{1- 2 \gamma} \norm{h}_{\contisp^{-\alpha_0}}^2.
  \end{align*}
  Therefore, we obtain
  \begin{equation*}
    \abs{\expect[ \Phi(u(t;v_{\lambda})) \chi(\znorm{\underline{Z}}_t)
    \int_0^t \realinp{\initdiff_h u(s; v_{\lambda})}{dW(s)} ]}
    \lesssim t^{\frac{1}{2} - \gamma} \norm{\Phi}_{L^{\infty}} \norm{h}_{\contisp^{-\alpha}}.
  \end{equation*}

  Consequently, substituting $h = v_1 - v_0$, we get
  \begin{equation*}
    I_1 \lesssim t^{-\frac{1}{2} - \gamma} \norm{\Phi}_{L^{\infty}} \norm{v_1 - v_0}_{\contisp^{-\alpha_0}} . \qedhere
  \end{equation*}

\end{proof}
\subsection{H\"older continuity in the total variation norm}\label{subsec:holder_continuity}
Here we go back to the original SPDE $\eqref{eq:CGL}$ instead of the approximation \eqref{eq:approximate_sdes}.
In the previous subsection, we obtained the estimate
\begin{equation*}
  \abs{P_t^{(n)} \Phi (v_0) - P_t^{(n)} \Phi (v_1)} \leq \frac{C}{t^{\theta}} \norm{\Phi}_{L^{\infty}} \norm{v_0 - v_1}_{\contisp^{-\alpha_0}}
  + 2 \norm{\Phi}_{L^{\infty}} \P(t \geq \tau^{(n)}_1)
\end{equation*}
for $t \leq T^*$ and $\norm{v_0}_{\contisp^{-\alpha_0}}, \norm{v_1}_{\contisp^{-\alpha_0}} \leq R$.
Since we have $\limsup_{n \to \infty} \P(t \geq \tau^{(n)}_1) \leq \P(t \geq \tau_1)$,
by taking the limit, we have
\begin{equation*}
  \abs{P_t \Phi (v_0) - P_t \Phi (v_1)} \leq \frac{C}{t^{\theta}} \norm{\Phi}_{L^{\infty}} \norm{v_0 - v_1}_{\contisp^{-\alpha_0}}
  + 2 \norm{\Phi}_{L^{\infty}} \P(t \geq \tau_1).
\end{equation*}
According to \cite[Lemma 7.1.5]{DPZ96}, this estimate is equivalent to
\begin{equation}\label{eq:total_variation}
  \tvnorm{P_t^* \delta_{v_0} - P_t^* \delta_{v_1}} \leq \frac{C}{2t^{\theta}} \norm{v_0 - v_1}_{\contisp^{-\alpha_0}} + \P(t \geq \tau_1).
\end{equation}
\begin{theorem}\label{thm:continuity_in_total_variation}
  Let $R, t_0 \in (0, \infty)$. Then
  there exist $\sigma \in (0, \infty)$ and $\tilde{\theta} \in (0, 1)$ such that
  \begin{equation*}
    \tvnorm{P_t^* \delta_{v_0} - P_t^*\delta_{v_1}} \lesssim (1 + t_0^{-1} + R)^{\sigma} \norm{v_0 - v_1}_{\contisp^{-\alpha_0}}^{\tilde{\theta}}
  \end{equation*}
  for every $v_0, v_1 \in \contisp^{-\alpha_0}$ with $\norm{v_0}_{\contisp^{-\alpha_0}},
  \norm{v_1}_{\contisp^{-\alpha_0}} \leq 1$ and $t \geq t_0$.
\end{theorem}
\begin{proof}
  We first observe $t \mapsto \tvnorm{P_t^* \delta_{v_0} - P_t^* \delta_{v_1}}$ is nonincreasing. Thus, we can assume
  $t_0 \leq 1$ and $t = t_0$.
  By  Proposition \ref{prop:L^p_estimate_of_nonstationary_OU} and Remark \ref{rem:continuity_of_nonstationary_OU}, we have
  \begin{equation*}
    \P( t \geq \tau_1) = \P(\znorm{\underline{Z}}_t \geq 1) \leq \expect[ \znorm{\underline{Z}}_t ] \lesssim t^{\theta_2}
  \end{equation*}
  for some $\theta_2 \in (0, 1)$.
  Therefore, combined with \eqref{eq:total_variation}, we have
  \begin{equation*}
    \tvnorm{P_t^* \delta_{v_0} - P_t^* \delta_{v_1}} \leq \inf_{s \leq \min\{t, T^*\}} f(s),
  \end{equation*}
  where $f(s) \defby \frac{C_1}{s^{\theta_1}} \norm{v_0 - v_1}_{\contisp^{-\alpha_0}} + C_2 t^{\theta_2}$.
  If we set
  \begin{equation*}
    s_0 \defby \left( \frac{C_1 \theta_1 \norm{v_0 - v_1}_{\contisp^{-\alpha}}}{C_2 \theta_2} \right)^{\frac{1}{\theta_1 + \theta_2}},
  \end{equation*}
  then we have $\inf_{s > 0} f(s) = f(s_0)$.
  If $s_0 \leq \min\{t, T^*\}$, we have
  \begin{equation*}
    \inf_{s \leq \min\{t, T^*\}} f(s) = f(s_0) = C \norm{v_0 - v_1}_{\contisp^{-\alpha_0}}^{\frac{\theta_2}{\theta_1 + \theta_2}}.
  \end{equation*}
  Otherwise, we have
  \begin{align*}
    \inf_{s \leq \min\{t, T^*\}} f(s)
     &= f(\min\{t, T^*\})\\ &= \frac{C_1}{(\min\{t, T^*\})^{\theta_1}} \norm{v_0-v_1}_{\contisp^{-\alpha_0}} + C_2 (\min\{t, T^*\})^{\theta_2} \\
    &\lesssim \Big( \frac{1}{t^{\theta_1}} + \frac{1}{(T^*)^{\theta_1}} \Big) \norm{v_0 - v_1}_{\contisp^{-\alpha_0}}
    + \norm{v_0 - v_1}_{\contisp^{-\alpha_0}}^{\frac{\theta_2}{\theta_1 + \theta_2}} \\
    &\leq  \Big\{ \frac{1}{t^{\theta_1}} + (1+R)^{\kappa \theta_1} \Big\} (2R)^{\frac{\theta_1}{\theta_1 + \theta_2}}
    \norm{v_0 - v_1}_{\contisp^{-\alpha_0}}^{\frac{\theta_2}{\theta_1 + \theta_2}}
    + \norm{v_0 - v_1}_{\contisp^{-\alpha_0}}^{\frac{\theta_2}{\theta_1 + \theta_2}}
  \end{align*}
  In the last inequality, we used the explicit representation of $T^*$ given in Lemma \ref{lem:exstence_of_T_star}.
\end{proof}
\begin{proof}[Proof of Theorem \ref{thm:strong_Feller}]
  Proposition \ref{prop:markov_process} shows that $\{P_t\}_{t \geq 0}$ is a Markov semigroup.
  Let $\Phi$ be a bounded measurable function on $\contisp^{-\alpha_0}$ and $t \in (0,\infty)$.
  Then, we have
  \begin{equation*}
    \abs{P_t \Phi(v_0) - P_t \Phi(v_1)} = \abs{ \inp{\Phi}{P_t^* \delta_{v_0} - P_t^* \delta_{v_1}} }
    \lesssim_{R, t} \norm{\Phi}_{L^{\infty}} \norm{v_0 - v_1}_{\contisp^{-\alpha_0}}^{\theta}
  \end{equation*}
  if $\norm{v_0}_{\contisp^{-\alpha_0}}, \norm{v_1}_{\contisp^{-\alpha_0}} \leq R$.
  This implies the continuity of the map $v \mapsto P_t \Phi(v)$.
\end{proof}

\begin{remark}
  As in \cite[Section 6]{TW18}, it is possible to prove exponential ergodicity of the SCGL.
  See \cite[Subsection 3.3]{Mat20}
\end{remark}

%% file: appendix_besov.tex
\section{Besov spaces}\label{sec:besov}
Here we summarize basic facts about Besov spaces on $\torus = \R^2/\Z^2$.
We refer the reader to \cite{Bah11}, \cite{MW2dim} and \cite{Saw18} for more details.
We set $A \defby (i+\mu) \Delta - 1$ for $\mu > 0$.

\subsection{Some estimates in Besov spaces}
\begin{lemma}\label{lem:embedding_between_besov_spaces}
  We have the following embeddings.
  \begin{enumerate}[(i)]
    \item For $\alpha_1 \leq \alpha_2$, $\B^{\alpha_2}_{p, q} \hookrightarrow \B^{\alpha_1}_{p,q}$.
    \item For $q_1 \leq q_2$, $\B^{\alpha}_{p, q_1} \hookrightarrow \B^{\alpha}_{p, q_2}$.
    \item For $p_1 \leq p_2$, $\B^{\alpha}_{p_2, q} \hookrightarrow \B^{\alpha}_{p_1, q}$.
    \item For $\epsilon > 0$ and $q, q' \in [1, \infty]$,
    $\B^{\alpha}_{p, q} \hookrightarrow \B^{\alpha - \epsilon}_{p, q'}$.
    \item $\B^{0}_{p, 1} \hookrightarrow L^p(\torus)$.
    \item $L^p(\torus) \hookrightarrow \B_{p, \infty}^{0}$.
  \end{enumerate}
\end{lemma}
\begin{proof}
  The proof easily follows from the definition of Besov spaces.
\end{proof}
\begin{proposition}\label{prop:besov_embedding}\leavevmode
  \begin{enumerate}[(i)]
  \item Let $\alpha \in \R$, $p, r \in [1, \infty]$ with $p \geq r$, $\beta \defby \alpha + 2(\frac{1}{r} - \frac{1}{p})$.
  Then, there exists a constant $C \in (0, \infty)$ such that for $q \in [1, \infty]$, we have
  \begin{equation*}
    \norm{f}_{\B^{\alpha}_{p, q}} \leq C \norm{f}_{\B^{\beta}_{r, q}}.
  \end{equation*}
  \item Let $\alpha \in \R$, $k \in \N^2$, $p, q \in [1, \infty]$.
  Then, there exists a constant $C \in (0, \infty)$ such that
  \begin{equation*}
    \norm{\partial^k f}_{\B_{p,q}^{\alpha - \abs{k}}} \leq C \norm{ f}_{\B_{p,q}^{\alpha}}.
  \end{equation*}
  \end{enumerate}
\end{proposition}
\begin{proof}
  See \cite[\largeprop 2, 3]{MW2dim}.
\end{proof}
\begin{proposition}\label{prop:interpolation}
  Let $\alpha_0, \alpha_1 \in \R$, $p_0, q_0, p_1, q_1 \in [1, \infty]$, $\nu \in [0, 1]$,
  $\alpha \defby (1-\nu)\alpha_0 + \nu\alpha_1$ and $p, q \in [1, \infty]$ such that
  \begin{equation*}
    \frac{1}{p} = \frac{1-\nu}{p_0} + \frac{\nu}{p_1},
    \quad \frac{1}{q} = \frac{1-\nu}{q_0} + \frac{\nu}{q_1}.
  \end{equation*}
  Then, we have
  \begin{equation*}
    \norm{f}_{\B^{\alpha, M}_{p, q}} \leq \norm{f}_{\B^{\alpha_0, M}_{p_0, q_0}}^{1-\nu} \norm{f}_{\B^{\alpha_1, M}_{p_1, q_1}}^{\nu}.
  \end{equation*}
\end{proposition}
\begin{proof}
  See \cite[\largeprop 4]{MW2dim}.
\end{proof}
\begin{proposition}\label{prop:smoothing_in_besov}
  Let $\alpha \geq \beta$ and $p, q \in [1, \infty]$. Then, there exists a constant $C \in (0, \infty)$
  such that for $t > 0$,
  \begin{equation*}
    \norm{e^{tA} f}_{\B^{\alpha}_{p,q}} \leq C t^{\frac{\beta - \alpha}{2}} \norm{f}_{\B^{\beta}_{p, q}}.
  \end{equation*}
\end{proposition}
\begin{proof}
  See \cite[\largeprop 5]{MW2dim}.
\end{proof}
\begin{proposition}\label{prop:time_regularity_of_the_heat_flow_in_besov}
  Let $0 \leq \beta - \alpha \leq 2$ and $p, q \in [1, \infty]$.
  Then, there exists a constant $C \in (0, \infty)$ such that for $t \geq 0$,
  \begin{equation*}
    \norm{(1-e^{tA})f}_{\B^{\alpha}_{p, q}} \leq
    C t^{\frac{\beta - \alpha}{2}} \norm{f}_{\B^{\beta}_{p,q}}.
  \end{equation*}
\end{proposition}
\begin{proof}
  See \cite[\largeprop 6]{MW2dim}.
\end{proof}
\begin{lemma}\label{lem:series_criterion}
  Let $\alpha \in \R$, $p, q \in [1, \infty]$ and $\mathcal{C}$ be an annulus in $\R^2$.
  Then, there exists a constant $C \in (0, \infty)$ such that for a sequence $(f_k)$ of $L^p(\torus)$ functions satisfying
  \begin{equation*}
    \supp f_k \subset 2^k \mathcal{C} \quad \mbox{and} \quad
    (2^{\alpha k} \norm{f_k}_{L^p(\torus)})_k \in l^q,
  \end{equation*}
  we have $f = \sum_{k=0}^{\infty} f_k \in \B^{\alpha}_{p, q}$ and
  \begin{equation*}
    \norm{f}_{\B^{\alpha}_{p, q}} \leq C \norm{(2^{\alpha k} \norm{f_k}_{L^p(\torus)})_k}_{l^q}.
  \end{equation*}
\end{lemma}
\begin{proof}
  See \cite[\largelem 6]{MW2dim}.
\end{proof}

\subsection{Products}
For $f, g \in C^{\infty}(\torus)$ we define the paraproduct
\begin{equation*}
  f \paral g \defby g \parag f \defby \sum_{j < k-1} \delta_j f \delta_k g,
\end{equation*}
and the resonance term
\begin{equation*}
  f \resonant g \defby \sum_{\abs{j-k} \leq 1} \delta_j f \delta_k g.
\end{equation*}
We have the following Bony decomposition
\begin{equation}
  fg = f \paral g + f \resonant g + f \parag g.
\end{equation}
\begin{proposition}\label{prop:paraproduct_estimates}\leavevmode
  \begin{enumerate}[(i)]
    \item Let $\alpha, \alpha_1, \alpha_2 \in \R$ and $p, p_1, p_2, q \in [1, \infty]$ such that
    \begin{equation*}
      \alpha_1 \neq 0, \quad \alpha = \min\{\alpha_1, 0\}, \quad \frac{1}{p} = \frac{1}{p_1} + \frac{1}{p_2}.
    \end{equation*}
    Then, there exists a constant $C \in (0, \infty)$ such that
    \begin{equation*}
    \norm{f \paral g}_{\B^{\alpha}_{p, q}} \leq C \norm{f}_{\B^{\alpha_1}_{p_1, \infty}} \norm{g}_{\B^{\alpha_2}_{p_2, q}}.
    \end{equation*}
    \item Let $\alpha_1, \alpha_2 \in \R$ such that $\alpha \defby \alpha_1 + \alpha_2 > 0$ and $p, p_1, p_2, q$ as above.
    Then, there exists a constant $C \in (0, \infty)$ such that
    \begin{equation*}
      \norm{f \resonant g}_{\B^{\alpha}_{p, q}} \leq C \norm{f}_{\B^{\alpha_1}_{p_1, \infty}} \norm{g}_{\B^{\alpha_2}_{p_2, q}}.
    \end{equation*}
  \end{enumerate}
\end{proposition}
\begin{proof}
  See \cite[\largethm 3.1]{MW2dim}.
\end{proof}
\begin{corollary}\label{cor:multiplicative_inequality}\leavevmode
  \begin{enumerate}[(i)]
    \item Let $\alpha > 0$ and $p, p_1, p_2, q \in [1, \infty]$ such that $\frac{1}{p} = \frac{1}{p_1} + \frac{1}{p_2}$.
    Then, there exists a constant $C \in (0, \infty)$ such that
    \begin{equation*}
      \norm{fg}_{\B^{\alpha}_{p, q}} \leq C \norm{f}_{\B^{\alpha}_{p_1, q}} \norm{g}_{\B^{\alpha}_{p_2, q}}.
    \end{equation*}
    In particular, $fg$ is well-defined for $f \in \B^{\alpha}_{p_1, q}$ and $g \in \B^{\alpha}_{p_2, q}$.
    \item Let $\alpha < 0 < \beta $ such that $\alpha + \beta > 0$ and let $p, p_1, p_2, q \in [1, \infty]$ such that
    $\frac{1}{p} = \frac{1}{p_1} + \frac{1}{p_2}$.
    Then, there exists a constant $C \in (0, \infty)$ such that
    \begin{equation*}
    \norm{fg}_{\B^{\alpha}_{p, q}} \leq C \norm{f}_{\B^{\alpha}_{p_1, q}} \norm{g}_{\B^{\beta}_{p_2, q}}.
    \end{equation*}
    In particular, $fg$ is well-defined for $f \in \B^{\alpha}_{p_1, q}$ and $g \in \B^{\beta}_{p_2, q}$.
  \end{enumerate}
\end{corollary}
\begin{proof}
  See \cite[\largecor 1, 2]{MW2dim}.
\end{proof}

\begin{proposition}\label{prop:duality_in_besov}
  Let $\alpha \in (0, 1]$. Then, there exists a constant $C \in (0, \infty)$ such that
  for $p, q \in [1, \infty]$ and $f, g \in C^{\infty}(\torus)$,
  \begin{equation*}
    \abs{ \inp{f}{g}_1 } \leq C \norm{f}_{\B^{\alpha}_{p, q}} \norm{g}_{\B^{-\alpha}_{p', q'}},
  \end{equation*}
  where $p'$ and $q'$ are the conjugate indices of $p$ and $q$ respectively.
\end{proposition}
\begin{proof}
  See \cite[\largeprop 7]{MW2dim}.
\end{proof}

\subsection{Relations among Besov, H\"older and Sobolev spaces}
\begin{proposition}\label{prop:besov_vs_holder}
  Let $\alpha \in (0, 1)$. Then, there exists a constant $C \in (0, \infty)$ such that
  \begin{equation*}
    C^{-1} \norm{f}_{\B^{\alpha}_{\infty, \infty}} \leq
    \norm{f}_{L^{\infty}} + \sup_{x \neq y} \frac{\abs{f(x) - f(y)}}{\abs{x-y}^{\alpha}}
    \leq C \norm{f}_{\B^{\alpha}_{\infty, \infty}}.
  \end{equation*}
\end{proposition}
\begin{proof}
  See \cite[Theorem 2.7]{Saw18}.
\end{proof}
\begin{lemma}\label{lem:holder_fractional_power}
  Let $q \in (0, \infty)$ and $\epsilon \in (0, \min\{1, q\})$. Then, the map
  \begin{equation*}
    F: \B^{1}_{\infty, \infty} \ni f \mapsto F(f) \defby \abs{f}^q \in \B^{\min\{1, q\} - \epsilon}
  \end{equation*}
  is continuous.
\end{lemma}
\begin{proof}
  We set $\contisp^{\alpha} \defby \B^{\alpha}_{\infty, \infty}$.

  \emph{STEP 1. }We first consider the case $q \in [1, \infty)$.
  Set $\alpha \defby 1 - \epsilon$.
  We use the following elementary inequality
  \begin{equation}\label{eq:elementary_ineq_for_holder}
    \abs{a^q - b^q} \leq q \max\{a^{q-1}, b^{q-1}\} \abs{a-b} \quad \mbox{for } a, b \geq 0.
  \end{equation}
  Substituting $a = \abs{f(x)}$ and $b = \abs{f(y)}$, we obtain
  \begin{equation*}
    \frac{\abs{\abs{f(x)}^q - \abs{f(y)}^q}}{\abs{x-y}^{\alpha}}
    \leq q (\max\{\abs{f(x)}, \abs{f(y)}\})^{q-1} \frac{\abs{f(x) - f(y)}}{\abs{x-y}^{\alpha}}.
  \end{equation*}
  Thus, combined with Proposition \ref{prop:besov_vs_holder},
  we obtain $\norm{ \abs{f}^q }_{\contisp^{\alpha}} \lesssim \norm{f}_{\contisp^{\alpha}}^q$.
  In particular, we see $\abs{f}^q \in \contisp^{\alpha}$.

  To prove the continuity of $F$, take $f_1, f_2 \in \contisp^{1}$.
  For $\delta \in (0, 1)$ with $(1 + \delta) \alpha < 1$, we have
  \begin{align*}
    \MoveEqLeft
    \frac{\abs{(\abs{f_1}^q - \abs{f_2}^q)(x) - (\abs{f_1}^q - \abs{f_2}^q)(y)}}{\abs{x-y}^{\alpha}} \\
    &\lesssim \left( \frac{\abs{\abs{f_1(x)}^q - \abs{f_2(x)}^q} + \abs{\abs{f_1(y)}^q - \abs{f_2(y)}^q}}
    {\abs{x-y}^{\alpha(1+\delta)}} \right)^{\frac{1}{1+\delta}}
    \norm{\abs{f_1}^q - \abs{f_2}^q}_{L^{\infty}}^{\frac{\delta}{1+\delta}} \\
    &\lesssim ( \norm{f_1}_{\contisp^1}^q + \norm{f_2}_{\contisp^1}^q)^{\frac{1}{1+\delta}}
    (\norm{f_1}_{L^{\infty}} + \norm{f_2}_{L^{\infty}})^{q-1} \norm{f_1 - f_2}_{L^{\infty}}^{\frac{\delta}{1+\delta}}.
  \end{align*}
  This implies the continuity of $F$.

  \emph{STEP 2. } We next consider the case $q \in (0, 1)$. Instead of \eqref{eq:elementary_ineq_for_holder}, we
  use the inequality $\abs{a^q - b^q} \leq \abs{a-b}^q$. This implies
  \begin{equation*}
    \frac{\abs{\abs{f(x)}^q - \abs{f(y)}^q}}{\abs{x-y}^q} \leq
    \left( \frac{\abs{f(x) - f(y)}}{\abs{x-y}}\right)^q \lesssim \norm{Y}^q_{\contisp^1}.
  \end{equation*}
  The remaining is similar to \emph{STEP 1}.
\end{proof}
\begin{proposition}\label{prop:besov_vs_sobolev}
  Let $\alpha \in (0, 1)$. Then, there exists a constant $C \in (0, \infty)$ such that
  \begin{equation*}
    C^{-1} \norm{f}_{\B^{\alpha}_{1, 1}} \leq \norm{f}_{L^1(\torus)}^{1- \alpha} \norm{\nabla f}_{L^1(\torus)}^{\alpha}
    + \norm{f}_{L^1(\torus)}.
  \end{equation*}
\end{proposition}
\begin{proof}
  See \cite[\largeprop 8]{MW2dim}.
\end{proof}

\subsection{Kolmogorov continuity theorem for Besov spaces}
We set
  $\eta_k^{(1)} (x) \defby \sum_{y \in \Z^2} \eta_k(x+y)$.
Note that $\eta_k^{(1)} \in C^{\infty}(\torus)$.
\begin{lemma}\label{lem:modification_of_distribution}
  Suppose a map $Z: \R \times L^2(\torus) \ni (t, \phi) \mapsto Z(t, \phi) \in L^2(\Omega, \F, \P)$
  is  continuous and linear with respect to $\phi$.
  We further suppose that there exist $p \in (1, \infty)$, $\alpha \in \R$ and $\kappa > \frac{1}{p}$ such that
  for all $T \in (0, \infty)$ we can find $K_T \in L^{\infty}$ such that for $k \geq -1$, $x \in \R^2$ and $s, t \in [-T, T]$,
  \begin{equation}\label{eq:assumption_on_bound_of_Z_eta_k}
    \expect[ \abs{Z(t, \eta_k^{(1)} (x - \cdot))}^p ]
    \leq K_T(x)^p 2^{-k \alpha p},
  \end{equation}
  \begin{equation}\label{eq:assumption_on_difference_of_Z_eta_k}
    \expect[ \abs{Z(t, \eta_k^{(1)}(x - \cdot)) - Z(s, \eta_k^{(1)}(x - \cdot))}^p] \leq
    K_T(x)^p 2^{-k\alpha p} \abs{t-s}^{\kappa p}.
  \end{equation}
  Then, for $\alpha' < \alpha$, there exists a random distribution $\tilde{Z} \in C(\R; \B^{\alpha'}_{p, p})$
  such that for $t \in \R$ and $\phi \in C^{\infty}(\torus)$ we have
  \begin{equation*}
    Z(t, \phi) = \inp{\tilde{Z}(t)}{\phi}_{1} \quad \mbox{almost surely.}
  \end{equation*}
  Moreover, there exists $\epsilon \in (0, 1)$ such that for $T \in (0, \infty)$,
  \begin{equation*}
    \expect [\sup_{-T \leq s < t \leq T} (t-s)^{-\epsilon p}
    \norm{\tilde{Z}(t) - \tilde{Z}(s)}^p_{\B^{\alpha'}_{p, p}}] \lesssim_{T, \alpha, \alpha', p}
    \int_{\torus} K_T(x)^p \, dx .
  \end{equation*}
\end{lemma}
\begin{proof}
  The proof for weighted Besov spaces is given in \cite[\largelem 9]{MW2dim}.
  Although the proof for periodic Besov spaces is essentially the same, we provide a complete proof below
  as this lemma is of great importance.
  Take $\epsilon_k \defby 1/n_k$ for sufficiently large $n_k \in \N$.
  For $k \geq 0$, let $\tilde{\chi}_k$ be a real valued, radial smooth function such that
  $\tilde{\chi}_k \equiv 1$ on the annulus $B(0, 2^k \frac{8}{3}) \setminus B(0, 2^k\frac{3}{4})$
  and $\tilde{\chi}_k \equiv 0$ outside the annulus $B(0, 2^k \frac{16}{3}) \setminus B(0, 2^k \frac{3}{8})$.
  For $k = -1$ we let $\tilde{\chi}_{-1} \equiv 1$ on $B(0, \frac{4}{3})$ and $\tilde{\chi}_{-1} \equiv 0$ outside $B(0, \frac{8}{3})$. We set $\tilde{\eta}_k \defby \F^{-1} \tilde{\chi}_k$.

  Set
  \begin{equation*}
    \hat{Z}_k(t, y) \defby Z(t, \eta_k^{(1)} (y - \cdot)) \qquad (y \in \epsilon_k \Z^2)
  \end{equation*}
  and
  \begin{equation}\label{eq:def_of_Z_k}
    Z_k(t, x) \defby \sum_{y \in \epsilon_k \Z^2} \epsilon_k^2 \hat{Z}_k(t, y) \tilde{\eta}_k (x-y).
  \end{equation}
  Note that $Z_k(t, \cdot)$ is $1$-periodic.
  By \eqref{eq:assumption_on_bound_of_Z_eta_k} we see that with probability $1$
  the sum in \eqref{eq:def_of_Z_k} absolutely converges uniformly for $x$,
  and therefore we have $Z(t, \cdot) \in C^{\infty}(\torus)$.
  Furthermore, for $\phi \in \S$, we have
  \begin{align*}
    \inp{\F(Z_k(t, \cdot))}{\phi} &= \inp{Z_k(t, \cdot)}{\F \phi} \\
    &= \sum_{y \in \epsilon_k \Z^2} \epsilon_k^2 \hat{Z}_k(t, y) \inp{\tilde{\eta}_k(\cdot - y)}{\F \phi}.
  \end{align*}
  Thus,
  \begin{align*}
    \supp \{\F Z_k(t, \cdot)\} &\subset B(0, 2^{k+4} 3^{-1}) \setminus B(0, 2^{k+3} 3^{-1}),\\
    \supp \{\F Z_{-1}(t, \cdot)\} &\subset B(0, 8\cdot3^{-1}).
  \end{align*}
  Now we set for $\phi \in C^{\infty}(\torus)$
  \begin{equation*}
    \tilde{Z}(t, \phi) \defby \inp{\tilde{Z}(t)}{\phi} \defby \sum_{k \geq -1} \inp{Z_k(t)}{\phi},
  \end{equation*}
  which is well-defined at least for $\phi$ with $\F \phi \in C^{\infty}_c(\R^2)$.
  We compute
  \begin{align*}
    \MoveEqLeft
    \sum_{y \in \epsilon_k \Z^2} \epsilon_k^2 \eta_k^{(1)}(y-z) \int_{\torus} \tilde{\eta}_k(x-y)\phi(x) \, dx \\
    &= \sum_{\substack{y \in \epsilon_k \Z^2, \\ w \in 1 \Z^2}} \epsilon_k^2 \eta_k(w+y-z)
    \int_{\torus} \tilde{\eta}_k(x-y) \phi(x) \, dx \\
    &= \sum_{\substack{y \in \epsilon_k \Z^2, \\ w \in  \Z^2}} \epsilon_k^2 \eta(y-z)
     \int_{\torus} \tilde{\eta}_k(x-y+w) \phi(x) \, dx \\
    &= \sum_{y \in \epsilon_k \Z^2} \epsilon_k^2 \eta_k(y-z) \int_{\R^2} \tilde{\eta}_k(x-y) \phi(x) \, dx \\
    &=: A.
  \end{align*}
  Take $\phi_{\lambda} \in C^{\infty}_c(\R^2)$ such that $\phi_{\lambda} \equiv \phi$ on $B(0, \lambda)$,
  $\norm{\phi_{\lambda}}_{L^{\infty}} \leq \norm{\phi}_{L^{\infty}}$ and  $\phi_{\lambda} \to \phi$ as $\lambda \to \infty$
  pointwise.
  Then by Lebesgue's dominated convergence theorem,
  \begin{equation*}
    A = \lim_{\lambda \to \infty} \sum_{y \in \epsilon_k \Z^2} \epsilon_k^2 \eta_k(y-z)
    \int_{\R^2} \tilde{\eta}_k(x-y) \phi_{\lambda}(x) \, dx.
  \end{equation*}

  We claim the identity
  \begin{equation}\label{eq:identity_of_chi_and_phi_lambda}
    \sum_{y \in \epsilon_k \Z^2} \epsilon_k^2 \eta_k(y-z) \int_{\R^2} \tilde{\eta}_k(x-y) \phi_{\lambda}(x) \, dx
    = \F^{-1}(\chi_k) * \phi_{\lambda}(z).
  \end{equation}
  Indeed, by taking Fourier transforms, it suffices to show
  \begin{equation*}
    \sum_{y \in \epsilon_k \Z^2} \epsilon_k^2 e^{-2 \pi i \xi \cdot y} \chi_k(\xi) \tilde{\eta}_k * \phi_{\lambda}(y)
    = \chi_k(\xi) \F\phi_{\lambda}(\xi) \quad \mbox{in } \,\,L^2(\R^2, d\xi).
  \end{equation*}
  As $\chi_k$ has compact support, it suffices to show for fixed $\xi \in \supp \chi_k$,
  \begin{equation}\label{eq:identity_of_eta_and_phi}
  \sum_{y \in \epsilon_k \Z^2} \epsilon_k^2 \tilde{\eta}_k * \phi_{\lambda}(y) e^{-2 \pi i \xi \cdot y} = \F \phi_{\lambda}(\xi).
  \end{equation}
  Recalling the Poisson summation formula, we calculate
  \begin{align*}
    \sum_{y \in \epsilon_k \Z^2} \epsilon_k^2 \tilde{\eta}_k * \phi_{\lambda}(y) e^{-2\pi i \xi \cdot y}
    &= \sum_{y \in \Z^2} \epsilon_k^2 \tilde{\eta}_k * \phi_{\lambda}(\epsilon_k y) e^{-2\pi i \epsilon_k \xi \cdot y} \\
    &= \sum_{y \in \Z^2} \epsilon_k^2 \F^{-1}[\tilde{\eta}_k * \phi_{\lambda} (\epsilon_k \cdot)] (-\epsilon_k \xi + y) \\
    &= \sum_{y \in \Z^2} \tilde{\chi}_k(\xi - \frac{y}{\epsilon_k}) \F \phi_{\lambda}(\xi - \frac{y}{\epsilon_k}).
  \end{align*}
  For sufficiently small $\epsilon_k$, $\tilde{\chi}(\xi - y/\epsilon_k) = 0$ whenever $y \in \Z^2 \setminus \{0\}$.
  Hence we proved \eqref{eq:identity_of_eta_and_phi}.

  Using \eqref{eq:identity_of_chi_and_phi_lambda},
  \begin{equation*}
    A = \lim_{\lambda \to \infty} \F^{-1}(\chi_k) * \phi_{\lambda}(z) = \F^{-1}(\chi_k) * \phi(z).
  \end{equation*}
  Therefore, by continuity of $Z(t, \cdot)$,
  \begin{equation*}
    \inp{\tilde{Z}(t)}{\phi} = \sum_{k \geq -1} Z(t, \F^{-1}(\chi_k)*\phi) = Z(t, \phi).
  \end{equation*}

  We still need to show that $\tilde{Z}$ can be realized as a distribution.
  The assumption \eqref{eq:assumption_on_bound_of_Z_eta_k} leads to
  \begin{align*}
    \expect [\norm{Z_k(t, \cdot)}_{L^p(\torus)}^p] &= \int_{\torus} \expect [\abs{Z_k(t, \eta_k(x-\cdot))}^p] \, dx \\
    &\leq 2^{-k\alpha p} \int_{\torus} K_T(x)^p \, dx .
  \end{align*}
  Thus for $\alpha' < \alpha$,
  \begin{align*}
    \expect [\sum_{k \geq -1} 2^{k \alpha' p} \norm{Z_k(t, \cdot)}_{L^p(\torus)}^p] &\leq
    \sum_{k \geq -1} 2^{-k(\alpha - \alpha') p} \int_{\torus}  K_T(x)^p \, dx \\
    &= C(\alpha, \alpha', p) \int_{\torus} K_T(x)^p \, dx < \infty.
  \end{align*}
  Therefore, we have $\sum 2^{k \alpha' p} \norm{Z_k(t, \cdot)}^p_{L^p(\torus)} < \infty$ almost surely and
  by Lemma \ref{lem:series_criterion} $\tilde{Z}(t) \defby \sum_{k \geq -1} Z_k(t, \cdot)$ is
  $\B^{\alpha'}_{p, p}$-valued.

  Similarly, by \eqref{eq:assumption_on_difference_of_Z_eta_k},
  \begin{equation*}
    \expect[ \norm{\tilde{Z}(t) - \tilde{Z}(s)}_{\B^{\alpha'}_{p, p} }^p ]
    \lesssim_{\alpha, \alpha', p, T} \abs{t-s}^{\kappa p} \int_{\torus} K_T(x)^p \, dx.
  \end{equation*}
  It remains to apply the usual Kolmogorov criterion.
\end{proof}

%% file: appendix_complex_ito_integral.tex
\section{Complex multiple It\^o-Wiener Integrals}\label{sec:complex_ito_integrals}
We recall basic properties of complex multiple It\^o-Wiener integrals.
See \cite{Ito52} for more details.

A complex random variable $Z$ is called isotropic complex normal
if $\Re Z$ and $\Im Z$ are independent, identically distributed and
$(\Re Z, \Im Z)$ is jointly normal with mean $0$.
A family of complex random variables $\{Z_{\lambda}\}$ is called jointly isotropic complex normal
if $\sum_{i=1}^n c_i Z_{\lambda_i}$ is isotropic complex normal for any $n$ and $c_1, \ldots, c_n \in \C$.
The distribution of jointly isotropic complex normal system $\{Z_{\lambda}\}$ is
uniquely determined by the positive-definite matrix
$\{\expect[ Z_{\lambda} \conj{Z_{\mu}} ] \}_{\lambda, \mu}$(\cite[Theorem 2.3]{Ito52}).

Let $(E, \mathcal{E}, m)$ be a $\sigma$-finite, atomless measure space and
$\mathcal{E}^*$ be the set of all $A \in \mathcal{E}$ such that $m(A) < \infty$.
Then there exists a jointly isotropic complex normal system
$\set{M(A) \given A \in \mathcal{E}^*}$ such that
\begin{equation*}
  \expect[ M(A) \conj{M(B)} ] = m(A \cap B),
\end{equation*}
see \cite[Theorem 3.1]{Ito52}.

Now we define the complex multiple It\^o-Wiener integral of $f \in L^2_{k,l}
\defby L^2 (E^k \times E^l)$ for $k, l \in \N$. First assume that
\begin{equation}\label{eq:simple_f_for_multiple_integral}
  f = \sum_{i_1, \ldots, i_{k+l}=1}^n a_{i_1 \ldots j_{k+l}}
  \indic_{E_{i_1} \times \cdots \times E_{i_{k+l}}},
\end{equation}
where $E_1, \ldots, E_n$ are disjoint sets of $\mathcal{E}^*$ and
$ a_{i_1 \ldots i_{k+l}} $ is a complex number which equals $0$ unless
$i_1, \ldots, i_{k+l}$ are all different.
Then we define
\begin{equation*}
  \mathcal{J}_{k,l}(f) \defby
  \sum_{i_1, \ldots, i_{k+l} =1}^n a_{i_1\ldots i_{k+l}} M(E_{i_1}) \cdots M(E_{i_k})
  \conj{M(E_{i_{k+1}})} \cdots \conj{M(E_{i_{k+l}})}.
\end{equation*}
We have
\begin{align*}
  &\expect\left[ \abs*{ \sum a_{i_1 \ldots i_{k+l}} M(E_{i_1}) \cdots M(E_{i_k}) \conj{M(E_{i_{k+1}})} \cdots \conj{M(E_{i_{k+l}})}}^2 \right] \\
  &= \sum a_{i_1 \ldots i_{k+l}} \conj{a_{j_1 \ldots j_{k+l}}} \expect[ M(E_{i_1}) \cdots \conj{M(E_{k+l})}
  \,\,\conj{M(E_{j_1})} \cdots M(E_{j_{k+l}}) ] \\
  &= \sum_{\substack{\set{i_1, \ldots, i_k} = \set{j_1, \ldots, j_k} \\ \set{i_{k+1}, \ldots, i_{k+l}} = \set{j_{k+1}, \ldots, j_{k+l}}}}
  a_{i_1 \ldots a_{i_{k+l}}} \conj{a_{j_1 \ldots j_{k+l}}} \expect[\abs{M(E_{i_1})}^2] \cdots \expect[ \abs{M(E_{i_{k+l}})}^2] \\
  &= \sum_{\substack{\set{i_1, \ldots, i_k} = \set{j_1, \ldots, j_k} \\ \set{i_{k+1}, \ldots, i_{k+l}} = \set{j_{k+1}, \ldots, j_{k+l}}}}
  \hspace{-1em} a_{i_1 \ldots i_{k+l}} [m(E_{i_1}) \cdots m(E_{i_{k+l}})]^{\frac{1}{2}} \conj{a_{j_1 \ldots j_{k+l}}}
  [m(E_{j_1}) \cdots m(E_{j_{k+l}})]^{\frac{1}{2}} \\
  &\leq k! l! \norm{f}_{L^2_{k,l}}^2,
\end{align*}
where in the last line we used Cauchy-Schwarz inequality.
For general $f \in L^2_{k, l}$, we can find a sequence of $\{f_n\}$ of the form \eqref{eq:simple_f_for_multiple_integral}
such that $f_n \to f$ in $L^2_{k,l}$.
We define $\mathcal{J}_{k, l}(f) \defby \lim_{n \to \infty} \mathcal{J}_{k,l}(f_n)$, where the limit is in $L^2(\P)$.
Well-definedness is guaranteed by the above inequality.

As with real multiple It\^o-Wiener integrals, we have the following $L^p$-estimates.
\begin{proposition}\label{prop:Nelson_estimate}
  For $f \in L^2_{k,l}$ and $p \in [2, \infty)$, we have
  \begin{equation*}
    \expect[ \abs{\mathcal{J}_{k,l}(f)}^p]^{\frac{1}{p}} \lesssim_{k,l} \norm{f}_{L^2_{k,l}}.
  \end{equation*}
\end{proposition}

Finally, we review the product formula and complex Hermite polynomials.
Let $k_1, l_1, k_2, l_2, r_1, r_2 \in \N$.
Let $G_1$ be the set of sets $\{(i_1, j_1), \ldots, (i_{r_1}, j_{r_1}) \}$ such that
$i_1, \ldots, i_{r_1} \in \{ 1, \ldots, k_1\}$ are distinct and
$j_1, \ldots, j_{r_1} \in \{ l_1 + 1, \ldots, l_1 + l_2 \}$ are distinct.
Let $G_1$ be the set of sets $\{(i_1, j_1), \ldots, (i_{r_2}, j_{r_2}) \}$ such that
$i_1, \ldots, i_{r_1} \in \{ k_1+1, \ldots, k_1+k_2\}$ are distinct and
$j_1, \ldots, j_{r_1} \in \{ 1, \ldots, l_1 \}$ are distinct.
Set $\mathcal{G}(k_1, l_1; k_2, l_2; r_1, r_2) \defby G_1 \cup G_2$.
Note that $\mathcal{G}(k_1, l_1; k_2, l_2; r_1, r_2) = \emptyset$ unless
$r_1 \leq \min\{ k_1, l_2 \}$ and $r_2 \leq \min \{ k_2, l_1 \}$.
For $f \in L^2_{k_1, l_1}$, $g \in L^2_{k_2, l_2}$ and $\gamma \in \mathcal{G}(k_1, l_1; k_2, l_2; r_1,$ $r_2)$,
we define $f \otimes_{\gamma} g$ by
\begin{multline*}
  f \otimes_{\gamma} g (t_1, \ldots, t_{k_1+k_2}, s_1, \ldots, s_{l_1+l_2}
  \setminus \{ (t_i, s_j)_{(i,j) \in \gamma} \}) \\
  \defby \int_{E^{r_1 + r_2}} h(\{(t_i, s_j)\}_{(i,j)\in \gamma})
  \prod_{(i, j) \in \gamma} dm(t_i, s_j),
\end{multline*}
where $h: E^{r_1 + r_2} \to \C$ is defined by
\begin{multline*}
  h(\{(t_i, s_j)\}_{(i,j)\in \gamma}) \defby
  f(t_1, \ldots, t_{k_1}, s_1, \ldots, s_{l_1}) \\
  \times g(t_{k_1+1}, \ldots, t_{k_1+k_2}, s_{l_1+1}, \ldots, s_{l_1 + l_2})
  \vert_{t_i = s_j \mbox{\small{ if }} (i,j) \in \gamma}.
\end{multline*}

\begin{proposition}\label{prop:product_formula}
  For $f \in L^2_{k_1, l_1}$ and $g \in L^2_{k_2, l_2}$, we have
  \begin{equation*}
    \mathcal{J}_{k_1, l_1}(f) \mathcal{J}_{k_2, l_2}(g)
    = \sum_{r_1, r_2 \in \N} \sum_{\gamma \in \mathcal{G}(k_1, l_1; k_2, l_2; r_1, r_2)}
    \mathcal{J}_{k_1+k_2-(r_1+r_2), l_1+l_2-(r_1+r_2)}(f \otimes_{\gamma} g).
  \end{equation*}
\end{proposition}
\begin{proof}
  When $(k_2, l_2) \in \{ (1, 0), (0, 1) \}$, the proposition is proved in \cite[\largethm 9]{Ito52}.
  For general $(k_2, l_2)$, we can prove the proposition by induction.
\end{proof}
\begin{definition}\label{def:complex_hermite_polynomials}
  We define complex Hermite polynomials $\{H_{k,l}(z, c) \}$ by the identity
  \begin{equation*}
    \exp(u \conj{z} + v z - c uv) = \sum_{m, n \in \Z} \frac{v^k u^l}{k! l!} H_{k, l}(z, c), \quad  z, c \in \C.
  \end{equation*}
\end{definition}
\begin{remark}\label{remark:explicit_form_of_hermite}
  We have the following explicit representation
  \begin{equation*}
    H_{k,l}(z, c) = \sum_{m=0}^{\min\{k,l\}} m! \binom{k}{m} \binom{l}{m} (-c)^m z^{k-m} \conj{z}^{l-m}.
  \end{equation*}
  In particular,
  \begin{align*}
    &H_{0, 0}(z, c) = 1, \\
    &H_{1, 0}(z, c) = z,  \quad H_{0, 1}(z, c) = \conj{z}, \\
    &H_{2, 0}(z, c) = z^2, \quad H_{1, 1}(z, c) = \abs{z}^2 - c, \quad H_{0, 2}(z, c) = \conj{z}^2, \\
    &H_{2, 1}(z, c) = \abs{z}^2 z - 2c z.
  \end{align*}
\end{remark}
\begin{proposition}\label{prop:properties_of_hermite_polynomials}
  We have the following identities.
  \begin{enumerate}[(i)]
    \item
      $H_{k+1, l}(z, c) = z H_{k,l}(z, c) - c l H_{k, l-1}(z, c)$.
    \item $H_{k, l+1}(z, c) = \conj{z} H_{k, l}(z, c) - c k H_{k-1, l}(z, c)$.
    \item
    $
      H_{k, l}(x + y, c) = \sum_{i \leq k, j \leq l} \binom{k}{i} \binom{l}{j} x^i \conj{x}^j
      H_{k-i, l-j}(y, c).
      $
  \end{enumerate}
\end{proposition}
\begin{proof}
  Since
  \begin{equation*}
    \frac{\partial}{\partial v} \exp(u \conj{z} + v z - c u v)
    = (z - c u) \exp(u \conj{z} + v z - c u v) ,
  \end{equation*}
  we have
  \begin{equation*}
    \sum \frac{v^k u^l}{k! l!} H_{k+1, l}(z, c)
    = \sum \frac{v^k u^l}{k! l!} z H_{k, l}(z, c) - c \sum \frac{v^k u^l}{k! l!} l H_{k, l-1}(z, c).
  \end{equation*}
  This proves (i). The proof of (ii) is similar.

  For (iii), we observe that
  \begin{align*}
    \exp(u (\conj{x+y}) + v (x+y) - c u v) &=
    \left( \sum \frac{v^k u^l}{k! l!} H_{k,l}(y,c) \right) \left( \sum \frac{v^k u^l}{k! l!} x^k \conj{x}^l \right) \\
    &= \sum \frac{v^{k_1 + k_2} u^{l_1 + l_2}}{k_1! k_2! l_1! l_2!} H_{k_1, l_1}(y, c) x^{k_2} \conj{x}^{l_2}.
    \qedhere
  \end{align*}
\end{proof}
\begin{corollary}\label{cor:multiple_integral_and_hermite_polynomial}
  Let $f \in L^2(E)$, $Z \defby \mathcal{J}_{1,0}(f)$ and $c \defby \norm{f}_{L^2(E)}^2$.
  Then, we have
  \begin{equation*}
    \mathcal{J}_{k, l}(f^{\otimes(k+l)}) = H_{k, l}(Z, c),
  \end{equation*}
  where $f^{\otimes n}(t_1, \ldots, t_n) \defby f(t_1) \cdots f(t_n)$.
\end{corollary}
\begin{proof}
  By Proposition \ref{prop:product_formula} and \ref{prop:properties_of_hermite_polynomials},
  $\{\mathcal{J}_{k,l}(f^{\otimes (k+l)})\}_{k,l}$ and $\{H_{k,l}(Z, c)\}_{k,l}$ satisfy the same
  recursive relation.
\end{proof}